\documentclass{amsart}

\usepackage{amssymb}
\usepackage{amsmath}
\usepackage{amsthm}
\usepackage{mathtools}
\usepackage{mltex}
\usepackage{parskip}
\usepackage{enumerate}
\usepackage{hyperref}
\usepackage{tcolorbox}
\usepackage{graphicx}
\graphicspath{ {images/} }
\usepackage{amsmath,microtype}
\usepackage{multicol}
\usepackage{todonotes}
\usepackage{wrapfig}
\usepackage{tikz-cd}
\usepackage[english]{babel}
\usepackage [autostyle, english = american]{csquotes}
\MakeOuterQuote{"}
\usetikzlibrary{backgrounds}
\usepackage{dynkin-diagrams}
\usepackage{booktabs}
\usepackage{blkarray}
\usetikzlibrary{babel}
\usepackage{upgreek}
\usepackage{soul}

\DeclareMathOperator{\Ad}{Ad}
\DeclareMathOperator{\ad}{ad}

\DeclareMathOperator{\End}{End}
\DeclareMathOperator{\Aut}{Aut}

\DeclareMathOperator{\SO}{SO}
\DeclareMathOperator{\SU}{SU}
\DeclareMathOperator{\G}{G}
\DeclareMathOperator{\Sp}{Sp}
\DeclareMathOperator{\U}{U}

\DeclareMathOperator{\tr}{tr}

\DeclareMathOperator{\Spin}{Spin}

\DeclareMathOperator{\Id}{Id}
\DeclareMathOperator{\inv}{inv}
\DeclareMathOperator{\Span}{span}
\DeclareMathOperator{\AS}{AS}
\DeclareMathOperator{\vectorial}{vec}
\DeclareMathOperator{\totallyskew}{skew}
\DeclareMathOperator{\CT}{CT}

\theoremstyle{plain}
\newtheorem{theorem}{Theorem}[section]
\newtheorem*{theorem*}{Theorem}
\newtheorem{lemma}[theorem]{Lemma}
\newtheorem{proposition}[theorem]{Proposition}
\newtheorem{corollary}[theorem]{Corollary}

\theoremstyle{definition}

\newtheorem{remark}[theorem]{Remark}
\newtheorem{example}[theorem]{Example}

\newcommand{\R}{\mathbb{R}}
\newcommand{\C}{\mathbb{C}}

\newcommand{\tad}{{\text{3-}(\alpha,\delta)\text{-Sasaki}}}
\renewcommand{\H}{\mathbb{H}}

\renewcommand{\:}{\colon}
\newcommand{\vect}{\mathfrak{X}}

\newcommand{\mf}{\mathfrak{m}_F}
\newcommand{\mb}{\mathfrak{m}_B}

\newcommand{\iso}{{\;\stackrel{_\sim}{\to}\;}}

\usepackage[margin=3cm]{geometry}

\newtcolorbox{mybox}
{colframe = purple!25,
  colback  = purple!10,
  coltitle = purple!20!black,  
  title    = Idea}
  
\newtcolorbox{mybox2}
{colframe = red!75,
  colback  = red!10,
  coltitle = red!20!white,  
  title    = Problem}

\newtcolorbox{mybox3}
{colframe = blue!75,
  colback  = blue!10,
  coltitle = blue!20!white,  
  title    = To be completed/added later}  
  
\begin{document}

\title{Invariant Spinors on Homogeneous Spheres}
\author{Ilka Agricola, Jordan Hofmann, and Marie-Am\'elie Lawn}
\begin{abstract}
	
	Using the characterization of the spin representation in terms of exterior forms, we give a complete classification of invariant spinors on the nine homogeneous realizations of the sphere $S^n$. In each of the cases we determine the dimension of the space of such spinors, give their explicit description, and study the underlying related geometric structures depending on the metric. We recover some known results in the Sasaki and 3-Sasaki cases and find several new examples: in particular we give the first known examples of generalized Killing spinors with four distinct eigenvalues. 
\end{abstract}

\maketitle


\noindent


\tableofcontents
\section{Introduction}

Special spinor fields on Riemannian manifolds are linked in many surprising ways to specific geometric structures. In particular, spinors parallelized by a connection are of remarkable interest. One of the first results in this direction is the work of Wang in \cite{Wang}, where he gives a classification of manifolds admitting Riemannian parallel spinors (i.\,e.\,parallel for the Levi-Civita connection) in terms of their possible holonomy groups. The fundamental idea behind this result is that the space of parallel spinors corresponds to the subspace of the spinor representation on which the lift of the holonomy group to the spin group acts trivially. In particular, manifolds admitting such spinors are necessarily Ricci-flat, so that for example the only symmetric spaces with parallel spinors are the flat ones. 

This severe restriction on the geometry of the space leads one to consider other connections on manifolds, e.g. those which are metric with non-vanishing torsion. This type of connection is of particular interest in the context of $G$-structures, where they arise naturally as the characteristic connections of non-integrable geometries. Famous examples are given by half-flat $\mathrm{SU}(3)$ manifolds and cocalibrated $\G_2$-manifolds, which in turn generalize the classical examples of Riemannian manifolds admitting Killing spinors \cite{Friedrich80,BFGK}. Unfortunately, because many classical results available in Riemannian holonomy theory fail to hold for $G$-connections with non-vanishing torsion, the theory of non-integrable geometries is not as well understood, and  classifications like the one of Wang are out of reach in this case. However various results in the past twenty years characterizing different geometries with torsion in terms of special spinors  (see e.g. \cite{dim67, hypo}) have proven the power of the spin geometry approach even in this context. An example of this are generalized Killing spinors, i.e. spinors $\varphi$ which are solutions of the equation $\nabla^{g}_X\varphi=A(X)\cdot\varphi,$ 
where  $\nabla^{g}$ and $\cdot$ are respectively the  spin Levi-Civita connection and the Clifford multiplication on $M$, and $A$ is a symmetric  endomorphism. Such spinors define a $G$-structure on the manifold, where $G$ is the stabilizer of the spinor at some point. Indeed due to the equation it satisfies, $\varphi$ is necessarily parallel for an appropriate connection depending on the endomorphism $A$, which furthermore determines the intrinsic torsion of this $G$-structure; the fact that $A$ is assumed to be symmetric implies that some component of the intrinsic torsion vanishes. This construction is very useful especially in low dimensions. It is well-known for instance that the existence of generalized Killing spinors is in one-to-one correspondence with half-flat $\SU(3)$-structures in dimension 6 and with cocalibrated $\G_2$-structures in dimension 7 \cite{CS02, CS06, dim67}. But, again, despite some progress in this direction \cite{GKSEinstein, GKSspheres} no classification of generalized Killing spinors is available at the moment, even in the simpler case of homogeneous spaces like the sphere. On the other hand, lots of examples of non-integrable geometries are given by (non-symmetric) homogeneous spaces \cite{BFGK, finoalmostcontact, nearly_parallel_g2, kathhabil, SRNI}. It is therefore a sensible question to ask if, in the case of homogeneous spaces, it is possible to classify spinors parallelized by some suitable and natural connection on the manifold. 

In the case of a homogeneous space $M=G/H$ with a reductive action of the group $G$ and a fixed reductive $Ad|_H$-invariant decomposition $\mathfrak{g}=\mathfrak{m}\oplus\mathfrak{h}$, there exists a distinguished, so-called Ambrose-Singer (or canonical) connection parallelizing G-invariant tensor fields on $M$, and in particular its torsion and curvature tensors. If the manifold admits a spin structure which is $G$-invariant, then $G$ acts on the spinor fields, and the parallel spinors of the Ambrose-Singer connection are precisely the $G$-invariant spinors. These are equivalent to  $H$-invariant elements of the spin representation and can be computed using representation theory arguments. However there is to our knowledge no systematic study of such spinors on homogeneous spaces, although they appear naturally in literature as illustrated above.
 
A particularly interesting case are the spheres. Due to their high level of symmetry they can be realized according to exactly nine different groups acting transitively and effectively on them, which were classified by Montgomery and Samelson \cite{MontgomerySamelson43}. In \cite{invariantspinstructures} the authors identify which of these actions leave the unique spin structure of the sphere invariant. A summary of these actions, their isotropy groups, and the existence of a G-invariant spin structure for each case is given in the following table:
\begin{table}[h!] 
\centering
\caption{Homogeneous Spheres and $G$-Invariant Spin Structures}
\begin{tabular}{ |l||l|l|l| }
	\hline
	Lie group & Manifold & Isotropy Subgroup & $G$-inv. spin struct. \cite{invariantspinstructures} \\
	\hline
	$\SO(n+1)$   &    $S^{n}$    &  $\SO(n)$   & No\\
	$\U(n+1)$   &     $S^{2n+1}$    & $\U(n)$  & No\\
	$\SU(n+1)$   &    $S^{2n+1}$    & $\SU(n)$ & Yes\\
	$\Sp(n)$   &    $S^{4n-1}$    & $\Sp(n-1)$ & Yes\\
$\Sp(n)\Sp(1)$   &$S^{4n-1}$    & $\Sp(n-1)\Sp(1)$ & $n$ even\\
	$\Sp(n)\U(1)$   &$S^{4n-1}$    & $\Sp(n-1)\U(1)$ & $n$ even\\
	$\G_2$   &$S^6$    & $\SU(3)$&  Yes\\
	$\Spin(7)$   &$S^7$    & $\G_2$ & Yes\\
	$\Spin(9)$   &$S^{15}$    & $\Spin(7)$&  Yes\\
	\hline
\end{tabular}
\label{Tab:homogeneousspheres}
\end{table}

It is worth noting that in all of the non-lifting cases, there exists a unique connected double covering, up to isomorphism, which acts non effectively on the sphere. In the cases where the spin structure  is not $G$-invariant, we can therefore take this double covering instead of the original effective action and this does preserve the spin structure \cite{invariantspinstructures}. 

Of course, in general, and even if the necessary condition of the existence of a G-invariant spin structure is satisfied, $G$-invariant spinors need not exist. Our goal is to compute the space of these spinors in each of the nine cases, find the equation they satisfy in terms of the Levi-Civita connection, and deduce properties of the geometric structure of the manifold. To that end, we use an explicit realization of the spin representation in terms of exterior forms, similar to \cite{GoodmanWallach}. This approach can be found in \cite{Wang}, though not very explicitly, and it has so far not been widely employed elsewhere in the spin geometry literature, with many authors instead using the traditional viewpoint found in \cite{BFGK, FriedrichBook}. The exterior form viewpoint greatly simplifies calculations, especially in larger dimensions, and allows us to present spinors of interest in a much more readable way--we thus believe that this alternative approach is a computational technique that will prove useful in other applications of spin geometry as well. We would like to emphasize that, in contrast to the realization in terms of exterior forms found in Section 3.1 of \cite{Gau1997}, we make no assumption about the existence of an almost complex structure. We also note that the dimension of the space of invariant spinors may in some cases be deduced from the main proposition in \cite{Wang}.
In each case we give an explicit construction of the invariant spinors and discuss the related geometric structures. We summarize our main results in the following theorem:
\begin{theorem*}
	The dimensions of the spaces of invariant spinors for each realization of the sphere as a homogeneous space are given in Table \ref{Tab:main_results_table}. For the realizations admitting non-trivial invariant spinors, we find:\emph{
	\begin{enumerate}[(1)]
		\item A pair of linearly independent generalized Killing spinors with two eigenvalues on $(S^{2n+1}=\SU(n+1)/\SU(n),g_{a,b})$, and a related invariant $\alpha$-Sasakian structure for $\alpha=\frac{\sqrt{a(n+1)}}{2b\sqrt{n}}$;
		\item A $2n$-dimensional space of invariant spinors on $(S^{4n-1}=\Sp(n)/\Sp(n-1),g_{\vec{a}})$, expressed in terms of the structure tensors of the invariant $3$-Sasakian structure;
		\begin{itemize}
			\item For $n=2$, four linearly independent invariant generalized Killing spinors with four eigenvalues;
		\end{itemize}
	\item A generalized Killing spinor with two eigenvalues on $(S^7=\Sp(2)\Sp(1)/\Sp(1)\Sp(1),g_{a,b})$;
	\item An invariant $\alpha$-Sasakian structure on $(S^{4n-1}=\Sp(n)\U(1)/\Sp(n-1)\U(1), g_{a,b,c})$ for $\alpha=\frac{a}{b\Omega} =\frac{a}{2c\Omega} $, together with a pair of linearly independent invariant spinors not associated to the $\alpha$-Sasakian structure and which, for $n>2$ are not generalized Killing spinors;
		\begin{itemize}
		\item For $n=2$, a pair of linearly independent invariant generalized Killing spinors with three eigenvalues;
	\end{itemize}
\item An invariant Killing spinor (resp. a pair of linearly independent invariant Killing spinors) on the round sphere $S^6=\G_2/\SU(3)$ (resp. the round sphere $S^7=\Spin(7)/\G_2$);
\item An invariant spinor on $(S^{15}=\Spin(9)/\Spin(7),g_{a,b})$ satisfying a differential equation depending on the $3$-form determined by the spinor via the squaring construction.
	\end{enumerate}
}
\end{theorem*}

\begin{table}[h!] 
	\centering
	\caption{Invariant Spinors and Geometric Structures on Homogeneous Spheres}
	\begin{tabular}{ |l||l|l|l| }
		\hline
		$G$  &  $\dim_{\C}\Sigma_{\inv}$ & Notable Spinors & Geometric Structures \\
		\hline
		$\SO(n+1)$         & $0$ &  & \\
		$\U(n+1)$       & $0$ &  & \\
		$\SU(n+1)$      & $2$ & generalized Killing & $\alpha$-Sasakian ({\tiny $\alpha = \frac{\sqrt{a(n+1)}}{2b\sqrt{n}}$}) \\
		$\Sp(n)$        & $2n$ & deformed Killing   & $\tad$   \\
		$\Sp(n)\Sp(1)$      & $1$ ($n=2$), $0$ ($n\neq 2$) & generalized Killing ($n=2$) & cocalibrated $\G_2$ ($n=2$)\\
		$\Sp(n)\U(1)$      & $2$ ($n$ even), $0$ ($n$ odd) & generalized Killing ($n=2$) & $\alpha$-Sasakian ($\frac{a}{b\Omega} = \frac{a}{2c\Omega}=\alpha$) \\
		$\G_2$    &   $2$ & Killing & nearly K\"{a}hler \\
		$\Spin(7)$       &  $1$ & Killing & nearly parallel $\G_2$ \\
		$\Spin(9)$     &  $1$ &  & \\
		\hline
	\end{tabular}
	\label{Tab:main_results_table}
\end{table}

The paper is organized as follows. In Section 2 we introduce the necessary background and definitions related to homogeneous spin structures, invariant metric connections, and connections with torsion. In Section 3 we treat the classical spheres with isometry groups defined over $\R$ or $\C$. These cases are simpler from the perspective of spin geometry, with invariant spinors only in the case $G=\SU(n+1)$. The invariant spinors in this case are generalized Killing spinors, and may be viewed as deformations of the invariant Killing spinors carried by the round (Sasakian) metric. In Section 4 we provide a unified approach to the three families with isometry groups defined over $\mathbb{H}$. We compare the resulting spinors with previously known results for Sasakian, 3-Sasakian, $\tad$, and $\G_2$ spaces, and provide to our knowledge the first examples of generalized Killing spinors with four eigenvalues. In Section 5 we provide a unified construction of the invariant spinors for the exceptional cases $G= \G_2$, $\Spin(7)$, $\Spin(9)$. In each case we also calculate explicitly the torsion of the Ambrose-Singer connection and its projection onto the space of totally skew-symmetric torsion tensors, as well as any relevant invariant differential forms.
To conclude, in Section 6 we consider homogeneous realizations of the round metric, and discuss the existence of invariant generalized Killing spinors. These are compared to the previous results of Moroianu and Semmelmann from \cite{GKSEinstein,GKSspheres}. 

\textbf{Acknowledgments:} The authors would like to thank Travis Schedler and Leander Stecker for helpful discussions, and Uwe Semmelmann for pointing out to us Remarks \ref{Uwe1} and \ref{Uwe2}. We would also like to thank the referee for their detailed and very helpful comments. This work was supported by the Engineering and Physical Sciences Research Council [EP/L015234/1]; the EPSRC Centre for Doctoral Training in Geometry and Number Theory (The London School of Geometry and Number Theory); University College London; King's College London; and Imperial College London.

\section{Preliminaries}
For a detailed introduction to spin geometry we refer to the classical monographies \cite{LM, BFGK, FriedrichBook}. For basic definitions and properties of (3-)contact, (3-)Sasakian, $\tad$ structures we recommend \cite{SasakianGeometry,3str}.
\subsection{Understanding the spin representation--the Lagrangian subspace approach}\label{understandingspinrep}
In order to understand spin structures we need to look at representations of the Spin group. For the sake of completeness and because this somewhat non-standard approach is crucial to our work, we will recall it here. We will deal only with complex representations in this survey and therefore consider an $n$-dimensional complex vector space $V$ together with a non-degenerate symmetric bilinear form $\beta$. Assuming first that $n=2l$, there exists a decomposition of the form
\begin{align}
\label{isotropic_decomposition_even} V= L\oplus L',
\end{align}
where $L$ and $L'$ are Lagrangians (i.e. $\beta$-isotropic subspaces of maximal dimension $l$). The bilinear form $\beta$ furnishes an isomorphism $L' \cong L^*$ via $l' \mapsto \beta(l',-)$. If instead $n=2l+1$, there is a decomposition
\begin{align}
V= L\oplus L'\oplus \mathbb{C}_0, \label{isotropic_decomposition_odd}
\end{align}
with $L$ and $L'$ Lagrangians with respect to $\beta$, $L \oplus L'$ nondegenerate with respect to $\beta$, and $\mathbb{C}_0 = (L \oplus L')^\perp$ the 1-dimensional $\beta$-orthogonal complement.  We put a subscript of zero to distinguish it from the ground field $\mathbb{C}$ itself. Since $\Spin(n)$ is a subgroup of the Clifford algebra $\mathbb{C}l(n)$, we first look at the irreducible Clifford algebra representations,
\[
\mathbb{C}l(n)\times \Sigma\rightarrow \Sigma.
\]
For a detailed description of these representations we refer to \cite{LM, FriedrichBook}. Recalling that $\dim \Sigma=2^l$, we note that a vector space of this dimension is furnished by the exterior algebra, $\Sigma:= \Lambda^{\bullet}L'$. Following the construction of \cite{GoodmanWallach}(albeit with a slightly different convention for the Clifford relation), we will give explicit formulas for the representation of $\mathbb{C}l(n)$ on $\Sigma$.

To begin, we view $V=\C^n$ as the complexification of $\R^n$, and take $\beta$ the $\C$-linear extension of the standard real Euclidean form, and $e_1,\dots, e_n$ the standard orthonormal basis vectors of $\R^n$. If $n=2l$ we choose the bases for $L$ and $L'$ given by
\begin{align} \label{originalLL'}
L=\text{span}_{\C}\{x_j:=\frac{1}{\sqrt{2}} (e_{2j-1}-ie_{2j})    \}_{j=1}^l, \quad
L'=\text{span}_{\C}\{y_j:=\frac{1}{\sqrt{2}} (e_{2j-1}+ie_{2j})     \}_{j=1}^l,
\end{align}
and if $n=2l+1$ we choose the same bases as above for $L$, $L'$, and take the vector 
\[
u_0:=ie_{2l+1}
\]
as a basis for the 1-dimensional complement $\C_0$ in (\ref{isotropic_decomposition_odd}). Throughout the paper we shall adopt the standard convention for the Clifford relation, 
\[
vw+ wv=-2\beta(v,w)1 . 
\]
With this convention, the action of $V$ on $\Sigma= \Lambda^{\bullet} L'$ is given by
\begin{align}
x_j\cdot \eta &:= i\sqrt{2} \ x_j\lrcorner \eta,\quad 
y_j\cdot \eta := i\sqrt{2} \ y_j\wedge \eta,\quad 
 u_0:= -\Id\rvert_{\Sigma^{\text{even}}} +\Id\rvert_{\Sigma^{\text{odd}}} , \label{clifford_rep_formulas}
\end{align}
where $\Sigma^{\text{even}}$ and $\Sigma^{\text{odd}} $ denote the even and odd graded parts of $\Sigma=\Lambda^{\bullet} L'$ (if $n=2l$ simply take the same operators for $x_j$, $y_j$ and ignore $u_0$). That this descends to a representation of $\C l(n)$ is a consequence of the identities (5.43) in \cite{GoodmanWallach}, which state the canonical (anti)commutation relations of the interior and exterior product operators. Noting that
\begin{align*}
e_{2j-1} &= \frac{1}{\sqrt{2}} (x_j+y_j),\quad e_{2j} = \frac{-i}{\sqrt{2}} (y_j-x_j), \quad 
e_{2l+1}= -iu_0,
\end{align*}
for $j=1,\dots,l$, we see that the corresponding operators associated to the real orthonormal basis vectors $e_1,\dots, e_{n}$ are 
\begin{align}\label{cliffordmultONB}
e_{2j-1}\cdot \eta  &=  i(x_j \lrcorner \eta + y_j\wedge \eta ),\quad e_{2j}\cdot \eta =  (y_j\wedge\eta - x_j\lrcorner\eta ),\quad 
e_{2l+1}=  i\Id\rvert_{\Sigma^{\text{even}}} -i\Id\rvert_{\Sigma^{\text{odd}}},
\end{align}
for all $\eta\in\Sigma$. 

We also note that there is an equivalent realization of this representation in terms of Kronecker products. Indeed, if we define subspaces $U_j:=\Lambda^{\bullet} \C y_j = \C 1 \oplus \C y_j$ of the spinor module $\Sigma$ then we have 
\[
\Sigma = \Lambda^{\bullet} L' = \Lambda^{\bullet} (\C y_1\oplus \dots \oplus \C y_l) \cong U_1\wedge \dots \wedge U_l, 
\]
giving the vector space isomorphism
\[
\Sigma \cong U_1\otimes \dots\otimes U_l. 
\]
Explicitly, this isomorphism is given by identifying $y_{j_1}\wedge \dots \wedge y_{j_p} \in\Lambda^{\bullet} L' $ with $u_1\otimes \dots\otimes u_l\in U_1\otimes \dots\otimes U_l$, where 
\[
u_j:=\begin{cases}
y_j &\quad\text{if } j\in \{j_1,\dots,j_p\},\\
1 &\quad\text{if } j\notin \{j_1,\dots,j_p\}
\end{cases} . 
\]
Under this identification we have 
\[
\End(\Sigma) \cong \End(U_1)\otimes \dots\otimes \End(U_l). 
\]
With the choice of ordered bases $\{1,y_j\}$ for the $U_j$, the representation (\ref{clifford_rep_formulas}) of $\C l(n)$ is realized by the Kronecker products
\begin{align*}
x_j&\mapsto  i\sqrt{2} \ H\otimes \dots \otimes H\otimes \underbrace{\begin{bmatrix} 0&1\\0&0\end{bmatrix}}_{\text{$j$th place}} \otimes \text{Id}\otimes \dots\otimes \text{Id} , \\
y_j&\mapsto  i\sqrt{2} \ H\otimes \dots \otimes H\otimes \underbrace{\begin{bmatrix} 0&0\\1&0\end{bmatrix}}_{\text{$j$th place}} \otimes \text{Id}\otimes \dots\otimes \text{Id},  \\
u_0&\mapsto   -H\otimes \dots\otimes H,
\end{align*}
where $H:=\text{diag}[1,-1]$. The corresponding operators associated to the real orthonormal basis $e_1,\dots, e_{n}$ are
\begin{align*}
e_{2j-1} &\mapsto i\ H\otimes \dots \otimes H\otimes \underbrace{\begin{bmatrix} 0&1\\1&0\end{bmatrix}}_{\text{$j$th place}} \otimes \text{Id}\otimes \dots\otimes \text{Id}, \\
e_{2j}&\mapsto  \ H\otimes \dots \otimes H\otimes \underbrace{\begin{bmatrix} 0&-1\\1&0\end{bmatrix}}_{\text{$j$th place}} \otimes \text{Id}\otimes \dots\otimes \text{Id}, \\
e_{2l+1}&\mapsto  i \ H \otimes\dots\otimes H.
\end{align*} 
\begin{remark}
	If $n=2l+1$ there is another inequivalent irreducible representation of the Clifford algebra, obtained by replacing the action of $u_0$ in (\ref{clifford_rep_formulas}) with
	\[
	u_0:= \Id\rvert_{\Sigma^{\text{even}}} -\Id\rvert_{\Sigma^{\text{odd}}}.
	\]
	The corresponding operator for the real orthonormal basis vector is
	\[
	e_{2l+1} =  -i\Id\rvert_{\Sigma^{\text{even}}} +i\Id\rvert_{\Sigma^{\text{odd}}},
	\]
	and, in the Kronecker product setting, 
	\[u_0\mapsto  H\otimes \dots \otimes H, \quad e_{2l+1} \mapsto -i H\otimes \dots \otimes H. \]
	In this paper we will always use the representation described in (\ref{clifford_rep_formulas}).	
\end{remark}
\begin{remark} \label{reorderedCliffalgrepresentation}
	In odd dimensions $n=2l+1$ it will be more convenient to use a slightly different basis ordering when defining the Clifford algebra representation. Indeed, we set
	\begin{align}\label{reorderedLL'}
	L&:= \text{span}_{\C} \{ x_j:= \frac{1}{\sqrt{2}} (e_{2j}-ie_{2j+1})\}_{j=1}^{2n-1} , \quad 
	L':= \text{span}_{\C} \{ y_j:= \frac{1}{\sqrt{2}} (e_{2j}+ie_{2j+1})\}_{j=1}^{2n-1} , \quad 
	u_0:= ie_1
	\end{align}  
	and then define the action of $x_j$, $y_j$, $u_0$ on $\Sigma:= \Lambda^{\bullet} L'$ via the formulas in (\ref{clifford_rep_formulas}). This alternative ordering will be used for all odd dimensional cases, and noted explicitly each time.
\end{remark}
\subsection{Spinors on Homogeneous Spaces}
The problem of explicitly finding spinors on arbitrary manifolds is in general quite difficult, but becomes more tractable if we restrict attention to reductive homogeneous spaces. Let $M=G/H$ be a reductive homogeneous space, and $\mathfrak{g}=\mathfrak{h}\oplus_{\perp} \mathfrak{m}$ a reductive decomposition which is orthogonal with respect to the Killing form on $\mathfrak{g}$. What follows is a general procedure described in the examples section of \cite{BFGK} for constructing invariant spinors on $G/H$, which will be used to construct all the spinors in this paper. Note that if $\mathfrak{g}$ is not semisimple (e.g. $G=\U(n+1)$ or $\Sp(n)\U(1)$), the reductive complement must be chosen differently since the Killing form is no longer non-degenerate. This will be explained in detail in the relevant sections. For a more detailed introduction to reductive homogeneous spaces we refer to \cite{Arvan}. 
First, we recall that the tangent space $T_e(G/H)$ is identified with $\mathfrak{m}$ as the kernel of the projection $G\to G/H$ at the identity element. The full tangent bundle is obtained via the associated bundle construction, $$TM= G\times_{\Ad\rvert_H} \mathfrak{m}, $$ where we are viewing $G$ as a principal $H$-bundle via the quotient map $G\to G/H$. The isotropy representation in this situation is given by the restriction $\Ad\rvert_H\: \mathfrak{m}\to \mathfrak{m}$. In fact $\Ad\rvert_H\in \SO(\mathfrak{m})$ with respect to any invariant metric (essentially by definition of invariance), and it is well-known that the oriented frame bundle is given by $$P_{\SO}=G\times_{\Ad\rvert_H} \SO(\mathfrak{m}) . $$ Assuming $H$ is connected, there exists a $G$-invariant spin structure if and only if the isotropy representation $\Ad\rvert_H$ lifts to the spin group, i.e.
\[
\begin{tikzcd}
& \Spin(\mathfrak{m}) \arrow[d, "2:1", two heads] \\
H \arrow[ru, "\widetilde{\Ad\rvert_H}", dashed] \arrow[r, "\Ad\rvert_H"'] & \SO(\mathfrak{m})                              
\end{tikzcd}
\]
and, in this case, the $G$-invariant spin structure is unique (see \cite[Prop.\@ 1.3, Cor.\@ 1.4]{invariantspinstructures}, as well as the earlier works \cite{Cahen_Gutt_spin_struct_symmetric_spaces,Hirzebruch_Slodowy_elliptic_genera} which cover certain special cases). If it exists, such a spin structure and its associated spinor bundle are given as homogeneous bundles via 
\[
P_{\Spin}:=G \times_{\widetilde{\Ad\rvert_H}} \Spin(\mathfrak{m}), \quad  \Sigma M :=P_{\Spin}\times_{\sigma} \Sigma = G\times_{\sigma\circ \widetilde{\Ad\rvert_H}} \Sigma, 
\]
where $\sigma\: \Spin(\mathfrak{m}) \to \Aut(\Sigma)$ denotes the spin representation (see \cite{invariantspinstructures}). As such, spinors are identified with $H$-equivariant maps $\varphi\: G\to \Sigma$, 
\begin{align}\label{spinoridentification}
\varphi(gh) = \sigma\circ\widetilde{\Ad\rvert_H}(h^{-1})\cdot  \varphi(g) \quad \forall g\in G,h\in H. 
\end{align}
The group $G$ acts on spinors by sending $\varphi \: G\to \Sigma$ to the map $(g\cdot \varphi)\: g' \mapsto \varphi(g^{-1}g')$, and we are interested in the space $\Sigma_{\text{inv}}$ of \emph{invariant spinors}, i.e. spinors fixed by this $G$-action. This is equivalent to the condition $\varphi(g^{-1}g')=\varphi(g')$ for all $g,g'\in G$, which holds precisely when the $H$-equivariant map $\varphi\: G \to \Sigma$ is constant. By setting $\varphi$ to be constant in (\ref{spinoridentification}), it is easy to see that invariant spinors correspond to trivial $H$-subrepresentations of the spinor module for the action $\sigma\circ\widetilde{\Ad\rvert_H}$. Using the Lie algebra isomorphism $\mathfrak{spin}(\mathfrak{m}) \cong \mathfrak{so}(\mathfrak{m})$ and assuming $H$ is connected, this condition may equivalently be phrased as looking for trivial subrepresentations of $\sigma\circ \widetilde{\ad\rvert_\mathfrak{h}}\: \mathfrak{h}\to \mathfrak{gl}(\Sigma)$. Counting such spinors is therefore a representation theory problem, which we approach using various methods for the homogeneous spheres in Table \ref{Tab:homogeneousspheres}.

A similar viewpoint may be taken when considering differential forms (or tensors, for that matter) on $M=G/H$. The isotropy group $H$ acts on $\Lambda^k\mathfrak{m}$ by the $k$th exterior power of the isotropy representation,
\[ 
(\Lambda^k  \Ad\rvert_H)(h)(\omega)(X_1,\dots, X_k) =  \omega(\Ad(h)^{-1}X_1,\dots,\Ad(h)^{-1}X_k  ), \quad h\in H,\  \omega\in\Lambda^k\mathfrak{m},  \ X_i\in \mathfrak{m},
\]
and the bundle of $k$-forms is realized via the associated bundle construction,
\[
\Lambda^k TM = G\times_{\Lambda^k\Ad\rvert_H} \Lambda^k\mathfrak{m}.
\]
Invariant $k$-forms then correspond to constant $H$-equivariant maps $\omega \: G\to \Lambda^k\mathfrak{m}$, or equivalently, trivial $H$-subrepresentations of $\Lambda^k\mathfrak{m}$. 

Let us now fix notation related to matrix Lie algebras. We will use $E_{i,j}^{(n)}$ (resp. $F_{i,j}^{(n)}$) throughout to denote the elementary skew-symmetric $n\times n$ matrix (resp. the elementary symmetric $n\times n$ matrix),
\begin{align*}
E_{i,j}^{(n)}=\begin{blockarray}{r*{4}{ >{}c}}
&  & i & j &  \\
\begin{block}{ r!{\,}[cccc]}
      & &  & \vdots &  \\
i      & &  & -1     & \hdots  \\
j     &\hdots & 1&  &  \\
      & & \vdots  &  &  \\
\end{block}
\end{blockarray} 
, \quad F_{i,j}^{(n)} = \begin{blockarray}{r*{4}{ >{}c}}
&  & i & j &  \\
\begin{block}{ r!{\,}[cccc]}
& &  & \vdots &  \\
i      & &  & 1     & \hdots  \\
j     &\hdots & 1&  &  \\
& & \vdots  &  &  \\
\end{block}
\end{blockarray} .
\end{align*}
We also adopt the convention that $F_{i,i}^{(n)}$ is the diagonal matrix with $1$ in the $(i,i)$ position and zeros elsewhere. 
The following commutator relations will be used extensively throughout the paper for calculations involving matrix Lie algebras.
\begin{align*}
[E_{i,j}^{(n)},E_{k,l}^{(n)}] &= \begin{cases*} E_{j,l}^{(n)} & if $i=k$, \\
0  & if $i,j,k,l$ distinct,
\end{cases*}    \\
 [E_{i,j}^{(n)},F_{k,l}^{(n)}] &=  \begin{cases*} F_{j,l}^{(n)} & if $i=k$, $j\neq l$, $k\neq l$, \\
2( F_{j,j}^{(n)} - F_{i,i}^{(n)}) & if $i=k$, $j=l$, $k\neq l$, \\
(\delta_{i,k}F_{j,k}^{(n)} - \delta_{j,k}F_{i,k}^{(n)}) & if $k=l$, \\
 0  & if $i,j,k,l$ distinct, 
 \end{cases*} \\
  [ F_{p,q}^{(n)}, F_{r,s}^{(n)} ] &=  \begin{cases*}  - E_{q,s}^{(n)} & if $p=r$, $q\neq s$, $p\neq q$, $r\neq s$, \\
  (-\delta_{q,r}E_{p,r}^{(n)} - \delta_{p,r}E_{q,r}^{(n)}) & if $p\neq q$, $r=s$, \\
  0 & if $p=q$, $r=s$, \\
  0  & if $p,q,r,s$ distinct 
  \end{cases*} \\
    [\lambda_1 F_{p,q}^{(n)},\lambda_2 F_{r,s}^{(n)} ] &=  \begin{cases*} \lambda_3 F_{q,s}^{(n)} & if $p=r$, $q\neq s$, $p\neq q$, $r\neq s$, \\
   2\lambda_3(F_{p,p}^{(n)} + F_{q,q}^{(n)}) & if $p=r$, $q=s$, $p\neq q$,     \\
   \lambda_3(\delta_{q,r}F_{p,r}^{(n)} + \delta_{p,r}F_{q,r}^{(n)}) &   if $p\neq q$, $r=s$,    \\
  2 \delta_{p,r}\lambda_3 F_{p,p}^{(n)} & if $p=q$, $r=s$,       \\
  0  & if $p,q,r,s$ distinct, 
  \end{cases*}
\end{align*}
where $(\lambda_1,\lambda_2,\lambda_3)$ is an even permutation of the imaginary quaternions $(i,j,k)$. Note that this doesn't cover all possible cases, however the rest can be deduced from above using skew-symmetry (resp. symmetry) of the matrices $E_{i,j}^{(n)}$ (resp. $F_{i,j}^{(n)}$). We shall use $B_0$ to denote the bilinar form on the space of matrices (of the appropriate size, depending on context) given by
\begin{align}\label{B0definition}
B_0(X_1,X_2) := - \Re \tr (X_1X_2).
\end{align}
 After fixing an invariant inner product on $\mathfrak{m}$, an orthonormal basis will be denoted by $e_1,\dots, e_{\dim \mathfrak{m}}$, and the shorthand $e_{i_1,\dots, i_p}:= e_{i_1}\wedge\dots \wedge e_{i_p}$ for differential forms will be used.
\subsection{Invariant Metric Connections on Homogeneous Spaces}
By applying translations, any invariant connection on $G/H$ is uniquely determined by its value at the origin $o:=eH$, and a similar principal applies to spinorial connections. Indeed, under the identification $\mathfrak{m}\cong T_o(G/H)$, it follows from a result of Nomizu in \cite{Nomizumap}, later generalized by Wang in \cite{Wangconnections}, that an invariant metric connection corresponds to the data of an $\Ad(H)$-equivariant \emph{Nomizu map} 
\begin{align}\label{nom_map_explanation}
\Uplambda\: \mathfrak{m}\to\mathfrak{so}(\mathfrak{m}).
\end{align}
Explicitly, the relationship between $\Uplambda$ and the covariant derivative $\nabla$ associated to the connection is
\[
(\nabla_{\widehat{X}} \omega)_o = \Uplambda(X)\omega_o, \quad X\in \mathfrak{m},
\]
for any invariant tensor (or invariant differential form) $\omega$, where $\widehat{X}$ is the fundamental vector field associated to $X\in\mathfrak{m}$ and the action of $\Uplambda(X)\in \mathfrak{so}(\mathfrak{m})$ on $\omega_o$ is the natural one (see Chapter 6 in \cite{ANT_book_principal_fibre_bundles} for a modern treatment of the topic). Moreover, by Proposition 2.3 in \cite{KN2}, the torsion and curvature tensors of $\nabla$ are given at the origin by
\begin{align}
	T_o(X,Y) &= \Uplambda(X)Y - \Uplambda(Y)X - [X,Y]_{\mathfrak{m}},\label{torsionatorigin}\\
	R_o(X,Y)   &= [\Uplambda(X),\Uplambda(Y)] - \Uplambda([X,Y]_{\mathfrak{m}}) - \ad([X,Y]_{\mathfrak{h}}),
\end{align}
for all $X,Y\in\mathfrak{m}$. Composing (\ref{nom_map_explanation}) with the Lie algebra isomorphism $\mathfrak{spin}(\mathfrak{m})\cong \mathfrak{so}(\mathfrak{m})$ gives the Nomizu map  
\[
\widetilde{\Uplambda}\: \mathfrak{m}\to \mathfrak{spin}(\mathfrak{m})
\]
associated to the spin lift $\widetilde{\nabla}$ of $\nabla$,
and the covariant derivative at the origin of an invariant spinor $\psi$ can be similarly described via
\[
(\widetilde{\nabla}_{\widehat{X}})_o \psi = \widetilde{\Uplambda}(X)\cdot \psi_o ,
\]
where the action of $\widetilde{\Uplambda}(X)$ on $\psi_0$ is via the spin representation.

For an invariant Riemannian metric $g$ on $G/H$, the Nomizu map $\Uplambda^g \: \mathfrak{m}\to \mathfrak{so}(\mathfrak{m})$ of the Levi-Civita connection is given by
\begin{align}
\Uplambda^g(X)Y= \frac{1}{2}[X,Y]_{\mathfrak{m}} + U(X,Y), \quad \forall X,Y\in\mathfrak{m}, \label{nomizumap}
\end{align}
where the symmetric $(2,0)$-tensor $U$ is determined by 
\begin{align} 
2g(U(X,Y) ,Z)  = g( [Z,X]_{\mathfrak{m}},Y) + g( X,[Z,Y]_{\mathfrak{m}}) \label{Utensor} . 
\end{align}
For a proof of this fact we refer to Theorem 13.1 in \cite{Nomizumap}, noting that there is a sign error in equation (13.1).

Another geometrically significant invariant connection is the \emph{Ambrose-Singer connection}, sometimes called the canonical connection, whose horizontal distribution is generated by left translations of $\mathfrak{m}\subset \mathfrak{g}$. Such a connection is unique after fixing a reductive complement $\mathfrak{m}$. In the present paper we shall always refer to this as the Ambrose-Singer connection, and reserve the term canonical connection for the distinguished metric connection on $\tad$ manifolds introduced in \cite{3str}. The Ambrose-Singer connection has Nomizu map identically equal to zero, 
\[
\Uplambda^{\AS} \equiv 0,
\]
and it parallelizes all invariant tensors (Proposition 2.7 in \cite{KN2}). Noting that the Ambrose-Singer connection coincides with the Levi-Civita connection if and only if its torsion tensor vanishes, it is evident from (\ref{torsionatorigin}) that they coincide precisely when the underlying space is symmetric. One sees furthermore that the Ambrose-Singer torsion is totally skew-symmetric (i.e. a 3-form) if and only if $g$ is a naturally reductive metric,
\[
g([X,Y]_{\mathfrak{m}}, Z) + g(Y,[X,Z]_{\mathfrak{m}})=0,\quad \forall X,Y,Z\in\mathfrak{m}.
\]

Fixing notation, the Levi-Civita and Ambrose-Singer connections, their corresponding Nomizu maps, and their torsion tensors will be denoted by $\nabla^g$, $\nabla^{\AS}$, $\Uplambda^g$, $\Uplambda^{\AS}$, and $T^g$, $T^{\AS}$ respectively. By abuse of notation we shall denote the corresponding spinorial connections also by $\nabla^g$, $\nabla^{\AS}$, and the associated spinorial Nomizu maps by $\widetilde{\Uplambda}^g$, $\widetilde{\Uplambda}^{\AS}$. Any other connections used will be introduced in the relevant sections.
\subsection{Metric Connections with Torsion}
Let $(M^n,g)$ be an $n$-dimensional Riemannian manifold. In certain situations it will be advantageous to consider metric connections other than the Levi-Civita connection which are better adapted to the geometry at hand. Such connections are uniquely determined by their torsion tensor, and for a detailed introduction to the subject we refer to \cite{SRNI}. The space of possible torsion tensors is given by
\[
\mathcal{T} := \{ T\in TM^{\otimes 3} \: T(X,Y,Z) + T(Y,X,Z) = 0\},
\]
and it splits as an $\text{O}(n)$-representation into three inequivalent irreducible submodules,
\[
\mathcal{T} \simeq  \mathcal{T}_{\vectorial} \oplus \mathcal{T}_{\totallyskew} \oplus \mathcal{T}_{\CT},
\]
called \emph{torsion classes}. Metric connections with torsion in these three spaces are called \emph{vectorial}, \emph{totally skew-symmetric}, and \emph{cyclic traceless} respectively.

For a metric connection $\nabla$, the \emph{difference tensor} is defined by
\[
A(X,Y) := \nabla_X Y - \nabla^g_XY.
\]
We note that the space
\[
\mathcal{A}^g := \{ A \in TM^{\otimes 3}\: A(X,Y,Z) + A(X,Z,Y)=0  \}
\]
of possible difference tensors is isomorphic to $\mathcal{T}$ as $\text{O}(n)$-representations via
\begin{align}\label{torsionisomorphism1}
T(X,Y,Z) &= A(X,Y,Z) - A(Y,X,Z) , \\ A(X,Y,Z)&= \frac{1}{2}(T(X,Y,Z) - T(Y,Z,X) + T(Z,X,Y)) .\label{torsionisomorphism2}
\end{align}

Let $\nabla$ be a metric connection with torsion $T\in\mathcal{T}$ and difference tensor $A \in\mathcal{A}^g$. With respect to an arbitrary orthonormal frame $e_1,\dots, e_n$, we define the trace
\[
c_{12}(A) := \sum_{i=1}^n A(e_i,e_i, - ) .
\]
The images of the torsion classes under the isomorphism (\ref{torsionisomorphism2}) are given in Chapter 3 of \cite{tricerri_vanhecke_1983} by 
\begin{align}
\mathcal{A}^g_{\vectorial} &= \{ A \in\mathcal{A}^g \: A(X,Y,Z) = g(X,Y)g(V,Z) - g(X,Z)g(V,Y), \  V\in TM \}, \\ \mathcal{A}^g_{\totallyskew}&=\{ A\in\mathcal{A}^g \: A(X,Y,Z)+A(Y,X,Z)=0  \}, \\ 
\mathcal{A}^g_{\CT} &= \{ A\in\mathcal{A}^g\: \mathfrak{S}_{X,Y,Z} A(X,Y,Z)=0, \ c_{12}(A)=0\},
\end{align}
as well as explicit formulas for the projections of $A$ onto each class,
\begin{align}
\label{vectorialprojection} A_{\vectorial}(X,Y,Z) &=  g(X,Y) \phi(Z) - g(X,Z) \phi(Y) , \\
\label{skewprojection} A_{\totallyskew}(X,Y,Z) &= \frac{1}{3}\mathfrak{S}_{X,Y,Z} A(X,Y,Z) , \\
\label{CTprojection} A_{\CT}(X,Y,Z) &= A(X,Y,Z) - A_{\vectorial}(X,Y,Z) - A_{\totallyskew} (X,Y,Z) ,
\end{align}
where
\[
\phi(v) := \frac{1}{n-1} c_{12}(A)(v), \quad v\in TM.
\]
Formulas for the projections of $T$ onto the three torsion classes may then be easily deduced using (\ref{torsionisomorphism1}). By examining the torsion type of the Ambrose-Singer connection, as we do in this article for each homogeneous realization of the sphere, one obtains intrinsic geometric information about the homogeneous structure.
\section{Classical Spheres, Part I: Spheres over $\R$ and $\C$}
\subsection{Symmetric Spheres, $S^n = \SO(n+1)/\SO(n)$}
The isotropy representation here is the standard representation of $\SO(n)$ on $\R^n$, which is irreducible, hence the only invariant metrics correspond to negative multiples of the Killing form (equivalently, positive multiples of $B_0$). We remark that any such metric is naturally reductive. The embedding $\SO(n)\hookrightarrow \SO(n+1)$ may be realized as the the lower right hand $n\times n$ block, and we choose the reductive complement $\mathfrak{m}=(\mathfrak{so}(n))^{\perp}$ with respect to the Killing form. Explicitly,
\begin{align*}
\mathfrak{so}(n+1)&= \Span_{\R} \{ E_{i,j}^{(n+1)}  \}_{1\leq i< j\leq n+1},\\
\mathfrak{so}(n)&= \Span_{\R} \{ E_{i,j}^{(n+1)}  \}_{2\leq i<j \leq n+1},
\end{align*}
and
\begin{align*}
\mathfrak{m}&= \Span_{\R} \{ E_{1,j}^{(n+1)} \}_{2\leq j \leq n+1}.
\end{align*}
One sees immediately from the main proposition in \cite{Wang} that these standard round spheres are not very interesting from the viewpoint of homogeneous spin geometry:
\begin{theorem}\label{SOtheorem}
	The spheres $S^n=\SO(n+1)/\SO(n)$ do not admit a non-trivial invariant spinor for any choice of invariant metric.
\end{theorem}
\begin{remark}
		Since any invariant metric on $S^n=\SO(n+1)/\SO(n)$ is normal homogeneous, and in particular naturally reductive, the Ambrose-Singer connection always has totally skew-symmetric torsion, $T^{\AS} \in \mathcal{T}_{\totallyskew}$ (in fact, $T^{\AS}=0$ since the space is symmetric). This also applies to the other two realizations of the sphere with irreducible isotropy representation, $S^6=\G_2/\SU(3)$ and $S^7=\Spin(7)/\G_2$.
\end{remark}
\subsection{Hermitian Spheres, $S^{2n+1}= \U(n+1)/\U(n)$}
The isotropy representation splits into one copy of the trivial representation $\R$ and one copy of $\R^{2n}\cong \C^n$, leading to a 2-parameter family of invariant metrics. Note, however, that the Killing form is no longer non-degenerate so more care must be taken when choosing a reductive complement. The embedding $\U(n)\hookrightarrow \U(n+1)$ may be realized as the lower right hand $n\times n$ block, leading to the realization of Lie algebras given by
\begin{align*}
\mathfrak{u}(n+1)&= \Span_{\R} \{  E_{j,k}^{(n+1)}, iF_{p,q}^{(n+1)}    \}_{\substack{1\leq j<k\leq n+1\\ p,q=1,\dots  n+1}}, \\
\mathfrak{u}(n)&= \Span_{\R} \{  E_{j,k}^{(n+1)}, iF_{p,q}^{(n+1)}    \}_{\substack{2\leq j<k\leq n+1\\ p,q=2,\dots  n+1}},
\end{align*}
and one verifies that
\begin{align*}
\mathfrak{m}&:= \Span_{\R} \{  iF_{1,1}^{(n+1)},  E_{1,j+1}^{(n+1)},  iF_{1,j+1}^{(n+1)} \}_{1\leq j \leq n}
\end{align*}
is a reductive complement. The two irreducible submodules are given by
\begin{align*}
\mathfrak{m}_1:= \Span_{\R}\{ iF_{1,1}^{(n+1)} \}, \quad \mathfrak{m}_2:= \Span_{\R} \{ E_{1,j+1}^{(n+1)},iF_{1,j+1}^{(n+1)}\}_{1\leq j \leq n},
\end{align*}
and the 2-parameter family of invariant metrics is given by
\begin{align*}
g_{a,b}:=a B_0\rvert_{\mathfrak{m}_1\times \mathfrak{m}_1} + b B_0\rvert_{\mathfrak{m}_2\times \mathfrak{m}_2}, \quad a,b>0.
\end{align*}
These spheres are the complex analog of the previous case and, as such, one may deduce a similar result about the space of invariant spinors from \cite{Wang} by noting that $\mathfrak{m}_2\simeq \C^{n}$ is isomorphic to the standard representation of $\U(n)$ and that the spinor modules in dimensions $2n$ and $2n+1$ are naturally identified. Here we give an alternative elementary proof of this result:
\begin{theorem}\label{Utheorem}
	The spheres $S^{2n+1}=\U(n+1)/\U(n)$ do not admit a non-trivial invariant spinor for any choice of invariant metric. 
\end{theorem}
\begin{proof}
The basis $e_1,\dots, e_{2n+1}$ for $\mathfrak{m}$ given by 
\begin{align*} 
e_1:=\frac{1}{\sqrt{a}} iF_{1,1}^{(n+1)}, \quad e_{2j}:=\frac{1}{\sqrt{2b}} E_{1,j+1}^{(n+1)}, \quad e_{2j+1} := \frac{1}{\sqrt{2b}} iF_{1,j+1}^{(n+1)},
\end{align*}
for $j=1,\dots, n$ is orthonormal with respect to $g_{a,b}$, and the isotropy algebra is spanned by the operators
\begin{align*}
\ad(E_{j,k}^{(n+1)}) &= e_{2j-2}\wedge e_{2k-2}+e_{2j-1}\wedge e_{2k-1}, \\  \ad(iF_{p,q}^{(n+1)}) &= e_{2p-2} \wedge e_{2q-1} + e_{2q-2} \wedge e_{2p-1}     \quad (p\neq q), \\
\ad(iF_{p,p}^{(n+1)}) &= e_{2p-2} \wedge e_{2p-1}  .
\end{align*}
In particular the lifts of the operators $\ad(iF_{p,p}^{(n+1)} )$ act on the spinor bundle via Clifford multiplication by $\frac{1}{2}e_{2p-2}\cdot e_{2p-1}$, and the result then follows by noting that if $\psi\in \Sigma_{\inv}$ then
\[
 0 = || e_{2p-2}\cdot e_{2p-1}\cdot \psi||^2 = \langle e_{2p-2}\cdot e_{2p-1}\cdot \psi,e_{2p-2}\cdot e_{2p-1}\cdot \psi\rangle = \langle \psi,\psi\rangle = ||\psi||^2.
\] 
\end{proof}
In the following proposition we calculate the Ambrose-Singer torsion and determine its type. We show that it has mixed skew-symmetric and cyclic traceless type in general, and has pure skew-symmetric type if and only if $a=b$, which means that the only naturally reductive metric on this space is the normal homogeneous metric. Our result also shows that it never has pure cyclic traceless type.
\begin{proposition}
	For any $a,b>0$ the sphere $(S^{2n+1}=\U(n+1)/\U(n),g_{a,b})$ has Ambrose-Singer torsion of type $\mathcal{T}_{\totallyskew} \oplus \mathcal{T}_{\CT}$, given by
	\begin{align*}
	T^{\AS}(e_1,e_{2j}) &= \frac{1}{\sqrt{a}} e_{2j+1}, \quad T^{\AS}(e_1,e_{2j+1}) = \frac{-1}{\sqrt{a}} e_{2j} , \\
	T^{\AS}(e_{2j},e_{2l}) &= T^{\AS}(e_{2j+1},e_{2l+1})= 0,\quad T^{\AS}(e_{2j},e_{2l+1})= \frac{\delta_{j,l}\sqrt{a}}{b} e_1   ,
	\end{align*}
for all $j,l=1,\dots, n $. The projection of $T^{\AS}$ onto $\mathcal{T}_{\totallyskew}$ is
\[
T^{\AS}_{\totallyskew}:= \left(\frac{a+2b}{3b\sqrt{a}}\right) \sum_{j=1}^{n} e_1\wedge e_{2j}\wedge e_{2j+1},
\]
with $T^{\AS}=T^{\AS}_{\totallyskew}$ if and only if $a=b$ (i.e. $g_{a,b}$ is a multiple of the Killing form).
\end{proposition}
\begin{proof}
	Straightforward calculation of the commutator relations, and subsequent application of (\ref{torsionatorigin}), (\ref{vectorialprojection})-(\ref{CTprojection}) and the isomorphism (\ref{torsionisomorphism1}).
\end{proof}
\subsection{Special Hermitian Spheres, $S^{2n+1}=\SU(n+1)/\SU(n)$}
The isotropy group $\SU(n)\hookrightarrow \SU(n+1)$ may be realized as the lower right hand $n\times n$ block. We take the reductive complement $\mathfrak{m}:=(\mathfrak{su}(n))^{\perp_{B_0}}$, where the orthogonal complement is taken with respect to $B_0$ (a multiple of the Killing form on $\mathfrak{su}(n+1)$). At the level of Lie algebras,
\begin{align*}
	\mathfrak{su}(n+1)&= \Span_{\R} \{ iF_{p,q}^{(n+1)}, E_{p,q}^{(n+1)} , i(-nF_{1,1}^{(n+1)} +\sum_{l=2}^{n+1} F_{l,l}^{(n+1)}), i(F_{r,r}^{(n+1)} - F_{r+1,r+1}^{(n+1)})  \}_{\substack{  1\leq p < q \leq n+1 \\ r=2,\dots ,n+1 }}, \\ \mathfrak{su}(n)&=     \Span_{\R} \{ iF_{p,q}^{(n+1)}, E_{p,q}^{(n+1)} , i(F_{r,r}^{(n+1)} - F_{r+1,r+1}^{(n+1)})  \}_{\substack{  2\leq p < q \leq n+1 \\ r=2,\dots ,n+1 }}, 
\end{align*}
and
\begin{align*}
	\mathfrak{m}&= \Span_{\R} \{ i(-nF_{1,1}^{(n+1)} +\sum_{l=2}^{n+1} F_{l,l}^{(n+1)}), iF_{1,p}^{(n+1)}, E_{1,p}^{(n+1)}  \}_{p=2,\dots, n+1}.
\end{align*}
The isotropy representation splits into one copy of the trivial representation and one copy of the standard representation, $\mathfrak{m}\simeq \mathfrak{m}_1\oplus \mathfrak{m}_2$, leading to the 2-parameter family of invariant metrics
\begin{align*}
	g_{a,b}&:= aB_0\rvert_{\mathfrak{m}_1} + b B_0\rvert_{\mathfrak{m}_2}, \quad a,b>0.
\end{align*}
A $g_{a,b}$-orthonormal basis of $\mathfrak{m}$ is given by $\{ e_i\}_{i=1}^{2n+1}$, where
\begin{align*}
	e_1 &:= \frac{1}{\sqrt{an(n+1)}} (-niF_{1,1}^{(n+1)} + \sum_{l=2}^{n+1} iF_{l,l}^{(n+1)}  ),\quad
	e_{2p}:= \frac{1}{\sqrt{2b}} E_{1,p+1}^{(n+1)},\quad 
	e_{2p+1}:= \frac{i}{\sqrt{2b}}F_{1,p+1}^{(n+1)},
\end{align*} 
for $p=1,\dots, n$, and the two isotropy summands may be written explicitly in terms of this basis as
\[
\mathfrak{m}_1= \Span_{\R}\{ e_1\}, \quad \mathfrak{m}_2=\Span_{\R}\{ e_2,\dots, e_{2n+1}\}.
\] 
The complexified algebra $\mathfrak{su}(n+1)^{\C}$ has a Cartan subalgebra spanned by
\begin{align*} 
	\tau_k&:= \frac{1}{\sqrt{k(k+1)}} (\sum_{p=2}^{k+1} F_{p,p}^{(n+1)} - k F_{k+2,k+2}^{(n+1)} ), \qquad k=1,\dots,n-1, \\
	\tau_n&:=i\sqrt{a} e_1,
\end{align*}
and the elements $\tau_1,\dots,\tau_{n-1}$ span a Cartan subalgebra for the complexified isotropy algebra $\mathfrak{su}(n)^{\C}$. A straightforward calculation then gives,
\begin{proposition}
	The above Cartan subalgebra of $\mathfrak{su}(n)^{\C}$ acts on $\mathfrak{m}^{\C}$ via
	\begin{align*}
		\ad(\tau_k)e_{2p} &=\begin{cases*}
			\frac{-i}{\sqrt{k(k+1)}} e_{2p+1}  & if $p\leq k$, \\
			\frac{ik}{\sqrt{k(k+1)}}e_{2p+1}       & if $p=k+1$, \\
			0        & if $p\geq k+2$,
		\end{cases*}  \qquad 
		\ad(\tau_k)e_{2p+1} =\begin{cases*}
			\frac{i}{\sqrt{k(k+1)}} e_{2p}  & if $p\leq k$, \\
			\frac{-ik}{\sqrt{k(k+1)}}e_{2p}       & if $p=k+1$, \\
			0        & if $p\geq k+2$,
		\end{cases*}
	\end{align*}
	and $\ad(\tau_k)e_1=0$ for $k=1,\dots, n-1$. 
\end{proposition}
\begin{corollary}
	The isotropy representation maps the above Cartan subalgebra of $\mathfrak{su}(n)^{\C}$ into $\mathfrak{so}(\mathfrak{m}^{\C},g_{a,b}^{\C} )\cong \mathfrak{so}(2n+1,\C)$ as the operators 
	\begin{align}\label{specialunitaryisotropyoeprators}
		\tau_k \mapsto \ad(\tau_k)\rvert_{\mathfrak{m}^{\C}} = \frac{-i}{\sqrt{k(k+1)}} \left(\sum_{p=1}^k e_{2p}\wedge e_{2p+1}  - k e_{2k+2}\wedge e_{2k+3} \right), 
	\end{align}
for $k=1,\dots, n-1$.
\end{corollary}
\begin{theorem} \label{deformedSasakianinvariantspinors}
	For any $a,b>0$, the space of invariant spinors on $ (S^{2n+1}=\frac{\SU(n+1)}{\SU(n)}, g_{a,b})$ is given by
	\begin{align*}
		\Sigma_{\inv} &= \Span_{\C} \{ \psi_+:=1, \ \psi_-:= y_1\wedge y_2\wedge \dots \wedge y_n\}.
	\end{align*}
\end{theorem}
\begin{proof}
Considering the spin lifts of the operators in (\ref{specialunitaryisotropyoeprators}),
%
%
%
%
%
%
%
%
%
%
one notes that $\widetilde{\ad(\tau_k)}\rvert_{\mathfrak{m}^{\C}}\cdot \psi=0$ for all $k=1,\dots, n-1$ if and only if $e_{2p}\cdot e_{2p+1}\cdot \psi = e_{2p+2}\cdot e_{2p+3}\cdot \psi $ for all $p=1,\dots,n-1$. Note that this condition is necessarily satisfied if $\psi \in \Sigma_{\inv}$. Using the representation described in Remark \ref{reorderedCliffalgrepresentation}, one has
\begin{align*}
e_{2p}\cdot e_{2p+1}\cdot \psi &= i(x_p\lrcorner + y_p\wedge)(y_p\wedge -x_p \lrcorner) \psi  = \dots = i[\psi -2y_p\wedge (x_p\lrcorner\psi)],\\
e_{2p+2}\cdot e_{2p+3}\cdot \psi &= i(x_{p+1}\lrcorner + y_{p+1}\wedge)(y_{p+1}\wedge -x_{p+1} \lrcorner) \psi  = \dots = i[\psi -2y_{p+1}\wedge (x_{p+1}\lrcorner\psi)],
\end{align*}
and hence $\Sigma_{\inv} \subseteq \Span_{\C} \{ 1, y_1\wedge y_2\wedge \dots \wedge y_n\}$. Thus it suffices to show that there are two linearly independent invariant spinors. Since the isotropy representation decomposes as the sum of one copy of the trivial representation and one non-trivial module, the number of invariant spinors is independent of the choice of $a,b>0$. In particular we consider the round metric, corresponding to the parameters $a=\frac{n}{n+1}$, $b=\frac{1}{2}$, together with its usual $\SU(n+1)$-invariant Sasakian structure (see \cite{draper_spheres} for a more detailed description). Denoting by $(\varphi,\xi:=e_1,\eta:=\xi^{\flat})$ the Sasakian structure tensors, it follows from results in \cite{Fried90,BFGK} that the spaces $$E_{\pm}:=\{\psi\in\Gamma(\Sigma M)\:( \pm 2\varphi(X)+\xi\cdot X-X\cdot \xi)\cdot \psi=0 \ \forall X\in\vect(M) \} $$
satisfy $\dim(E_+ + E_-)=2$, and hence it suffices to show that they have a basis consisting of invariant spinors. One also remarks from \cite{draper_spheres} that $\varphi$ is an invariant tensor (in fact, using their setup one finds the explicit algebraic description $\varphi= \frac{n}{n+1}\ad(\xi)$).
Let $\phi\in\Gamma(E_+)$, so that $$(2\varphi(X) +\xi\cdot X-X\cdot \xi)\cdot \phi =0 \quad \forall X\in\vect(M).$$
Since $\varphi$ and $\xi$ are both invariant tensors, it suffices to consider this defining equation at the origin (i.e. for $X\in\mathfrak{m}$). By performing a similar type of calculation as in the proof of Proposition 7.1 in \cite{kath_Tduals}, it follows that for any $g_0\in \SU(n+1)$ we have
\begin{align*}
((2\varphi(X) +\xi\cdot X-X\cdot \xi)\cdot (g_0\phi))(g)   &= 	(2\varphi(X)\cdot (g_o\phi ) +\xi\cdot X\cdot (g_0\phi)-X\cdot \xi\cdot (g_0\phi))(g) \\
&= ((2\varphi(X)+\xi\cdot X -X\cdot \xi)\cdot \phi)(g_0^{-1}g)\\
&=0,
\end{align*}
where we have slightly abused notation to denote a spinor and the corresponding $\SU(n)$-equivariant map $\SU(n+1)\to \Sigma$ by the same symbol. Similarly for $\phi\in \Gamma(E_-)$. This shows that the spaces $E_{\pm}$ are representations of $\SU(n+1)$. However we know from \cite{BFGK} that $\dim(E_{\pm})\leq 2$, and thus they must be trivial representations for $n\geq 2$, proving the result in these cases. For $n=1$ the isotropy group is trivial $\SU(1)=\{e\}$, so every spinor is invariant. The spinor module in this dimension is 2-dimensional, so in particular there are two linearly independent invariant spinors.
\end{proof}
\begin{remark}
The fact that the space of invariant spinors is 2-dimensional may also be deduced from part (a) of the main proposition in \cite{Wang}, since the spinor modules in dimensions $2n$ and $2n+1$ are naturally identified and the isotropy representation acts trivially on $\R e_1$.
\end{remark}
\begin{remark}\label{change_of_basis_effect}
A priori, choosing a different orthonormal basis for $\mathfrak{m}$ can lead to different expressions for the invariant spinors, since the identification from Section \ref{understandingspinrep} of spinors with (algebraic) exterior forms is very much basis dependent. This runs counter to the natural expectation that the invariants here should be spanned by $1$ and the anti-holomorphic volume form, however, by choosing a well-suited orthonormal basis for $\mathfrak{m}$ one can avoid this problem. Indeed, our chosen $g_{a.b}$-orthonormal basis $\{e_i\}$ is \emph{adapted} to the invariant almost complex structure $\varphi=\sqrt{\frac{an}{n+1}}\ad(e_1)$ on $\mathfrak{m}_2=(\R e_1)^{\perp}$ in the sense that $\varphi(e_{2p})=e_{2p+1}$ for $p=1,\dots, n$. This gives $\mathfrak{m}_2$ the structure of a complex representation (which is isomorphic to the standard representation of $\SU(n)$ on $\C^n$), and complexifying the full isotropy representation therefore gives:
\[
\mathfrak{m}^{\C} = (\mathfrak{m}_1\oplus \mathfrak{m}_2)^{\C} \simeq (\R \oplus \mathfrak{m}_2)^{\C} \simeq \C \oplus \mathfrak{m}_2 \oplus \mathfrak{m}_2^*.
\]
In particular, this shows that the image of the isotropy representation lies inside $\mathfrak{gl}(L)\subseteq \mathfrak{so}(\mathfrak{m^{\C}})$ (see Section \ref{quaternionic_case_setup} for the details of this inclusion), which by Proposition \ref{p:gl-action} then implies that $\Sigma \simeq \Lambda^{0,\bullet}\mathfrak{m}$ as complex representations. It follows that the spinors $\psi_+ =1$ and $\psi_-= y_1\wedge\dots \wedge y_n$ are unaffected by orthonormal changes of adapted basis, since they are unaffected when viewed as anti-holomorphic forms. More generally, this argument shows that it is possible to choose expressions for the spinors in a consistent way whenever $G/H$ admits an invariant orthogonal almost complex structure or an invariant almost contact metric structure, which are precisely the cases $G=\SU(n+1)$, $\Sp(n)$, $\Sp(n)\U(1)$, and $\G_2$ in our classification.
\end{remark}
Next, we calculate the Ambrose-Singer torsion and determine its type:
\begin{proposition}
	For any $a,b>0$ the sphere $(S^{2n+1}=\SU(n+1)/\SU(n),g_{a,b})$ has Ambrose-Singer torsion of type $\mathcal{T}_{\totallyskew}\oplus \mathcal{T}_{\CT}$, given by
	\begin{align*}
	T^{\AS}(e_1,e_{2p}) &= -\sqrt{\frac{n+1}{an}} e_{2p+1}, \quad T^{\AS}(e_1,e_{2p+1}) =\sqrt{\frac{n+1}{an}}  e_{2p} , \\
	T^{\AS}(e_{2p},e_{2q}) &= T^{\AS}(e_{2p+1},e_{2q+1})=0, \quad 
	 T^{\AS} (e_{2p},e_{2q+1}) = \frac{-\delta_{p,q}\sqrt{a(n+1)}}{b\sqrt{n}} e_1,
	\end{align*}
	for all $p,q=1,\dots, n$. The projection of $T^{\AS}$ onto $\mathcal{T}_{\totallyskew}$ is
\[
T^{\AS}_{\totallyskew} := -\frac{(a+2b)\sqrt{n+1}}{3b\sqrt{an}}\sum_{p=1}^{n} e_1\wedge e_{2p}\wedge e_{2p+1},
\]
with $T^{\AS} = T^{\AS}_{\totallyskew}$ if and only if $a=b$.
\end{proposition}
Since $a,b>0$ (and hence $a+2b\neq 0$) we deduce from the preceding proposition that the Ambrose-Singer connection again never has torsion of pure cyclic traceless type. Next, we would like to differentiate the spinors $\psi_{\pm}$ from Theorem \ref{deformedSasakianinvariantspinors}. By calculating explicitly in the Lie algebra (or by other means), one finds:
\begin{lemma} The Nomizu map for the Levi-Civita connection on $(S^{2n+1}=\SU(n+1)/\SU(n),g_{a,b}) $ is given by
	\begin{align*}
		&\Uplambda^{g_{a,b}}(x_1)x_2 = 0, \quad
		\Uplambda^{g_{a,b}}(x)y= (1-\frac{a}{2b})[x,y]_{\mathfrak{m}},\\
		&\Uplambda^{g_{a,b}}(y)x = \frac{a}{2b}[y,x]_{\mathfrak{m}}, \quad
		\Uplambda^{g_{a,b}}(y_1)y_2= \frac{1}{2}[y_1,y_2]_{\mathfrak{m}},
	\end{align*} 
	for $x,x_1,x_2\in\mathfrak{m}_1$, $y,y_1,y_2\in\mathfrak{m}_2$.
\end{lemma}
Combining the preceding proposition and lemma then allows us to give an explicit expression for the Nomizu map of the Levi-Civita connection:
\begin{corollary}
	With respect to the above basis $\{e_i\}$ for $\mathfrak{m}$, the Nomizu map $\Uplambda^{g_{a,b}}$ takes the form
	\begin{align*}
		\Uplambda^{g_{a,b}}(e_1) &=(1-\frac{a}{2b}) \sqrt{\frac{n+1}{an}} \sum_{l=1}^n e_{2l}\wedge e_{2l+1} ,\quad
		\Uplambda^{g_{a,b}}(e_{2p})=-\frac{1}{2b}\sqrt{\frac{a(n+1)}{n}} e_1\wedge e_{2p+1} , \\
		\Uplambda^{g_{a,b}}(e_{2p+1})&=\frac{1}{2b}\sqrt{\frac{a(n+1)}{n}} e_1\wedge e_{2p} , \quad (p=1,\dots n).
	\end{align*} 
\end{corollary}
Lifting these to the spin bundle and applying them to $\psi_{\pm}$ gives,
\begin{theorem}\label{special_unitary_GKS_theorem}
	The invariant spinors $\psi_{\pm}$ are generalized Killing spinors, i.e. $\nabla^{g_{a,b}}_X\psi_{\pm}= A_{\pm}(X)\cdot \psi_{\pm}$, for the endomorphisms
	\begin{align*}
		A_+&:=  \lambda_1  \Id\rvert_{\mathfrak{m}_1} +  \lambda_2 \Id\rvert_{\mathfrak{m}_2}, \quad
		A_-:= (-1)^{n+1}A_+,
	\end{align*}
	where $\lambda_1:=\frac{(2b-a)\sqrt{n(n+1)}}{4b\sqrt{a}}  $, $\lambda_2:=\frac{\sqrt{a(n+1)}}{4b\sqrt{n}} $.
\end{theorem}
\begin{proof}
	The proof proceeds by direct calculation. As an example, we show that the desired equation holds for $\psi_+$ in the direction of $X=e_1$. Using the preceding corollary, we differentiate at the origin $o=eH$:
	\begin{align*}
		\widetilde{\Uplambda}^{g_{a,b}}(e_1)\cdot \psi_+ &= \frac{1}{2}(1-\frac{a}{2b})\sqrt{\frac{n+1}{an}}\sum_{l=1}^n e_{2l}\cdot e_{2l+1} \cdot \psi_+ = \frac{2b-a}{4b}\sqrt{\frac{n+1}{an}} \sum_{l=1}^n i(x_l\lrcorner + y_l\wedge)(y_l\wedge - x_l\lrcorner)1 \\
		&= \frac{2b-a}{4b}\sqrt{\frac{n+1}{an}} \sum_{l=1}^n i = \left( \frac{(2b-a)\sqrt{n(n+1)}}{4b\sqrt{a}}\right) i = \lambda_1 e_1\cdot \psi_+.
	\end{align*}
\end{proof}
\begin{corollary} \label{sunweakprop}
	The spinors $\psi_{\pm}$ are Killing spinors if and only if $a=\frac{2bn}{n+1}$, leading to $\lambda_1=\lambda_2 = \frac{1}{2\sqrt{2b}}$. The round metric corresponds to the parameters  $a=\frac{n}{n+1}$, $b= \frac{1}{2} $, in which case we recover the usual Sasakian Killing spinors for the constants $\frac{1}{2}$, $\frac{-1}{2}$ (or $\frac{1}{2}$, $\frac{1}{2}$, depending on $n$). 
\end{corollary}
Generalizing the usual Sasakian structure, we have:
\begin{proposition}\label{specialunitaryalphaSasakian}
	The sphere $(S^{2n+1}=\frac{\SU(n+1) }{\SU(n) },  g_{a,b})$ admits:
	\begin{enumerate}[(i)] \item a compatible invariant normal almost contact metric structure for all $a,b >0$.
		\item a compatible invariant $\alpha$-contact structure if and only if $\alpha = \frac{\sqrt{a(n+1)}}{2b\sqrt{n}}$.
		\item a compatible invariant $\alpha$-K-contact structure if and only if $\alpha = \frac{\sqrt{a(n+1)}}{2b\sqrt{n}}$.
	\end{enumerate}
	In particular there exists a compatible invariant $\alpha$-Sasakian structure if and only if $\alpha = \frac{\sqrt{a(n+1)}}{2b\sqrt{n}}$.
\end{proposition}
\begin{proof}
	In order for the structure to be invariant, the only choices for the Reeb vector field are $\xi=\pm e_1$. We note that the 2-form $\Phi:=g_{a,b}(\cdot, \varphi(\cdot))$ is invariant if and only if
 \[
 \Phi \in (\Lambda^2 \mathfrak{m}_2)^{\SU(n)} \simeq \Span_{\R}\{ \ad\xi\rvert_{\mathfrak{m}_2} \},
 \] and the metric compatibility condition $g_{a,b}(\varphi(X),\varphi(Y)) = g_{a,b}(X,Y) -g_{a,b}(\xi,X)g_{a,b}(\xi,Y)$ is satisfied if and only if
 \[
 \varphi = \sqrt{\frac{an}{n+1}} \ad\xi.
 \]
A tedious but straightforward Lie algebra computation then shows that the Nijenhuis tensor vanishes for any values of $a,b$, and the structure is $\alpha$-contact ($d\eta = 2\alpha \Phi$) and $\alpha$-K-contact ($\nabla^g_X\xi = -\alpha\varphi(X)$) if and only if $\alpha = \frac{\sqrt{a(n+1)}}{2b\sqrt{n}}$. 
\end{proof}
\begin{remark}
	For the parameters $a=\frac{-n\epsilon}{n+1}$, $b=\frac{1}{2}$ one has the Berger metrics $g_{\epsilon}$ described in \cite{draper_spheres}, with $\epsilon=-1$ corresponding to the round metric (see also Section 6). We would like to see what spinorial equations are satisfied by our spinors for the invariant connections constructed in \cite{draper_spheres}. In order to deal only with the Riemannian case, we will require $\epsilon<0$. Let us focus on dimensions not equal to 5, 7 ($n\neq 2,3$), in which case there is a 1-parameter family of invariant connections with skew torsion,
	\begin{align}\label{invconndimneq23}
		\nabla^s = \nabla^{g_{\epsilon}} -\epsilon s \ \Phi\wedge \eta,\quad s\in\R,
	\end{align}
	with torsion $T^s= -2\epsilon s \ \Phi\wedge \eta$, where $\Phi$ is the invariant 2-form defined in Section 2.2 of \cite{draper_spheres} and $\eta $ is the metric dual of $\xi:= e_1$. 
\end{remark}
Generalizing Theorem \ref{special_unitary_GKS_theorem}, we have
\begin{proposition}
	For $n\neq 2,3$ the invariant spinors $\psi_{\pm}$ on $( S^{2n+1}=\frac{\SU(n+1)}{\SU(n)},g_{\epsilon})$ satisfy the generalized Killing equation with torsion,
	\[
	\nabla^s_X\psi_+ = A^s_+(X)\cdot \psi_+, \quad \nabla^s_X\psi_- = A_-^s(X)\cdot \psi_-,
	\]
	for the endomorphisms
	\begin{align*} 
		A^s_+ &:= A_+ -\frac{\epsilon s n}{2} \Id\rvert_{\mathfrak{m}_1}+\frac{\epsilon s}{2}\Id\rvert_{\mathfrak{m}_2} ,  \\
		A^s_- &:= A_-  - \frac{(-1)^{n+1} \epsilon s n}{2} \Id\rvert_{\mathfrak{m}_1} + \frac{(-1)^{n+1} \epsilon s}{2} \Id\rvert_{\mathfrak{m}_2}.
	\end{align*}
\end{proposition}
\begin{proof}
Suppose that $a=\frac{-n\epsilon}{n+1}$, $b=\frac{1}{2}$. With respect to our chosen orthonormal basis $\{e_i\}_{i=1}^{2n+1}$, the invariant 2-form $\Phi$ takes the form  $$\Phi = -\sum_{p=1}^{n} e_{2p}\wedge e_{2p+1}  = - \sum_{p=1}^n e_{2p}\wedge \varphi(e_{2p})  , $$ where $\varphi = \frac{n\sqrt{-\epsilon}}{n+1}\ad(e_1)$. One easily calculates, 
\begin{align*} \Phi\cdot \psi_{+} &= n\xi \cdot \psi_+ , \quad \Phi\cdot \psi_{-} = (-1)^{n+1}n \xi\cdot \psi_{-},\quad \xi\cdot e_{2p}\cdot \psi_+  =  e_{2p+1}\cdot \psi_+   , \\
\xi\cdot e_{2p}\cdot \psi_- &= (-1)^{n+1} e_{2p+1} \cdot \psi_-, \quad \xi \cdot e_{2p+1} \cdot \psi_+ = -e_{2p}\cdot \psi_+, \quad
\xi\cdot e_{2p+1} \cdot \psi_- = (-1)^n e_{2p}\cdot \psi_-.
\end{align*}
We now consider all possible cases:
\begin{enumerate}
\item If $Z=\xi$ then $Z\lrcorner T^s  =-2\epsilon s \ \Phi  $, and we have
\begin{align*}
\nabla_{\xi}^s \psi_+ &= \nabla^{g_{\epsilon}}_{\xi}  \psi_+ - \frac{1}{2}\epsilon s \Phi\cdot \psi_+  = A_+(\xi) \cdot \psi_+ -\frac{1}{2}\epsilon s n \xi\cdot \psi_+ , \\
\nabla_{\xi}^s \psi_- &=  \nabla^{g_{\epsilon}}_{\xi}  \psi_- - \frac{1}{2}\epsilon s \Phi\cdot \psi_-  = A_-(\xi) \cdot \psi_- -\frac{1}{2}\epsilon s n(-1)^{n+1} \xi\cdot \psi_- .
\end{align*}
\item If $Z= e_{2p}$ or $Z=e_{2p+1}$ then $Z\lrcorner T^s = 2\epsilon s \ \varphi(Z)\wedge \eta = -2\epsilon s \ \eta\wedge \varphi(Z)$, and we have
\begin{align*}
\nabla^s_{e_{2p}} \psi_+ &=  \nabla^{g_{\epsilon}}_{e_{2p}}  \psi_+ - \frac{1}{2}\epsilon s \xi\cdot e_{2p+1} \cdot \psi_+ = A_+(e_{2p})\cdot \psi_+ +\frac{1}{2}\epsilon s  e_{2p}\cdot \psi_+    ,   \\
\nabla^s_{e_{2p}}\psi_- &= \nabla_{e_{2p}}^{g_{\epsilon}} \psi_- -\frac{1}{2} \epsilon s \xi \cdot e_{2p+1} \cdot \psi_- = A_-(e_{2p}) \cdot \psi_- +\frac{1}{2}\epsilon s (-1)^{n+1} e_{2p}\cdot \psi_-   , \\
\nabla_{e_{2p+1}}^s \psi_+ &= \nabla_{e_{2p+1}}^{g_{\epsilon}} \psi_+ +\frac{1}{2}\epsilon s \xi \cdot e_{2p} \cdot \psi_+ = A_+(e_{2p+1}) \cdot \psi_+ +\frac{1}{2}\epsilon s e_{2p+1} \cdot \psi_+ , \\
\nabla_{e_{2p+1}}^s \psi_- &= \nabla_{e_{2p+1}}^{g_{\epsilon}} \psi_- +\frac{1}{2} \epsilon s \xi\cdot e_{2p} \cdot \psi_- = A_- (e_{2p+1}) \cdot \psi_- +\frac{1}{2}\epsilon s (-1)^{n+1} e_{2p+1}\cdot \psi_-.  
\end{align*}
\end{enumerate}
\end{proof}
\begin{remark}
	For $n=2,3$ the families of invariant metric connections with skew torsion are larger, and depend on certain special tensors available in these dimensions \cite{draper_spheres}. We omit these cases here in the interest of brevity.
\end{remark}
\section{Classical Spheres, Part II: Spheres over $\H$}
This section is devoted to finding invariant spinors on the quaternionic spheres $S^{4n-1} = \frac{\Sp(n)\cdot K}{\Sp(n-1)\cdot K}$, where $K= \{1\}$, $\U(1)$, or $\Sp(1)$. We begin with a general discussion of the isotropy representations of these spaces and the action on the spin representation, then provide a method of finding the invariant spinors for each case $K=\{1\}$, $\U(1)$, $\Sp(1)$ in a unified way. In subsequent subsections we discuss each case in detail, including constructions of the invariant spinors in terms of explicit bases, and a discussion of the relevant induced geometries and spinorial equations. 
\subsection{The General Case, $S^{4n-1} = \frac{\Sp(n)\cdot K}{\Sp(n-1)\cdot K}$}

\subsubsection{The isotropy representation}
For the action of $G:= \Sp(n)\cdot K$ on $\mathbb{H}^n$ via $(A,z)\cdot v = Avz^{-1}$, it is well-known that the isotropy group of $(1,0,\dots, 0)\in\mathbb{H}^n$ is $H:=\Sp(n-1)\cdot K$, where the inclusion $H\subseteq G$ is realized as in (\ref{inclusion_isotropy_productgroup}). Since the spinor module is complex, we may pass freely between the isotropy Lie algebra and its complexification. Next observe that, for any $K =\{1\}$, $\U(1)$, $\Sp(1)$, with Lie algebra $\mathfrak{k}$, the complexified isotropy representation of $\Sp(n-1) \cdot K$ is
$(\mathfrak{sp}_{\mathbb{C}}(2n) \oplus \mathfrak{k}_{\mathbb{C}}) / (\mathfrak{sp}_{\mathbb{C}}(2n-2) \oplus \mathfrak{k}_{\mathbb{C}}) \cong \mathfrak{sp}_{\mathbb{C}}(2n) / \mathfrak{sp}_{\mathbb{C}}(2n-2)$. As a vector space this does not depend on $K$; the action is merely restricted from the action of $\Sp(n) \cdot \Sp(1)$.
\begin{remark}\label{Uwe1}
	Note that we can also use the $H$-$E$ formalism  introduced by Salamon \cite{Sal82} where $H:=\mathbb{H}\cong\mathbb{C}^2$, and $E:=\mathbb{H}^n\cong\mathbb{C}^{2n}$ are the standard (faithful, self-dual) complex representations of $\Sp(1)$ and $\Sp(n)$ respectively.
\end{remark}
Let us review the structure of this representation. Recall that $\Sp( n)\cdot \Sp( 1)=\Sp( n)\times \Sp( 1)/\{\pm(Id,1)\}$, so the group is doubly covered by $\Sp( n)\times \Sp( 1)$. We denote the elements of $\Sp( n)\times \Sp( 1)$ by $(A,z)$ and the elements of $\Sp( n)\cdot \Sp( 1)$ by $[A,z]$. Given $[A,z]\in \Sp( n)\cdot \Sp( 1)$, we define $\mu_{[A,z]}:S^{4n-1}\longrightarrow S^{4n-1}$ via $\mu_{[A,z]}(v)=Avz^{-1}$, where $vz^{-1}=(v_0z^{-1},...,v_nz^{-1})$. Note that this shows why we take $\Sp( n)\cdot \Sp( 1)$ instead of $\Sp( n)\times \Sp( 1)$ to get an effective action. The action is moreover transitive, since the usual action of $\Sp(n)$ on $S^{4n-1}$ is transitive. The isotropy subgroup $H$ stabilizing the point $p=(1,0,...,0)\in S^{4n-1}$ can be described in the following way: if $[A,z]\in H$ then $Ap=(z,0,...,0)$. From there, it is straightforward to see that  
\[H=\{[\left(\begin{array}{c|c}
z  & 0 \\ \hline
0 & A
\end{array}\right),z]:A\in \Sp( n-1),z\in \Sp( 1)\}=\Sp( n-1)\cdot \Sp'( 1)\cong \Sp( n-1)\cdot \Sp(1),
\]
where $\Sp'(1)$ denotes the group $\{[\left(\begin{array}{c|c}
z  & 0 \\ \hline
0 & 0
\end{array}\right),z]:z\in \Sp( 1)\}\cong\Sp(1)$ diagonally embedded in $\Sp(n)\cdot \Sp( 1)$, and the inclusion $\iota:\Sp( n-1)\cdot \Sp'( 1)\longrightarrow \Sp( n)\cdot \Sp( 1)$ is given by
\begin{eqnarray}\label{inclusion_isotropy_productgroup}\iota([A,z])=[\left(\begin{array}{c|c}
z  & 0 \\ \hline
0 & A
\end{array}\right),z].
\end{eqnarray}
Since $\Sp(n)\times \Sp(1)$  is a double cover of $\Sp(n)\cdot \Sp(1)$, they have the same complexified Lie algebra, 
\[
\mathfrak{g}_{\mathbb{C}}\cong\mathfrak{sp}_{\mathbb{C}}(2n)\oplus\mathfrak{sp}_{\mathbb{C}}(2).
\]
It is well-known (see for example \cite{Arvan}, p32) that $\mathfrak{sp}_{\mathbb{C}}(2n)\simeq  S^2(\mathbb{C}^{2n})$ as $\mathfrak{g}_{\mathbb{C}}$-representations, where $S^2(\mathbb{C}^{k})$ denotes the second complex linear symmetric power of the (irreducible) standard representation of $\mathfrak{sp}_{\mathbb{C}}(k)$, hence 
\[
\mathfrak{g}_{\mathbb{C}}\cong S^2(\mathbb{C}^{2n})\oplus S^2(\mathbb{C}^{2}).
\]
Restricting to the isotropy Lie algebra $\mathfrak{h}_{\mathbb{C}}\cong \mathfrak{sp}_{\mathbb{C}}(2n-2)\oplus\mathfrak{sp}'_{\mathbb{C}}(2)$ we get by \eqref{inclusion_isotropy_productgroup}
$$\mathfrak{g}_{\mathbb{C}}\cong S^2(\mathbb{C}^{2n-2}\oplus\mathbb{C}^{2})\oplus S^2(\mathbb{C}^{2})\cong [S^2(\mathbb{C}^{2n-2})\oplus(\mathbb{C}^{2n-2}\otimes \mathbb{C}^{2})\oplus S^2(\mathbb{C}^{2})]\oplus S^2(\mathbb{C}^{2}),$$
with $\mathfrak{h}_{\mathbb{C}}\cong S^2(\mathbb{C}^{2n-2})\oplus S^2(\mathbb{C}^{2})$.
Note that here $\mathbb{C}^{2n-2}$ and $\mathbb{C}^2$ are the standard representations of $\mathfrak{sp}_{\mathbb{C}}(2n-2)$ and $\mathfrak{sp}'_{\mathbb{C}}(2)$ respectively, as $\mathfrak{h}_{\mathbb{C}}$-representations. The isotropy representation is therefore given by the quotient
$$\mathfrak{g}_{\mathbb{C}}/\mathfrak{h}_{\mathbb{C}}\cong
\frac{[S^2(\mathbb{C}^{2n-2})\oplus(\mathbb{C}^{2n-2}\otimes \mathbb{C}^{2})\oplus S^2(\mathbb{C}^{2})]\oplus S^2(\mathbb{C}^{2})}{S^2(\mathbb{C}^{2n-2})\oplus S^2(\mathbb{C}^{2})}\cong (\mathbb{C}^{2n-2}\otimes \mathbb{C}^{2})\oplus S^2(\mathbb{C}^{2}).$$
Finally, since $G$ is compact, we may give $G/H$ the Riemannian metric coming from the Killing form on $\mathfrak{g}$, which descends to
a metric on $\mathfrak{g}/\mathfrak{h}$. Here we have used the observation that, if the isotropy representation of $G/H$ doesn't have a pair of non-trivial isomorphic subrepresentations (as is the case for $K=\{1\}$, $\U(1)$, $\Sp(1)$), any invariant metric gives the same dimension for the space of invariant spinors.  Since $G$ has finite center, the complexified Killing form on $\mathfrak{g}_{\mathbb{C}}$ is the unique adjoint-invariant
bilinear form up to scaling. Another adjoint-invariant bilinear form is given in terms of the isomorphism $S^2(\mathbb{C}^{2n}) \cong \mathfrak{sp}_{\mathbb{C}}(2n)$ by
\[
(v_1 v_2, w_1 w_2) = \frac{1}{2} \bigl(\omega_{2n}(v_1,w_1)\omega_{2n}(v_2,w_2)+\omega_{2n}(v_1,w_2)\omega_{2n}(v_2,w_1) \bigr),
\]
where $\omega_{2k}$ ($k\geq 1$) is the symplectic form on $\C^{2k}$ and $v_1v_2$ (resp. $w_1w_2$) the symmetric products of $v_1,v_2\in \C^{2n}$ (resp. $w_1,w_2\in \C^{2n}$).
Thus, this must be a (nonzero) scalar multiple of the Killing form under the isomorphism. (Actually, the scalar multiple is determined by the choice of isomorphism $S^2(\mathbb{C}^{2n-2}) \cong \mathfrak{sp}_{\mathbb{C}}(2n-2)$, which is also only unique up to scaling.)  Using this bilinear form, we see that the isomorphism
\[
\mathfrak{g}_{\mathbb{C}}/\mathfrak{h}_{\mathbb{C}} \cong (\mathbb{C}^{2n-2} \otimes \mathbb{C}^2) \oplus S^2(\mathbb{C}^2)
\]
is compatible with the symmetric bilinear form, in the sense that under the above isomorphism, the invariant inner product is given by tensoring the standard symplectic forms together:{\small
\[(v_1 \otimes v_2 \oplus v_3 \cdot v_4, w_1 \otimes w_2 \oplus w_3 \cdot w_4) = \omega_{2n-2}(v_1, w_1) \omega_2(v_2, w_2) + \frac{1}{2} \bigl(\omega_2(v_3,w_3)\omega_2(v_4,w_4) + \omega_2(v_3, w_4)\omega_2(v_4,w_3)\bigr).
\]
}
Up to scaling, this is the Riemannian metric on $\mathfrak{g}/\mathfrak{h}$.
\begin{proposition} For any decomposition $\mathbb{C}^2 = \mathbb{C} \oplus \mathbb{C}'$ into lines, we obtain a decomposition of the isotropy representation $V$ into $V = L \oplus L' \oplus \mathbb{C}_0$, for
	\begin{equation}
	L := (\mathbb{C}^{2n-2} \otimes \mathbb{C}) \oplus (\mathbb{C} \otimes \mathbb{C}), \quad L' := (\mathbb{C}^{2n-2} \otimes \mathbb{C}') \oplus (\mathbb{C}' \otimes \mathbb{C}'),
	\quad \mathbb{C}_0 := \mathbb{C} \otimes \mathbb{C}'.
	\end{equation}
\end{proposition}
Here we used the well-known fact that $S^2(V\oplus W)=S^2V\oplus( V\otimes W)\oplus S^2W$.
\subsubsection{The general linear and orthogonal Lie algebras, and the Clifford algebra}\label{quaternionic_case_setup}
Next we want to compute the action of $\mathfrak{h}_{\mathbb{C}}$ on the spin representation.  In order to do this, we will exploit the fact that $\mathfrak{sp}_{\mathbb{C}}(2n-2)$ and $\mathfrak{gl}_{\mathbb{C}}(1)$ (the complexification of $\mathfrak{u}(1)$) are both Lie subalgebras of the general linear Lie algebra $\mathfrak{gl}(L)$, for $L$ a maximal isotropic subspace of the complexified tangent space $V$ of $G/H$. Before we do this, we recall how to embed the general linear Lie algebra inside the orthogonal Lie algebra, and the orthogonal Lie algebra inside the Clifford algebra. As before, let 
\[
V:=\mathbb{C}^k=\begin{cases} L\oplus L' & \textrm{ if } k \textrm{ is even},\\
L\oplus L'\oplus \mathbb{C}_0 & \textrm{ if }$k$ \textrm{ is odd},\end{cases}
\]
together with its invariant bilinear form $\beta$. We will focus on the odd case $V=L\oplus L'\oplus \mathbb{C}_0$, which is the only one of interest for us here, but the even case is exactly the same if one ignores the scalar part $\mathbb{C}_0$ in the computation.

We first recall that there is a canonical embedding $\iota \: \mathfrak{gl}(L) \hookrightarrow \mathfrak{so}(L\oplus L' \oplus \mathbb{C}_0)$ given by the formula:
\begin{equation}\label{e:iota}
\iota (\xi)(l,l',c) = (\xi(l),-\xi^*l' ,0), \qquad \xi\in \mathfrak{gl}(L), \  l\in L, \ l'\in L',
\end{equation}
where we use the isomorphism $L'\cong L^*$ and notation $-\xi^*\: l' \mapsto -l'\circ \xi$ for the representation of $\mathfrak{gl}(L)$ on $L'$ dual to the standard representation on $L$. To check that $\iota$ is well-defined we have to show that 
\[
\beta(\iota (\xi)(l_1,l'_1,c_1),(l_2,l'_2,c_2) ) + \beta\left((l_1,l'_1,c_1),\iota (\xi)(l_2,l'_2,c_2) \right)=0.
\]
This follows from the calculation
\begin{multline*}
	\beta\left((\xi(l_1), (-\xi)^*(l_1'),0), (l_2, l_2', c_2)\right) = \beta\left(\xi(l_1), l_2'\right)+ \beta\left((-\xi)^*(l_1'), l_2\right)  \\ = \beta\left( l_1, \xi^*(l_2')\right) + \beta\left( l_1', -\xi(l_2)\right) 
	= -\beta\left((l_1, l_1' ,c_1), (\xi(l_2), (-\xi^*)(l_2') , 0)\right).
\end{multline*}
Next, it is well known that there is a Lie algebra embedding $\mathfrak{so}(V) \hookrightarrow Cl(V)$. To write this down explicitly, we follow
\cite[p.\@ 303]{FultonHarris1991}: we first identify $\Lambda^2 V \cong \mathfrak{so}(V)$, via the map
\begin{equation}\label{e:wedge-so}
v_1 \wedge v_2 \mapsto \varphi_{v_1,v_2}, \quad \varphi_{v_1, v_2}(w) =   \beta (v_1, w)v_2 - \beta(v_2, w)v_1 .
\end{equation}
This endows  $\Lambda^2 V$ with a Lie bracket,
\begin{equation}\label{e:bracket-wedge}
[v_1 \wedge v_2, w_1 \wedge w_2] =  - \beta(v_2, w_1) v_1 \wedge w_2  + \beta(v_2, w_2) v_1 \wedge w_1  + \beta(v_1,w_1) v_2 \wedge w_2 - \beta(v_1,w_2) v_2 \wedge w_1 .
\end{equation}
We now consider the linear embedding $\Lambda^2 V \hookrightarrow Cl(V)$ given by the formula
\begin{equation}\label{e:wedge-cl}
v \wedge w \mapsto \frac{1}{4} (v \cdot w - w \cdot v),
\end{equation}
with the product $\cdot$ the Clifford multiplication. This map preserves the Lie bracket (which, on the Clifford algebra, is the Clifford commutator $v \cdot w - w \cdot v$). We conclude the following (cf.~\emph{op.~cit.}):
\begin{lemma}
	The composition $\mathfrak{so}(V) \cong \Lambda^2 V \hookrightarrow Cl(V)$ given by the inverse of \eqref{e:wedge-so} followed by \eqref{e:wedge-cl} is a Lie algebra embedding. 
\end{lemma}
Now, let $V = L \oplus L' \oplus \mathbb{C}_0$.   Using \eqref{e:iota}, \eqref{e:wedge-so}, and \eqref{e:wedge-cl}, we obtain by composition a Lie algebra embedding
\[
\mathfrak{gl}(L) \hookrightarrow \mathfrak{so}(V) \cong \Lambda^2 V \hookrightarrow Cl(V).
\]
Let us give an alternative expression for this composition. Consider the linear isomorphism
\begin{equation}
\beta^\flat: L'  \to L^*, \quad \beta^\flat(\ell')(\ell) = \beta(\ell', \ell).
\end{equation}
Again here we follow the convention in \cite{FultonHarris1991}.
Using this we obtain a composition
\begin{equation}\label{e:gl-cl}
L \otimes L' \iso L \otimes L^*  \iso \mathfrak{gl}(L) \hookrightarrow \mathfrak{so}(V) \cong \Lambda^2 V \hookrightarrow Cl(V).
\end{equation}
\begin{lemma}\label{l:ll'-image}
	The composition \eqref{e:gl-cl} equals the map
	\begin{equation}
	\ell \otimes \ell' \mapsto - \frac{1}{4} \bigl(\ell \cdot \ell' - \ell' \cdot \ell\bigr).
	\end{equation}
\end{lemma}
\begin{proof}
	The image of $\ell \otimes \ell'$ under the composition
	\[
	L \otimes L' \iso L \otimes L^*  \iso \mathfrak{gl}(L) \hookrightarrow \mathfrak{so}(V)
	\]
	is $ - \varphi_{\ell,\ell'}$. Then the image under \eqref{e:wedge-so} and \eqref{e:wedge-cl} is $- \frac{1}{4} (\ell \cdot \ell' - \ell' \cdot \ell)$.
\end{proof}
\subsubsection{The spin representation}
Next we pass to the spin representation, which is defined on the vector space $\Sigma:= \Lambda^{\bullet}  L'$.
\begin{proposition}\label{p:gl-action}
	The action of an element $\varphi \in \mathfrak{gl}(L)$ on $\Sigma $
	is given by the formula
	\begin{equation}
	\varphi \cdot (\ell_1' \wedge \cdots \wedge \ell_m')
	= - \sum_{i=1}^m (-1)^{i+1} (\varphi^* \ell_i') \wedge \ell_1' \wedge \cdots \wedge \widehat{\ell_i'} \wedge \cdots  \wedge \ell_m'
	+ \frac{1}{2} \tr(\varphi)  (\ell_1' \wedge \cdots \wedge \ell_m')
	\end{equation}
	for all $\ell_1', \ldots, \ell_m' \in L'$, where $\widehat{\ell_i'}$ indicates omission of that term.
\end{proposition}
\begin{proof}
	To prove this, it is enough to take $\varphi$ to be the image of $\ell \otimes \ell'$, for $\ell \in L, \ell' \in L'$, as these span all of $\mathfrak{gl}(L)$. Applying Lemma \ref{l:ll'-image},
	its image in $Cl(V)$ is 
	\[
	-\frac{1}{4}(\ell \cdot \ell' - \ell' \cdot \ell) =-\frac{1}{2}\ell \cdot \ell' - \frac{1}{2} \beta(\ell,\ell').
	\]
	Using the Clifford representation (\ref{clifford_rep_formulas}), and noting that $\beta(\ell,\ell')$ is the trace of the given element of $\mathfrak{gl}(L)$, we then calculate
	\begin{align*}
		-\frac{1}{2}(\ell \cdot \ell') \cdot (\ell_1' \wedge \cdots \wedge \ell_m') &= -\frac{1}{2}\left( i\sqrt{2} \ \ell \lrcorner (i\sqrt{2} \  \ell' \wedge \ell_1' \wedge \cdots \wedge \ell_m' )\right) =  \ell \lrcorner ( \  \ell' \wedge \ell_1' \wedge \cdots \wedge \ell_m' )  \\
		&= \beta(\ell,\ell') \ell_1' \wedge \cdots \wedge \ell_m' - \ell'\wedge \sum_{i=1}^m (-1)^{i+1}  \beta(\ell, \ell_i') \ell_1' \wedge \cdots \wedge  \widehat{\ell_i'} \wedge \cdots \wedge \ell_m' \\
		&= \tr(\varphi)(\ell_1' \wedge \cdots \wedge \ell_m') - \sum_{i=1}^m (-1)^{i+1} (\varphi^* \ell_i') \wedge \ell_1' \wedge \cdots \wedge \widehat{\ell_i'} \wedge \cdots  \wedge \ell_m',
	\end{align*}
	and the result follows.
\end{proof}
Note that $L$ is the standard representation of $\mathfrak{gl}(L)$, and every exterior power $\Lambda^i L' $ is an irreducible representation of it. Since $\mathbb{C}_0$ carries the trivial action of $\mathfrak{gl}(L)$, we obtain:
\begin{corollary}
	The restriction of the spin representation to $\mathfrak{gl}(L)$ decomposes as a direct sum of irreducible representatations,
	\begin{equation}
	\Sigma\rvert_{\mathfrak{gl}(L)}\simeq \bigoplus_{k=0}^{\dim L'} \Lambda^k L'.
	\end{equation}
\end{corollary}
\subsubsection{Case I: $K=\{1\}$}\label{caseIk=1}
We are now ready to analyse the isotropy Lie algebra action on the spin representation  in the case $K=\{1\}$, i.e. the isotropy algebra is $ \mathfrak{sp}(2n-2)$. Note that $\mathfrak{sp}(2n-2)$ acts on $V$ in such a way that
$L' = (\mathbb{C}^{2n-2} \otimes \mathbb{C}') \oplus  (\mathbb{C}' \otimes \mathbb{C}')$ is a subrepresentation, as is $L$.  Moreover, it acts trivially on $\mathbb{C}_0 =\mathbb{C} \otimes \mathbb{C}'$.  Therefore, $\mathfrak{sp}(2n-2) \to \mathfrak{so}(V)$ lands in the subalgebra $\iota(\mathfrak{gl}(L))$.  Next, as a representation of $\mathfrak{sp}(2n-2)$, $L'$ is a direct sum of the dual $\mathbb{C}^{2n-2} \otimes \mathbb{C}'$ of the standard representation, and the trivial representation $\mathbb{C}'\otimes \mathbb{C}'$.  Finally, recall that $\mathfrak{sp}(2n-2) \subseteq \mathfrak{sl}(2n-2)$ consists of elements of trace zero.
Thus, Proposition \ref{p:gl-action} reduces in this case to the statement: the spin representation of $\mathfrak{sp}_{\mathbb{C}}(2n-2)$ is isomorphic to the direct
sum of representations
\begin{equation}
\Sigma\rvert_{\mathfrak{sp}_{\C}(2n-2)} \simeq  \bigoplus_{k=0}^{\dim L'} \Lambda^k (\mathbb{C}^{2n-2} \otimes \mathbb{C}')  \oplus \bigoplus_{k=0}^{\dim L'} \Lambda^k (\mathbb{C}^{2n-2} \otimes \mathbb{C}') \otimes (\mathbb{C}' \otimes \mathbb{C}').
\end{equation}
By the first fundamental theorem of invariant theory (Proposition F.13 on pg 510 of \cite{FultonHarris1991}), we have
\begin{equation}
(\Lambda^k \mathbb{C}^{2n-2})^{\mathfrak{sp}(2n-2)} \cong \begin{cases} \omega^{k/2}, & \text{for $k$ even}, \\
0, & \text{for $k$ odd}.
\end{cases}
\end{equation}
We therefore find that the invariant elements of the spin representation are
\begin{equation} \label{e:sp-invts}
\Sigma_{\text{inv}} = \bigoplus_{k=0}^{n-1} \mathbb{C} \cdot (\omega^k  \otimes 1) \oplus \bigoplus_{k=0}^{n-1} \mathbb{C} (\omega^k \otimes 1) \otimes (1 \otimes 1),
\end{equation}
spanning a space of complex dimension $2n$.
\subsubsection{Case II: $K=\U(1)$}\label{caseIIK=U1}
We have a unique decomposition $\mathbb{C}^2 = \mathbb{C} \oplus \mathbb{C}'$ into one-dimensional (nontrivial) irreducible representations of $\mathfrak{gl}_{\mathbb{C}}(1)\cong \C\otimes \mathfrak{u}(1)$.  In more detail, $\mathbb{C}^2$ is the standard representation of $\mathfrak{sp}_{\mathbb{C}}(2) \cong \mathfrak{sl}_{\mathbb{C}}(2)$, and we may realize $\mathfrak{gl}_{\mathbb{C}}(1) \subseteq \mathfrak{sl}_{\mathbb{C}}(2)$ as the subset of diagonal two-by-two matrices of trace zero:
\begin{equation} \label{e:gl1-2d}
\mathfrak{gl}_{\mathbb{C}}(1)  = \left\{ \tau_{\lambda}:=\begin{pmatrix} \lambda & 0 \\ 0 & -\lambda \end{pmatrix} \: \lambda \in \mathbb{C}\right\}.
\end{equation}
Then the standard representation $\mathbb{C}^2$ decomposes uniquely into the coordinate axes in $\mathbb{C}^2$, which are the two eigenspaces of every nonzero element of $\mathfrak{gl}_{\mathbb{C}}(1)$. With this caveat in place, the preceding decomposition goes through again.

The action of $\mathfrak{gl}_{\mathbb{C}}(1)$ on \eqref{e:sp-invts} can be realised as follows: first, by \eqref{e:gl1-2d}, it acts on the first coordinate axis, $\mathbb{C}$, by the standard one-dimensional representation. Next, note that $\mathfrak{gl}_{\mathbb{C}}(1) \not \subseteq \mathfrak{sl}(2n-2)$.  Indeed, $\tau_{\lambda} \in \mathfrak{gl}_{\mathbb{C}}(1)$ acts on $L=(\mathbb{C}^{2n-2} \otimes \mathbb{C}) \oplus( \mathbb{C} \otimes \mathbb{C})$ as the direct sum  $\lambda \ Id_{2n-2} \oplus 2\lambda \ Id_1$, of trace $(2n) \lambda$.
So to compute its action on $\Lambda^k(\mathbb{C}^{2n-2} \otimes \mathbb{C}')$, we follow Proposition \ref{p:gl-action}.  This yields the action, for $\tau_{\lambda} \in \mathfrak{gl}_{\mathbb{C}}(1)$,
\begin{equation}
\tau_{\lambda}\rvert_{\Lambda^k(\mathbb{C}^{2n-2} \otimes \mathbb{C}')} = (n-k)\lambda \ Id.
\end{equation}
We also note that the action on $\mathbb{C}'\otimes \mathbb{C}'$ is 
\[
\tau_{\lambda}\rvert_{\mathbb{C}'\otimes \mathbb{C}'} = (-2\lambda) \ Id_1.
\]
Since $\omega \in \Lambda^2 L'$, we get
\begin{equation}\label{e:gl1-action-invts}
\tau_{\lambda} \cdot (\omega^k \otimes 1) = \lambda(n-2k) \  (\omega^k \otimes 1), \quad \tau_{\lambda} \cdot ((\omega^k \otimes 1) \otimes (1 \otimes 1))  = \lambda(n-2k-2) \
((\omega^k \otimes 1) \otimes (1 \otimes 1)).
\end{equation}
The action is zero when $n-2k = 0$ and $n-2k-2=0$, respectively. This can happen if and only if $n$ is even, in which case the invariant elements are
\begin{align}\label{invariantelementsu1twistedMA}
	\Sigma_{\inv} = \text{span}_{\mathbb{C}} \{\omega^{n/2} \otimes 1, \ (\omega^{(n - 2)/2} \otimes 1) \otimes (1 \otimes 1)\}.
\end{align}
\subsubsection{Case III: $K= \Sp(1)$}\label{caseIIIK=Sp(1)}
It remains to compute the
action of $\mathfrak{sp}_{\mathbb{C}}(2) \supseteq \mathfrak{gl}_{\mathbb{C}}(1)$ on $(\Lambda^{\bullet}  L')^{\mathfrak{sp}_{\mathbb{C}}(2n-2)}$.
In this case, $\mathfrak{h}_{\mathbb{C}}$ is no longer in $\mathfrak{gl}(L)$ at all, so we cannot apply Proposition \ref{p:gl-action}.  However, by the above, we may view $\mathfrak{gl}_{\mathbb{C}}(1) \subseteq \mathfrak{sp}_{\mathbb{C}}(2)\cong  \mathfrak{sl}_{\mathbb{C}}(2)$ as the inclusion $\mathbb{C} \cdot \tau_1 \subseteq \mathfrak{sl}_{\mathbb{C}}(2)$, where $\tau_1= \begin{pmatrix} 1 & 0 \\ 0 & -1 \end{pmatrix}$ is the standard element.  Every finite-dimensional representation of $\mathfrak{sl}_{\mathbb{C}}(2)$ is a direct sum of irreducible representations, where the eigenvalues of $\tau_1$ are of the form $-m, 2-m, \ldots, m-2, m$.  The eigenvalues of $h$ on a representation uniquely determine this decomposition up to isomorphism.  Looking at \eqref{e:gl1-action-invts}, we see that the highest $\tau_1$-eigenvalues are $n$ and $n-2$, corresponding to $k=0$.  This means that the $\mathfrak{sp}_{\mathbb{C}}(2)$-action on \eqref{e:sp-invts} must decompose as a direct sum of the irreducible representations of dimension $n+1$ and $n-1 $ (of highest weight $n$ and $n-2$, respectively).  The trivial representation is the unique irreducible representation of dimension one, and this can occur if and only if $n=2$. This is the case of $\Sp(2) \cdot \Sp(1)$ acting on $S^7$, in which case we get a unique $\mathfrak{h}_{\mathbb{C}}$-invariant element of the spin representation up to scaling. For $n\geq 2$ we do not get any $\mathfrak{h}_{\mathbb{C}}$-invariant elements.  We summarize this as follows:
\begin{proposition}\label{MAcaseK=sp1}
	Let $\Sigma\rvert_{\mathfrak{h}_{\C}}$ be the restriction of the spin representation to $\mathfrak{h}_{\mathbb{C}}$.  Then, as a representation of $\mathfrak{sp}_{\mathbb{C}}(2)$, $(\Sigma\rvert_{\mathfrak{h}_{\C}})^{\mathfrak{sp}_{\mathbb{C}}(2n-2)}$ is isomorphic to a direct sum of irreducible representations of dimensions $n+1$ and $n-1$. There are no $\Sp(n-1)\cdot \Sp(1)$-invariant spinors unless $n=2$, in which case there is a unique invariant spinor, up to scaling, and it lies in the span of $ (\omega\otimes 1)$ and $(1\otimes 1)\otimes (1\otimes 1)$.
\end{proposition}
Note that the preceding proposition for $n>2$ parallels the fact that quaternionic K\"{a}hler manifolds don't admit parallel spinors in dimensions larger than $4$ (see \cite{Wang}). For $n=2$ the qualitative situation differs from Wang's result; in dimension $4$, the only compact Riemannian manifolds admitting parallel spinors are Calabi-Yau manifolds, for which the space of parallel spinors is $2$-dimensional, whereas in the preceding proposition the dimension of the space of invariant spinors is $1$. 

Since the isotropy representation in each case $K=\{1\}$, $\U(1)$, $\Sp(1)$ may be obtained by restriction from the case $K=\Sp(1)$, another consequence of Proposition \ref{MAcaseK=sp1} is the immediate corollary:
\begin{corollary} \label{K_restriction} For any $K ={1}$, $\U(1)$, $\Sp(1)$, the  $\mathfrak{sp}_{\mathbb{C}}(2n-2) \oplus \mathfrak{k}_{\mathbb{C}}$-invariant spinors are identified with the $\mathfrak{k}_{\mathbb{C}}$-invariant vectors in the direct sum of irreducible representations of 
	$\mathfrak{sp}_{\mathbb{C}}(2) \cong \mathfrak{sl}_{\mathbb{C}}(2)$
	of dimensions $n+1$ and $n-1$.
\end{corollary}
\begin{remark} This can be made quite explicit: consider $\Sp(2,\C) = \operatorname{SU}(2,\mathbb{C})$ to be a group of two-by-two complex matrices acting on the vector space $\mathbb{C} x \oplus \mathbb{C} y$.  Then it also acts on the vector space $\mathbb{C}[x,y]_m$ of homogeneous polynomials of degree $m$ in $x,y$.  This representation is irreducible of dimension $m+1$.  So one may take the invariants $\mathbb{C}[x,y]_m^K$ under any subgroup $K < \Sp(1)$, as well as the invariants $\mathbb{C}[x,y]_m^{\mathfrak{k}_{\mathbb{C}}}$ under its complexified Lie algebra (the former is contained in the latter, and the two are equal if $K$ is connected).
\end{remark}
\begin{remark}
Finally, we note that the above results are independent of the choice of invariant Riemannian metric. To see this it suffices to note that, for each of the three cases, any decomposition of the isotropy representation into irreducible submodules doesn't contain pairwise isomorphic nontrivial summands.
\end{remark}
In what follows we examine each case in more detail and give explicit constructions of the invariant spinors.
\subsection{Standard Quaternionic Spheres, $S^{4n-1}= \Sp(n)/\Sp(n-1)$}
\label{deformed3Sas}
This is the case from Section \ref{caseIk=1} corresponding to $K=\{1\}$.
The organization of this section is as follows. To begin, we find the invariant spinors and Ambrose-Singer torsion for general invariant metrics, then subsequently specialize to the $\tad$ setting, where we discuss the relationship between the invariant spinors and the $\tad$ structure. Particular attention is paid to the $7$-dimensional example $S^7=\Sp(2)/\Sp(1)$, where we compare our results with the known results from \cite{nearly_parallel_g2,3Sasdim7,3str,AHduality}. We conclude by returning to the setting of a general invariant metric and giving examples of generalized Killing spinors with $4$ eigenvalues in dimension $7$.

To describe the invariant metrics on $S^{4n-1}=\Sp(n)/\Sp(n-1)$, we first recall that the isotropy representation splits into three copies of the trivial representation and one copy of the standard representation of $\Sp(n-1)$,
\[
\mathfrak{m}\simeq \bigoplus_{i=1}^4\mathfrak{m}_i, \quad \text{ where } \mathfrak{m}_i\simeq \R \  (i=1,2,3), \  \mathfrak{m}_4\simeq \R^{4(n-1)}.
\]Up to isometry, there is a 4-parameter family of invariant metrics (see \cite{Ziller_homogeneous_einsten_metrics}), and these are given by rescaling $B_0$ separately on the isotropy components:
\begin{align} 
	g_{\vec{a}}&:=  a_1 B_0 \rvert_{\mathfrak{m}_1\times \mathfrak{m}_1} + a_2 B_0 \rvert_{\mathfrak{m}_2\times \mathfrak{m}_2} + a_3 B_0 \rvert_{\mathfrak{m}_3\times \mathfrak{m}_3} + a_4 B_0\rvert_{\mathfrak{m}_4\times \mathfrak{m}_4}, \qquad a_1,a_2,a_3,a_4>0. \label{diagonalmetric} 
\end{align}

Next, we explicitly describe the reductive complement $\mathfrak{m}$. From (\ref{inclusion_isotropy_productgroup}), the embedding $\Sp(n-1)\hookrightarrow \Sp(n)$ may be realized as the lower right hand $(n-1)\times (n-1)$ block, and we take $\mathfrak{m}:=\mathfrak{sp}(n-1)^{\perp}$ with respect to the Killing form, $\kappa_{\mathfrak{sp}(n)}:= -4(n+1) B_0$. We then have, at the level of Lie algebras,
\begin{align*}
\mathfrak{sp}(n)&= \Span_{\R} \{iF_{p,p}^{(n)},jF_{p,p}^{(n)}, kF_{p,p}^{(n)}, iF_{r,s}^{(n)},jF_{r,s}^{(n)}, kF_{r,s}^{(n)}, E_{r,s}^{(n)}  \}_{\substack{ p=1,\dots, n, \\ 1\leq r<s\leq n  }} \\
\mathfrak{sp}(n-1)&= \Span_{\R} \{iF_{p,p}^{(n)},jF_{p,p}^{(n)}, kF_{p,p}^{(n)}, iF_{r,s}^{(n)},jF_{r,s}^{(n)}, kF_{r,s}^{(n)}, E_{r,s}^{(n)}  \}_{\substack{ p=2,\dots, n, \\ 2\leq r<s\leq n  }} 
\end{align*}
and
\begin{align*}
\mathfrak{m}&= \Span_{\R}\{ iF_{1,1}^{(n)}, jF_{1,1}^{(n)}, kF_{1,1}^{(n)}, iF_{1,p}^{(n)}, jF_{1,p}^{(n)}, kF_{1,p}^{(n)}, E_{1,p}^{(n)}   \}_{p=2,\dots, n}.
\end{align*}
Using this description, the isotropy components are
\begin{align*}
	\mathfrak{m}_1= \R iF_{1,1}^{(n)} , \quad \mathfrak{m}_2=  \R jF_{1,1}^{(n)} , \quad 
	\mathfrak{m}_3= \R  kF_{1,1}^{(n)}  , \quad 
	\mathfrak{m}_4 = \Span_{\R}\{ iF_{1,p}^{(n)}, jF_{1,p}^{(n)}, kF_{1,p}^{(n)}, E_{1,p}^{(n)}   \}_{p=2,\dots, n},
\end{align*}
and a $g_{\vec{a}}$-orthonormal basis for $\mathfrak{m}$ is given by 
\begin{align}
	e_1&:= \frac{1}{\sqrt{a_1}} \  iF_{1,1}^{(n)}, \quad 
	e_2:= \frac{-1}{\sqrt{a_2}}\  kF_{1,1}^{(n)}, \quad 
	e_3:= \frac{1}{\sqrt{a_3}}\  jF_{1,1}^{(n)}, \quad 
	e_{4p}:= \frac{1}{\sqrt{2a_4}}\ j F_{1,p+1}^{(n)}, \label{spnONBline1}\\
	e_{4p+1}&:= \frac{1}{\sqrt{2a_4}}\ k F_{1,p+1}^{(n)} , \quad  e_{4p+2}:= \frac{1}{\sqrt{2a_4}}\ i F_{1,p+1}^{(n)}, \quad e_{4p+3}:= \frac{1}{\sqrt{2a_4}} \ E_{1,p+1}^{(n)}, \label{spnONBline2}
\end{align}
for $ p=1,\dots,n-1$. From equation (\ref{e:sp-invts}), we immediately obtain:
\begin{theorem}\label{invspinorssp(n)sp(n-1)}
	Using the above orthonormal basis and the corresponding realization of the spinor module from Remark \ref{reorderedCliffalgrepresentation}, the space of invariant spinors on $(S^{4n-1}=\Sp(n)/\Sp(n-1), g_{\vec{a}})$ for any $a_1,a_2,a_3,a_4>0$ is given by
\[
\Sigma_{\inv} = \Span_{\C} \{ \omega^j, \  y_1\wedge \omega^j \}_{j=0}^{n-1},
\]
where $\omega:=\sum_{i=1}^{n-1} y_{2i}\wedge y_{2i+1}$.
\end{theorem}
\begin{remark}
	The fact that the space of invariant spinors is $2n$-dimensional may also be deduced from part (b) of the main proposition in \cite{Wang} by noting that the spinor module in dimension $4n-1$ is the tensor product of $\C^2$ with the spinor module in dimension $4(n-1)$, and the isotropy representation acts trivially on the span of $e_1,e_2,e_3$. Explicitly, the proposition in \cite{Wang} gives $n$ linearly independent $\Sp(n-1)$-stabilized spinors in $\Sigma_{4n-4}$, and one then takes the tensor products of these with a basis of $\C^2$ to obtain $2n$ invariant spinors in $\Sigma_{4n-1}$. Our approach has the added benefit of providing an explicit description of the spinors, and allowing us to treat the cases $G=\Sp(n)$, $\Sp(n)\Sp(1)$, and $\Sp(n)\U(1)$ in a unified way.
\end{remark}
Before discussing the $\tad$ case in more detail, we first calculate the Ambrose-Singer torsion in the general case and determine its type:
\begin{proposition}\label{spnASTORSION}
	For any $a_1,a_2,a_3,a_4>0$ the sphere $(S^{4n-1}=\Sp(n)/\Sp(n-1),g_{\vec{a}} )$ has Ambrose-Singer torsion of type $\mathcal{T}_{\totallyskew}\oplus \mathcal{T}_{\CT}$, given by {\small
		\begin{align*}
			T^{\AS}(e_1,e_2)&= \frac{-2\sqrt{a_3}}{\sqrt{a_1a_2}} e_3, \quad T^{\AS}(e_1,e_3) = \frac{2\sqrt{a_2}}{\sqrt{a_1a_3}} e_2, \quad  T^{\AS}(e_2,e_3) = \frac{-2\sqrt{a_1}}{\sqrt{a_2a_3}} e_1 ,  \\
			T^{\AS}(e_1, - )\rvert_{\mathfrak{m}_4} &= \frac{1}{\sqrt{a_1}}\Phi_1\rvert_{\mathfrak{m}_4}, \quad T^{\AS}(e_2, - )\rvert_{\mathfrak{m}_4} = \frac{1}{\sqrt{a_2}}\Phi_2\rvert_{\mathfrak{m}_4}, \quad  T^{\AS}(e_3, - )\rvert_{\mathfrak{m}_4} = \frac{1}{\sqrt{a_3}}\Phi_3\rvert_{\mathfrak{m}_4},        \\
			T^{\AS}(e_{4p},e_{4q}) &= T^{\AS}(e_{4p+1},e_{4q+1}) = T^{\AS}(e_{4p+2},e_{4q+2}) = T^{\AS}(e_{4p+3},e_{4q+3})=0\\
			T^{\AS} (e_{4p},e_{4q+1}) &= \frac{-\delta_{p,q}\sqrt{a_1}}{a_4} e_1 , \quad T^{\AS}(e_{4p},e_{4q+2}) = \frac{-\delta_{p,q} \sqrt{a_2}}{a_4} e_2, \quad T^{\AS}(e_{4p},e_{4q+3}) = \frac{-\delta_{p,q} \sqrt{a_3} }{a_4} e_3, \\
			T^{\AS}(e_{4p+1},e_{4q+2}) &= \frac{-\delta_{p,q}\sqrt{a_3} }{a_4} e_3,\quad  T^{\AS} (e_{4p+1},e_{4q+3}) = \frac{\delta_{p,q}\sqrt{a_2}}{a_4} e_2, \quad T^{\AS}(e_{4p+2},e_{4q+3}) = \frac{-\delta_{p,q}\sqrt{a_1}}{a_4}e_1,  
		\end{align*}
	}
	for $p,q=1,\dots, n-1$, where $\Phi_1,\Phi_2,\Phi_3$ are defined formally as in (\ref{Phi1})-(\ref{Phi3}). The projection of $T^{\AS}$ onto $\mathcal{T}_{\totallyskew}$ is 
	\begin{align}\label{diagonalmetricsAStorsionprojection}
		T^{\AS}_{\totallyskew} = -\frac{2}{3}\left( \frac{a_1+a_2+a_3}{\sqrt{a_1a_2a_3}} \right)e_1\wedge e_2\wedge e_3    +\frac{1}{3}\sum_{i=1}^3 \left( \frac{a_i+2a_4}{a_4\sqrt{a_i}} \right) e_i \wedge \Phi_i\rvert_{\mathfrak{m}_4}   ,
	\end{align}
	with $T^{\AS} = T^{\AS}_{\totallyskew}$ if and only if $a_1=a_2=a_3=a_4$.
\end{proposition}
\subsubsection{Spinors on $\tad$ spheres}
Among the metrics (\ref{diagonalmetric}), we consider in this subsection the distinguished subfamily of $\tad$ metrics $g_{\alpha,\delta}$. Following \cite{hom3alphadelta}, and noting that the Killing form on $\mathfrak{sp}(n)$ is $\kappa_{\mathfrak{sp}(n)}:= -4(n+1)B_0$, the $\tad$ structure tensors are defined by
\begin{align}
\xi_1&:= i\delta F_{1,1},\quad
\xi_2:=-k\delta F_{1,1},\quad 
\xi_3:= j\delta F_{1,1}, \label{Reeb_vec_fields_definition_SpnSpnminus1} \\ \label{tadstructuretensorsdefinition}
g_{\alpha,\delta}&:= \frac{1}{\delta^2} B_0\rvert_{\mathcal{V}\times \mathcal{V}}  + \frac{1}{2\alpha\delta} B_0\rvert_{\mathcal{H}\times \mathcal{H}} , \quad 
\varphi_p := \frac{1}{2\delta} \ad(\xi_p)\rvert_{\mathcal{V}} + \frac{1}{\delta}\ad(\xi_p)\rvert_{\mathcal{H}} ,
\end{align}
for $p=1,2,3$, where 
\begin{align*}
\mathcal{V}&:= \Span_{\R} \{ \xi_1,\xi_2,\xi_3\}, \quad \mathcal{H}:= \mathcal{V}^{\perp_{g_{\alpha,\delta}}}
\end{align*}
denote the vertical and horizontal spaces, and we note that this corresponds to the metric $g_{\vec{a}}$ in (\ref{diagonalmetric}) for the parameters $a_1=a_2=a_3=\frac{1}{\delta^2}$, $a_4= \frac{1}{2\alpha\delta}$.
Substituting these into (\ref{spnONBline1}) and (\ref{spnONBline2}) gives a $g_{\alpha,\delta}$-orthonormal basis
\begin{align}
	e_r&:= \xi_r , \quad e_{4p}:= j\sqrt{\alpha\delta}F_{1,p+1}^{(n)} , \quad
	e_{4p+1}:= k\sqrt{\alpha\delta}F_{1,p+1}^{(n)} ,\label{tadONB} \\
	e_{4p+2}&:= i\sqrt{\alpha\delta}F_{1,p+1}^{(n)} , \quad 
	e_{4p+3}:= \sqrt{\alpha\delta}E_{1,p+1}^{(n)}, \label{tadONBline2}
\end{align}
for $r=1,2,3$ and $p=1,\dots, n-1$, and the vertical and horizontal spaces are given with respect to this basis by
\[ 
\mathcal{V}=\Span_{\R}\{ e_1,e_2,e_3\},\quad \mathcal{H}=\Span_{\R}\{ e_4,\dots, e_{4n-1}\}.
\]
The fundamental 2-forms $\Phi_r(X,Y):=g_{\alpha,\delta}(X,\varphi_r(Y))$ are given by
\begin{align}
\Phi_1&= -\xi_2\wedge \xi_3 -\sum_{p=1}^{n-1} (e_{4p}\wedge e_{4p+1} + e_{4p+2}\wedge e_{4p+3}) ,  \label{Phi1}\\
\Phi_2&= \xi_1\wedge \xi_3 -\sum_{p=1}^{n-1}(e_{4p}\wedge e_{4p+2} - e_{4p+1}\wedge e_{4p+3}) ,  \label{Phi2} \\
\Phi_3&= -\xi_1\wedge \xi_2 -\sum_{p=1}^{n-1}(e_{4p} \wedge e_{4p+3}+ e_{4p+1}\wedge e_{4p+2}) . \label{Phi3} 
\end{align}
\begin{remark}\label{deformed_KS_bundle}
Of all the homogeneous realizations of spheres, $S^{4n-1}=\Sp(n)/\Sp(n-1)$ is the only one to admit an invariant $3$-Sasakian structure \cite{BG3Sas}, hence also the only one to admit an invariant $\tad$ structure \cite{hom3alphadelta}. It is shown in \cite{AHduality} that any $\tad$ manifold admits a certain bundle $E\subseteq \Sigma M$ spanned by spinors satisfying the \emph{deformed Killing equation}
\begin{align}
\nabla^g_X\psi &= \frac{\alpha}{2} X\cdot \psi + \frac{\alpha-\delta}{2}\sum_{p=1}^3 \eta_p(X)\Phi_p\cdot \psi \qquad \text{for all } X\in TM, \label{deformed_KS}
\end{align}
generalizing the Killing equation in the $3$-Sasakian case ($\alpha=\delta=1$). Indeed, the bundle $E$ can been constructed explicitly in terms of the structure tensors, using the technique of Friedrich and Kath in \cite{Fried90}. They considered the spaces
\begin{align*}
E_{i}^{\pm} &= \{ \psi\in\Sigma M \: (\pm 2\varphi_i(X) +\xi_i\cdot X - X\cdot \xi_i)\cdot \psi =0 \text{ for all } X\in TM\}, \qquad i=1,2,3
\end{align*}
and showed that, in the Sasakian and 3-Sasakian cases, these are spanned by Killing spinors. In dimension $4n-1$, by choosing the Clifford algebra representation with $u_0 \cdot \eta^{\pm} =\mp \eta^{\pm}$, we have $\dim E_i^+=0$, and $\dim E_i^- =2$ (see \cite{BFGK}, \cite{Fried90}), and the space $E$ from the preceding theorem is given by the (non-direct) sum $E=E_1^-+E_2^- +E_3^-$ (Theorem 3.1 in \cite{AHduality}). Moreover, an argument analogous to the proof of Theorem \ref{deformedSasakianinvariantspinors} shows that $E$ has a basis consisting of invariant spinors.
\end{remark}
On the sphere $S^{4n-1}=\Sp(n)/\Sp(n-1)$, the space $E_1^-$ has a simple description in terms of the basis of invariant spinors from Theorem \ref{invspinorssp(n)sp(n-1)}:
\begin{proposition}\label{E1spinorbasis}
On $S^{4n-1}= \Sp(n)/\Sp(n-1)$ with the invariant $\tad$ metric $g_{\alpha,\delta}$ one has
\begin{align} 
E_1^-&= \Span_{\C} \{ 1,\  y_1\wedge \omega^{n-1}   \}. 
\end{align}
\end{proposition}
\begin{proof}
As noted in \cite{Fried90}, the defining condition of $E_1^-$ is equivalent to $e_j\cdot \varphi_1(e_j)\cdot \psi = \xi_1\cdot \psi$ for all $j=2,\dots 4n-1$, so it suffices to show that 
\begin{align*}
e_{2p}\cdot e_{2p+1} \cdot 1 &= \xi_1\cdot 1 \quad  \text{ and } \quad 		e_{2p}\cdot e_{2p+1} \cdot (y_1\wedge \omega^{n-1}) = \xi_1\cdot (y_1\wedge \omega^{n-1})
\end{align*}
for all $p=1,\dots, 2n-1$. Using the realization of the spin representation described in Remark \ref{reorderedCliffalgrepresentation}, we calculate
\begin{align*}
e_{2p}\cdot e_{2p+1} &= i(x_p\lrcorner \ + \ y_p\wedge) \circ (y_p\wedge \ - \ x_p\lrcorner) = i\left[ x_p\lrcorner \circ y_p \wedge   -y_p\wedge \circ x_p\lrcorner   \right] ,
\end{align*}
and hence
\begin{align*}
e_{2p}\cdot e_{2p+1} \cdot 1 &= i\left[ x_p\lrcorner(y_p\wedge 1) - y_p \wedge (x_p\lrcorner 1)   \right] =i =\xi_1\cdot 1
\end{align*}
and
\begin{align*}
e_{2p}\cdot e_{2p+1}\cdot (y_1\wedge \omega^{n-1}) &= i\left[ x_p\lrcorner (y_p\wedge y_1 \wedge \omega^{n-1}) - y_p\wedge (x_p\lrcorner(y_1\wedge \omega^{n-1}))     \right] \\
&= -i y_1\wedge \omega^{n-1} = \xi_1 \cdot (y_1\wedge \omega^{n-1})
\end{align*}  
for all $p=1,\dots, 2n-1$. The penultimate equality follows by considering separately the cases $p=1$, $p\neq 1$ and using the fact that $\omega^{n-1}$ is a multiple of $y_2\wedge y_3\wedge \dots \wedge y_{2n-1}$, hence $y_p\wedge (x_p\lrcorner \omega^{n-1}) = \omega^{n-1}$ for $p\neq 1$.
\end{proof}
\begin{remark}
It is worth noting that the spinors $\omega^j $, $y_1\wedge \omega^j $ appearing in Theorem \ref{invspinorssp(n)sp(n-1)} have an interpretation in terms of the $\tad$ structure tensors. Indeed, using the spin representation described in Remark \ref{reorderedCliffalgrepresentation} one has $$ y_1 = \frac{1}{\sqrt{2}}(\xi_2 + i\xi_3)     , \quad  \omega = -\frac{1}{2}(\Phi_2\rvert_{\mathcal{H}} + i\Phi_3\rvert_{\mathcal{H}} )  .  $$
\end{remark}
Finally, before discussing the situation in dimension $7$, we recall the existence of the second Einstein metric on a $3$-Sasakian manifold:
\begin{remark}\label{spn_second_einstein_metric}
	It was shown in \cite{BG3Sas} that a $3$-Sasakian manifold admits (uniquely up to homothety) a second Einstein metric of positive scalar curvature, which differs from the $3$-Sasakian metric by a rescaling along the fibres of the canonical fibration. In dimension $7$ it is known from \cite{nearly_parallel_g2} that this scaling factor is $\frac{1}{5}$, and, more generally, it was shown in Proposition 2.3.3 in \cite{3str} that a $\tad$ manifold of dimension $4n-1$ is Riemannian Einstein if and only if $\delta=\alpha$ or $\delta= (2n+1)\alpha$; in particular, comparing with (\ref{tadstructuretensorsdefinition}) easily recovers the factor of $\frac{1}{5}$ in the $7$-dimensional case. It was furthermore shown in \cite{nearly_parallel_g2} that the second Einstein metric admits a proper nearly parallel $\G_2$-structure (equivalently, a unique Killing spinor up to scaling), and we shall see in the following example that this spinor turns out to be the canonical spinor of the $\tad$ structure. 
\end{remark}
\begin{example}\label{sp2sp1}Let us examine more closely the situation for the $\tad$ 7-sphere, $(S^7=\Sp(2)/\Sp(1),g_{\alpha,\delta})$, and compare with the spinors found in \cite{3Sasdim7} and \cite{3str}. At the Lie algebra level we decompose $\mathfrak{sp}(2)=\mathfrak{sp}(1)\oplus_{\perp_{\kappa_{\mathfrak{sp}(2)}}} \mathfrak{m}$, where
\begin{align*}
\mathfrak{sp}(1)&= \Span_{\R}\left\{ iF_{2,2}^{(2)}, jF_{2,2}^{(2)} , kF_{2,2}^{(2)}  \right\},\\
\mathcal{V}&= \Span_{\R}\left\{\xi_1:= i \delta  F_{1,1}^{(2)} ,\ \xi_2:= -k\delta F_{1,1}^{(2)} , \  \xi_3:=  j\delta F_{1,1}^{(2)}  \right\} , \\
\mathcal{H}&=\Span_{\R}\left\{e_4:= j\sqrt{\alpha\delta}F_{1,2}^{(2)} ,\ e_5:= k\sqrt{\alpha\delta}F_{1,2}^{(2)},\  e_6:= i\sqrt{\alpha\delta}F_{1,2}^{(2)}   , \ e_7:= \sqrt{\alpha\delta}E_{1,2}^{(2)}    \right\}, \\
\mathfrak{m}&:= \mathcal{V}\oplus \mathcal{H},
\end{align*}
and orthogonality is with respect to the Killing form $\kappa_{\mathfrak{sp}(2)}= -12 B_0$ on $\mathfrak{sp}(2)$. From \cite{hom3alphadelta}, the $\tad$ structure is given by the tensors $g_{\alpha,\delta}$, $\xi_p$, $\varphi_p$, $(p=1,2,3)$ described above. The above basis for $\mathfrak{m}$ is $g_{\alpha,\delta}$-orthonormal, and adapted to the $\tad$ structure in the sense that the fundamental 2-forms are given by
\begin{align*}
\Phi_1&= -( \xi_{2,3} +  e_{4,5} + e_{6,7} ), \quad 
\Phi_2= -(\xi_{3,1} + e_{4,6} - e_{5,7} ) ,\quad 
\Phi_3= -(\xi_{1,2}  +e_{4,7} +e_{5,6} ). 
\end{align*}
Using the spin representation described in Remark \ref{reorderedCliffalgrepresentation}, it follows from Theorem \ref{invspinorssp(n)sp(n-1)} that the space of invariant spinors is
\begin{align*}
\Sigma_{\inv} = \Span_{\C} \{   1,  \omega,   y_1,  y_1\wedge \omega  \},
\end{align*}
where $\omega:= y_2\wedge y_3$.

Let us illustrate in detail the process of finding these invariant spinors by hand. To begin, one finds that the isotropy operators are given by
\begin{align}
\ad(iF_{2,2}^{(2)} )\rvert_{\mathfrak{m}}&=  e_{4,5} -e_{6,7}, \quad  \ad(jF_{2,2}^{(2)} )\rvert_{\mathfrak{m}}=   -e_{4,7} +e_{5,6},\quad \ad(kF_{2,2}^{(2)} )\rvert_{\mathfrak{m}}= -e_{4,6} - e_{5,7}. \label{sp2sp1isotropyoperators}
\end{align}
Now, applying the first operator in (\ref{sp2sp1isotropyoperators}) to aritrary $\eta \in \Sigma = \Lambda^{\bullet} L'$ gives
\begin{align*}
\widetilde{	\ad(iF_{2,2}^{(2)})} \cdot \eta = \frac{1}{2}(e_4\cdot e_5 - e_6\cdot e_7) \cdot \eta   = \frac{1}{2}\left[  i(x_2\lrcorner +y_2\wedge) \ (y_2\wedge -x_2\lrcorner)\eta  - i(x_3\lrcorner+y_3 \wedge) \ (y_3\wedge -x_3 \lrcorner)\eta       \right]   ,
\end{align*}
and hence 
\begin{align*}
\widetilde{\ad(iF_{2,2}^{(2)})\rvert_{\mathfrak{m}}}\cdot 1 &= 0, \quad \widetilde{\ad(iF_{2,2}^{(2)})\rvert_{\mathfrak{m}}}\cdot y_1 =0, \quad \widetilde{\ad(iF_{2,2}^{(2)})\rvert_{\mathfrak{m}}}\cdot y_2 =-iy_2 , \\ \widetilde{\ad(iF_{2,2}^{(2)})\rvert_{\mathfrak{m}}}\cdot y_3 &= iy_3 , \quad
\widetilde{\ad(iF_{2,2}^{(2)})\rvert_{\mathfrak{m}}}\cdot (y_1\wedge y_2) = -iy_1\wedge y_2  ,  \quad \widetilde{\ad(iF_{2,2}^{(2)})\rvert_{\mathfrak{m}}}\cdot (y_2\wedge y_3) = 0, \\ 
\widetilde{\ad(iF_{2,2}^{(2)})\rvert_{\mathfrak{m}}}\cdot (y_1\wedge y_3) &= iy_1\wedge y_3 , \quad
\widetilde{\ad(iF_{2,2}^{(2)})\rvert_{\mathfrak{m}}}\cdot (y_1\wedge y_2\wedge y_3) =  0.
\end{align*}
The kernel of this operator is therefore given by
\[
\ker \widetilde{\ad(iF_{2,2}^{(2)})\rvert_{\mathfrak{m}}} = \Span_{\C} \{1, \ y_2\wedge y_3, \ y_1, \ y_1\wedge y_2\wedge y_3\}.
\]
Continuing similarly for the other two operators in (\ref{sp2sp1isotropyoperators}) and taking the intersection of the three kernels gives{\small
\[
\Sigma_{\inv} = \left( \ker \widetilde{\ad(iF_{2,2}^{(2)})\rvert_{\mathfrak{m}}} \right)\cap \left(\ker \widetilde{\ad(jF_{2,2}^{(2)})\rvert_{\mathfrak{m}}} \right) \cap \left(\ker \widetilde{\ad(kF_{2,2}^{(2)})\rvert_{\mathfrak{m}}}\right)  = \Span_{\C} \{1, \ y_2\wedge y_3, \ y_1, \ y_1\wedge y_2\wedge y_3\}.
\]
}
\begin{remark}
The canonical spinor $\psi_0$ and three auxiliary spinors $\psi_r:=\xi_r\cdot \psi_0$ $(r=1,2,3)$ described in Theorem 4.5.2 of \cite{3str} are given in terms of the above basis of $\Sigma_{\inv}$ by
\begin{align} \label{7tad_sphere_spinors}
\psi_0&= \frac{1}{\sqrt{2}}(\omega+ i y_1), \quad \psi_1= \frac{1}{\sqrt{2}} (i\omega + y_1), \quad \psi_2=\frac{1}{\sqrt{2}}( -1+i y_1\wedge \omega), \quad \psi_3= \frac{1}{\sqrt{2}} (-i+y_1\wedge \omega).
\end{align}
These spinors were studied in \cite{3Sasdim7}, and subsequently \cite{3str}, from the perspective of $\G_2$-geometry using the $\G_2$-form 
\[
\omega=\sum_{p=1}^3 \eta_p\wedge \Phi_p\rvert_{\mathcal{H}} + \eta_{123}
\]
arising naturally from the $\tad$ structure. They prove that the associated $\G_2$-structure is cocalibrated, and its characteristic connection agrees with the $\tad$ canonical connection (Theorem 4.5.1 in \cite{3str}). The \emph{canonical spinor} $\psi_0$ (or rather, its real part) arises as a unit vector in the $(-7)$-eigenspace for Clifford multiplication of $\omega$ on the real spin bundle, and is unique up to sign. 
\end{remark}
Let us now differentiate the spinors (\ref{7tad_sphere_spinors}) and compare with Theorem 4.5.2 of \cite{3str}. From \cite{hom3alphadelta}, the Nomizu map for the Levi-Civita connection is given by
\[
\Uplambda^{g_{\alpha,\delta}}(X)Y =  \begin{cases} 
\frac{1}{2}[X,Y]_{\mathfrak{m}}  & X,Y\in\mathcal{V} \ \text{or} \ X,Y\in\mathcal{H}, \\
(1-\frac{\alpha}{\delta})[X,Y] & X\in\mathcal{V}, Y\in\mathcal{H}, \\
\frac{\alpha}{\delta}[X,Y]  & X\in\mathcal{H}, Y\in\mathcal{V},
\end{cases}  
\]
and we calculate
\begin{align*}
\Uplambda^{g_{\alpha,\delta}}(\xi_1) &= \delta \xi_{2,3} +\delta (1-\frac{\alpha}{\delta}) (e_{4,5} +e_{6,7}), \quad \Uplambda^{g_{\alpha,\delta}}(\xi_2)= \delta\xi_{3,1}+ \delta (1-\frac{\alpha}{\delta}) (e_{4,6} -e_{5,7}), \\
\Uplambda^{g_{\alpha,\delta}}(\xi_3)&= \delta \xi_{1,2} + \delta (1-\frac{\alpha}{\delta}) (e_{4,7} +e_{5,6} ),  \quad \Uplambda^{g_{\alpha,\delta}}(e_4)= \alpha ( - \xi_1\wedge e_5 - \xi_2\wedge e_6 - \xi_3\wedge e_7  ), \\
\Uplambda^{g_{\alpha,\delta}}(e_5)&= \alpha ( \xi_1\wedge e_4 + \xi_2\wedge e_7 - \xi_3\wedge e_6    ) ,\quad 
\Uplambda^{g_{\alpha,\delta}}(e_6)= \alpha( - \xi_1\wedge e_7 +\xi_2\wedge e_4 +\xi_3\wedge e_5    ), \\
\Uplambda^{g_{\alpha,\delta}}(e_7)&= \alpha( \xi_1\wedge e_6 -\xi_2\wedge e_5 +\xi_3\wedge e_4   ).
\end{align*}
Lifting these to the spin bundle and calculating in the spin representation, as we did above for the isotropy operators, then gives the generalized Killing equations from Theorem 4.5.2 in \cite{3str}:
\begin{align} \label{dim7spinorialeqn3ad}
\nabla^{g_{\alpha,\delta}}_X\psi_0 &=  \begin{cases} 
-\frac{3\alpha}{2}X\cdot \psi_0  &  X\in\mathcal{H}, \\
\frac{2\alpha-\delta}{2}X\cdot \psi_0 & X\in\mathcal{V}, \\
\end{cases} \qquad 
\nabla^{g_{\alpha,\delta}}_X\psi_i = \begin{cases} 
\frac{2\alpha-\delta}{2} \xi_i\cdot \psi_i &  X=\xi_i, \\
\frac{3\delta-2\alpha}{2}\xi_j\cdot \psi_i  & X=\xi_j \ (j\neq i), \\
\frac{\alpha}{2}X\cdot \psi_i & X\in\mathcal{H},
\end{cases}
\end{align}
for $i=1,2,3$. For example, one calculates
{\small
\begin{align*} 
\widetilde{\Uplambda^{g_{\alpha,\delta}}(\xi_1)}\cdot \psi_0 &= \frac{\delta}{2}\left[  i( x_1\lrcorner + y_1\wedge) (y_1\wedge - x_1\lrcorner)\ \frac{1}{\sqrt{2}} (\omega + iy_1 ) \right] \\
&\qquad  + \frac{ (\delta-\alpha) }{2}\left[i( x_2\lrcorner + y_2\wedge) (y_2\wedge - x_2\lrcorner)\ \frac{1}{\sqrt{2}} (\omega + iy_1 ) + i( x_3\lrcorner + y_3\wedge) (y_3\wedge - x_3\lrcorner)\ \frac{1}{\sqrt{2}} (\omega + iy_1 ) \right]   \\
&= \frac{\delta}{2}\left[ i(x_1\lrcorner+y_1\wedge)\ \frac{1}{\sqrt{2}}  (-i +  y_1\wedge \omega )    \right] \\
&\qquad +\frac{(\delta-\alpha)}{2}\left[ i(x_2\lrcorner+y_2\wedge)\ \frac{1}{\sqrt{2}}  (iy_2\wedge y_1 - y_3 ) + i(x_3\lrcorner+y_3\wedge)\ \frac{1}{\sqrt{2}}( iy_3\wedge y_1 +y_2  )      \right] \\
&= \frac{i\delta}{2\sqrt{2}} \left[ -iy_1 +\omega    \right] + \frac{i(\delta-\alpha)}{2\sqrt{2}} \left[ (iy_1 -y_2\wedge y_3 ) +(iy_1 +y_3\wedge y_2 )      \right] \\
&=  \frac{(2\alpha-\delta)}{2} \ \frac{1}{\sqrt{2}} (y_1+i\omega) = \frac{(2\alpha-\delta)}{2} \psi_1 = \frac{(2\alpha-\delta)}{2} \xi_1\cdot \psi_0     .
\end{align*}}A similar calculation in the spin representation also shows that equation (\ref{dim7spinorialeqn3ad}) for the auxiliary spinors $\psi_r$, $r=1,2,3$ is equivalent to the deformed Killing equation (\ref{deformed_KS}) in dimension 7 \cite{AHduality}. Moreover, substituting the parameters for the second Einstein metric, $g_2:=g_{\alpha,\delta}\rvert_{\delta=5\alpha}$ (see Remark \ref{spn_second_einstein_metric}), into (\ref{dim7spinorialeqn3ad}) gives 
\[
\nabla^{g_2}_X\psi_0 =  -\frac{3\alpha}{2}X\cdot \psi_0 , \quad 
\nabla^{g_2}_X\psi_i = \begin{cases} 
	-\frac{3\alpha}{2} \xi_i\cdot \psi_i &  X=\xi_i ,\\
	\frac{13\alpha}{2}\xi_j\cdot \psi_i  & X=\xi_j \ (j\neq i), \\
	\frac{\alpha}{2}X\cdot \psi_i & X\in\mathcal{H},
\end{cases}  
\]
which, in particular, shows that $\psi_0$ is the Killing spinor determining the proper nearly parallel $\G_2$-structure described in \cite{nearly_parallel_g2}.
\end{example}
Finally, before discussing the general invariant metrics (\ref{diagonalmetric}) in more detail, we compare the Ambrose-Singer connection to the canonical connection of the $\tad$ structure introduced in \cite{3str}:
\begin{corollary}
	The canonical connection of the $\tad$ space $(S^{4n-1}=\Sp(n)/\Sp(n-1), g_{\alpha,\delta})$ coincides with the Ambrose-Singer connection if and only if the $\tad$ structure is parallel ($\delta=2\alpha$). 
\end{corollary}
\begin{proof}
	We have seen that $g_{\alpha,\delta}$ is obtained from $g_{\vec{a}} $ by setting $a_1=a_2=a_3=\frac{1}{\delta^2}$ and $a_4=\frac{1}{2\alpha\delta}$. Recalling that the canonical connection has skew torsion (Theorem 4.4.1 in \cite{3str}), if the two connections are assumed to coincide then Proposition \ref{spnASTORSION} implies $a_1=a_2=a_3=a_4$, hence $\delta=2\alpha$. Conversely, if $\delta=2\alpha$, then $a_1=a_2=a_3=a_4= \frac{1}{4\alpha^2}$, and Proposition \ref{spnASTORSION} implies that the Ambrose-Singer connection has skew torsion given by
	\begin{align*}
		T^{\AS} =   -4\alpha \  e_1\wedge e_2\wedge e_3 +2\alpha \sum_{i=1}^3 e_i\wedge \Phi_i\rvert_{\mathfrak{m}_4} .
	\end{align*}
	The result then follows by comparing this to Theorem 4.4.1 in \cite{3str}.
\end{proof}
\subsubsection{General invariant metrics on $\Sp(n)/\Sp(n-1)$} 
We now leave the $\tad$ setting and return to the general invariant metrics (\ref{diagonalmetric}). In order to differentiate the invariant spinors from Theorem \ref{invspinorssp(n)sp(n-1)}, it is helpful to compare $g_{\vec{a}}$ with the round ($3$-Sasakian) metric, $g':=g_{\alpha,\delta}\rvert_{\alpha=\delta=1}$; they are related by
\[
g_{\vec{a}} = b_1g'\rvert_{\mathfrak{m}_1\times \mathfrak{m}_1} + b_2 g'\rvert_{\mathfrak{m}_2\times \mathfrak{m}_2} + b_3g'\rvert_{\mathfrak{m}_3\times \mathfrak{m}_3} + b_4g'\rvert_{\mathfrak{m}_4\times \mathfrak{m}_4},
\]
where $b_i:= a_i$ $(i=1,2,3)$ and $b_4:=2a_4$. We denote by $\{\overline{e}_i\} $ the $g_{\vec{a}}$-orthonormal basis defined in (\ref{spnONBline1})-(\ref{spnONBline2}), and by $\{e_i\}$ the $g'$-orthonormal basis defined by setting $\alpha=\delta=1$ in (\ref{tadONB})-(\ref{tadONBline2}). By adapting the proof of Proposition 2.33 in \cite{bourguignon2015spinorial}, we obtain:
%
%
%
%
%
\begin{lemma}\label{connection1formsrelation}
The Levi-Civita connection 1-forms $\overline{\omega}_{i,j}:= g_{\vec{a}}(\nabla^{g_{\vec{a}}} \overline{e}_i, \overline{e}_j)$ and $\omega_{i,j}':= g'(\nabla^{g'} e_i, e_j)$ are related by 
\[
\overline{\omega}_{i,j}(\overline{e}_k) = \frac{1}{2}\left( \Theta^p_{q,r}+\Theta^q_{p,r}    \right)  \omega_{i,j}'(e_k) +\frac{1}{2}\left( \Theta^q_{p,r} - \Theta^r_{p,q}   \right)\omega_{j,k}'(e_i) + \frac{1}{2}\left( \Theta^r_{p,q} -\Theta^p_{q,r}     \right)\omega_{i,k}'(e_j)  
\]
for $ e_i\in\mathfrak{m}_p, \ e_j\in\mathfrak{m}_q,\ e_k \in \mathfrak{m}_r$, where $ \Theta^l_{m,n} := \sqrt{ \frac{b_l}{b_m b_n}}$.
\end{lemma}
\begin{proof}
	Let $e_i,e_j,e_k,\Theta^l_{m,n}$ be as in the statement of the lemma. Using the Koszul formula and the fact that $\nabla^{g'}$ is torsion-free, we calculate
\begin{align*}
\overline{\omega}_{i,j}(\overline{e}_k) &= g_{\vec{a}}( \nabla^{g_{\vec{a}}}_{\overline{e}_k} \overline{e}_i, \overline{e}_j) = \frac{1}{2} \left[ -g_{\vec{a}}([\overline{e}_i,\overline{e}_k]_{\mathfrak{m}}, \overline{e}_j) - g_{\vec{a}}([\overline{e}_k,\overline{e}_j]_{\mathfrak{m}}, \overline{e}_i) - g_{\vec{a}}([\overline{e}_i,\overline{e}_j]_{\mathfrak{m}}, \overline{e}_k)     \right] \\
&= \frac{1}{2}\left[ -\Theta^q_{p,r}   g'([e_i,e_k]_{\mathfrak{m}},e_j) - \Theta^p_{q,r}    g'([e_k,e_j]_{\mathfrak{m}},e_i) - \Theta^r_{p,q}    g'([e_i,e_j]_{\mathfrak{m}},e_k)        \right] \\
&=  -\frac{1}{2}\Theta^q_{p,r}   g'(\Uplambda^{g'}(e_i)e_k - \Uplambda^{g'}(e_k)e_i,e_j) - \frac{1}{2}\Theta^p_{q,r}   g'(\Uplambda^{g'}(e_k)e_j-\Uplambda^{g'}(e_j)e_k,e_i) \\
&\qquad - \frac{1}{2}\Theta^r_{p,q}   g'(\Uplambda^{g'}(e_i)e_j-\Uplambda^{g'}(e_j)e_i,e_k)        ,
\end{align*}
and the result then follows from the fact that $\nabla^{g'}$ is metric (for $g'$). 
\end{proof}
In dimension 7, this comparison with the round metric allows us to easily find new examples of generalized Killing spinors:
\begin{proposition}\label{gks4EVs}
	The spinors $\psi_i$, $i=0,1,2,3$ on the 7-sphere $(S^7=\frac{\Sp(2)}{\Sp(1)},g_{\vec{a}} )$, defined as in (\ref{7tad_sphere_spinors}), are generalized Killing spinors for the endomorphisms
\begin{align*}
A_i&=  \lambda_{i,1} \Id\rvert_{\mathfrak{m}_1} +   \lambda_{i,2}\Id\rvert_{\mathfrak{m}_2}+  \lambda_{i,3}\Id\rvert_{\mathfrak{m}_3}+\lambda_{i,4} \Id\rvert_{\mathfrak{m}_4}, \quad i=0,1,2,3,
\end{align*}
with eigenvalues
\begin{align*}
	\lambda_{0,p}&= \begin{cases} 
		\frac{1}{2}( - \Theta^p_{p+1, p+2}+\Theta^{p+1}_{p, p+2} + \Theta^{p+2}_{p, p+1}) - (\Theta^4_{p,4} -\Theta^p_{4,4}) & p=1,2,3,\\
  -\frac{1}{2}(\Theta^1_{4,4} + \Theta^2_{4,4} +\Theta^3_{4,4})  & p=4,\\
	\end{cases}\\
\lambda_{k,p} &= 
\begin{cases}   
		\frac{1}{2}( - \Theta^p_{p+1, p+2}+\Theta^{p+1}_{p, p+2} + \Theta^{p+2}_{p, p+1}) -  (\Theta^4_{p,4} -\Theta^p_{4,4}) & k=p \text{ and } p=1,2,3, \\
			\frac{1}{2}( - \Theta^p_{p+1, p+2}+\Theta^{p+1}_{p, p+2} + \Theta^{p+2}_{p, p+1}) +  (\Theta^4_{p,4} -\Theta^p_{4,4}) & k\neq p \text{ and } p=1,2,3, \\
	\frac{1}{2}( -\Theta^k_{4,4} + \Theta^{k+1}_{4,4} + \Theta^{k+2}_{4,4}     )& p=4  ,  \\
\end{cases}
\end{align*}
where $k=1,2,3$, and the indices $k, k+1, k+2, p, p+1, p+2$ on the right hand side are taken modulo 3.
\end{proposition}
\begin{proof}
	Using the preceding lemma together with the explicit form of the Nomizu map for the 3-Sasakian 7-sphere (see Example \ref{sp2sp1}), we obtain:
	\begin{align*}
		\Lambda^{g_{\vec{a}} }(\overline{e}_p) &= (\Theta^p_{p+1,p+2} - \Theta^{p+1}_{p,p+2} -\Theta^{p+2}_{p,p+1}) \ \overline{\Phi}_p\rvert_{\mathcal{V}}  - (\Theta^4_{p,4} -\Theta^p_{4,4})\ \overline{\Phi}_p\rvert_{\mathcal{H}}, \quad p=1,2,3,\\
		\Lambda^{g_{\vec{a}}}(\overline{e}_4) &= -\Theta^1_{4,4} \overline{e}_{1,5} -\Theta^2_{4,4}\overline{e}_{2,6} - \Theta^3_{4,4} \overline{e}_{3,7} , \quad \Lambda^{g_{\vec{a}}}(\overline{e}_5) = \Theta^1_{4,4} \overline{e}_{1,4} +\Theta^2_{4,4}\overline{e}_{2,7} - \Theta^3_{4,4} \overline{e}_{3,6}, \\
		\Lambda^{g_{\vec{a}}}(\overline{e}_6) &= -\Theta^1_{4,4} \overline{e}_{1,7} +\Theta^2_{4,4}\overline{e}_{2,4} + \Theta^3_{4,4} \overline{e}_{3,5}, \quad \Lambda^{g_{\vec{a}}}(\overline{e}_7) = \Theta^1_{4,4} \overline{e}_{1,6} -\Theta^2_{4,4}\overline{e}_{2,5} + \Theta^3_{4,4} \overline{e}_{3,4},
	\end{align*}
where the indices $p,p+1,p+2$ are taken modulo 3, and $\overline{\Phi}_p$ are the forms defined by replacing each $e_i$ with $\overline{e}_i$ (and replacing each $\xi_i$ with $\overline{e}_i$, $i=1,2,3$) in (\ref{Phi1})-(\ref{Phi3}). The result then follows by lifting these operators and calculating the Clifford products with $\psi_i$, $i=0,1,2,3$ in the spin representation.
\end{proof}
\begin{remark}
	By choosing the metric parameters $a_1,a_2,a_3,a_4$ in (\ref{diagonalmetric}) appropriately, the endomorphisms $A_i$ can be arranged to have $4$ distinct eigenvalues. As far as the authors are aware, this provides the first example of generalized Killing spinors whose endomorphism has four distinct eigenvalues (see \cite{3Sasdim7, 3str} for examples of generalized Killing spinors with two or three distinct eigenvalues). We also note that, in the case of the $\tad$ metric ($b_1=b_2=b_3=\frac{1}{\delta^2}$, $b_4=\frac{1}{\alpha\delta}$), we have 
	\[
	\Theta^p_{q,r} = \Theta^4_{p,4} = |\delta|, \qquad \Theta^p_{4,4} = |\alpha|, \qquad \text{ for } p,q,r\in\{1,2,3\}.
	\]
	Since $S^{4n-1}$ is compact we have, by convention, $\alpha\delta>0$ (cf. \cite[Thm.\@ 3.1.1]{hom3alphadelta}), and thus $\alpha$ and $\delta$ have the same sign. If $\alpha, \delta>0$, then the generalized Killing equations in Proposition \ref{gks4EVs} immediately recover the known equations (\ref{dim7spinorialeqn3ad}). If $\alpha,\delta<0$, then we recover the equations (\ref{dim7spinorialeqn3ad}) up to a factor of $-1$, corresponding to the fact that replacing $\alpha,\delta$ with $-\alpha,-\delta$ in the orthonormal basis (\ref{tadONB})-(\ref{tadONBline2}) gives a basis with the opposite orientation.
\end{remark}
Somewhat surprisingly, performing a similar deformation of the 3-Sasakian Killing spinors in dimensions larger than $7$ is not guaranteed to produce generalized Killing spinors, as the following proposition shows:
\begin{proposition}\label{deformationofE1minus}
	Let $\psi = \mu_1 1 + \mu_2 y_1\wedge \omega^{n-1} \in E_1^-$ ($\mu_1,\mu_2\in \C$) be an invariant Killing spinor for the $3$-Sasakian metric on $S^{4n-1}=\Sp(n)/\Sp(n-1)$. If $n>2$, then the spinor $\psi$ on $(S^{4n-1}=\frac{\Sp(n)}{\Sp(n-1)},g_{\vec{a}} )$ defined by the same formula is a generalized Killing spinor if and only if $b_2=b_3=b_4$. If $b_2=b_3=b_4$, then $\psi$ is a generalized Killing spinor for the endomorphism
	\begin{align*}
		A&= \frac{1}{2}\left[(1-2n) \Theta^1_{2,2}+ 2n  \Theta^2_{1,2} \right] \Id\rvert_{\mathfrak{m}_1} + \frac{1}{2}\Theta^1_{2,2} \Id\rvert_{\mathfrak{m}_2\oplus\mathfrak{m}_3\oplus \mathfrak{m}_4}
	\end{align*}
with at most two distinct eigenvalues.
\end{proposition}
\begin{proof}
	Using Lemma \ref{connection1formsrelation}, the Nomizu map for the Levi-Civita connection of $g_{\vec{a}} $ takes the same form as in the proof of the preceding proposition, with $\overline{e}_4$ replaced with $\overline{e}_{4p}$, $\overline{e}_5$ replaced with $\overline{e}_{4p+1}$, and so on. Using the spin representation ordering described in Remark \ref{reorderedCliffalgrepresentation}, one sees that Clifford multiplication by $\overline{\Phi}_2\rvert_{\mathcal{H}}$ and $\overline{\Phi}_3\rvert_{\mathcal{H}}$ (resp. $\overline{\Phi}_2\rvert_{\mathcal{V}}$ and $\overline{\Phi}_3\rvert_{\mathcal{V}}$) changes the degree of the spinors $1$ and $y_1\wedge \omega^{n-1}$ by two (resp. one). On the other hand, Clifford multiplication by a vector changes the degree by at most one. Thus, by comparing the degrees of $\widetilde{\Uplambda^{g_{\vec{a}}}}(\overline{e}_2)\cdot \psi$ and $\widetilde{\Uplambda^{g_{\vec{a}}}}(\overline{e}_3)\cdot \psi $ with elements of $\mathfrak{m}\cdot \psi$, we see that if $\psi$ is a generalized Killing spinor and $n>2$ then $\Theta^4_{2,4}-\Theta^2_{4,4}=0=\Theta^4_{3,4}-\Theta^3_{4,4}$. Simplifying these equations gives $b_2=b_3=b_4$, as desired. Conversely, if $b_2=b_3=b_4$ then lifting the Nomizu operators and calculating the Clifford product with $\psi$ in the spin representation gives the result. 
\end{proof}
The preceding proposition shows that attempting to produce generalized Killing spinors with a certain number of distinct eigenvalues by rescaling the isotropy components of metrics carrying Killing spinors is not a straightforward process. Indeed, if one starts with an arbitrary invariant Killing spinor for the round metric on $S^{4n-1}=\Sp(n)/\Sp(n-1)$, which, by a result in \cite{Hof22}, may be written as a linear combination of $\psi_{k}:= \omega^{k+1}-i(k+1)y_1\wedge \omega^k$ ($-1\leq k\leq n-1$), the resulting system of algebraic equations determining precisely which linear combination of the $\psi_k$'s is needed is difficult to solve. It remains to be understood why this deformation technique works in some situations but not others, and whether it can be used to produce other interesting examples of generalized Killing spinors.
\subsection{$S^3$-Quaternionic Spheres, $S^{4n-1}=\frac{\Sp(n) \Sp(1)}{\Sp(n-1) \Sp(1)}$}
This is the case from Section \ref{caseIIIK=Sp(1)} corresponding to $K=\Sp(1)$. We begin by discussing the general case, then pass to the $7$-dimensional setting, where the invariant spinor is related to the exceptional $\G_2$-geometry available in this dimension. 

From (\ref{inclusion_isotropy_productgroup}), we have at the level of Lie algebras, {\small
\begin{align*}
	\mathfrak{sp}(n)\oplus \mathfrak{sp}(1)&= \Span_{\R} \{ (iF_{p,q}^{(n)},0) , (jF_{p,q}^{(n)},0) ,(kF_{p,q}^{(n)},0) ,(E_{r,s}^{(n)},0), (0,i) ,(0,j),(0,k)    \}_{\substack{1\leq p\leq q \leq n \\ 1\leq r < s \leq n} }     ,\\ 
	\mathfrak{sp}(n-1) \oplus \mathfrak{sp}(1)&=  \Span_{\R} \{ (iF_{p,q}^{(n)},0) , (jF_{p,q}^{(n)},0) ,(kF_{p,q}^{(n)},0) ,(E_{r,s}^{(n)},0), (iF_{1,1}^{(n)},i) ,(jF_{1,1}^{(n)},j),(kF_{1,1}^{(n)},k)    \}_{\substack{2\leq p\leq q \leq n \\ 2\leq r < s \leq n} }  ,
\end{align*}
}and for a reductive complement we take the orthogonal complement $\mathfrak{m}:=(\mathfrak{sp}(n-1) \oplus \mathfrak{sp}(1))^{\perp}$ with respect to the Killing form $\kappa$ on $\mathfrak{sp}(n)\oplus \mathfrak{sp}(1)$,
\begin{align} 
	\label{squashedKF} \kappa((A,z),(A',z')) := -4(n+1)B_0(A,A') +8\Re(zz') .
\end{align}
The isotropy representation decomposes into two inequivalent irreducible summands, $\mathfrak{m}\simeq \mathfrak{m}_1\oplus \mathfrak{m}_2$, leading to the 2-parameter family of invariant metrics,
\begin{align*}
g_{a,b}:= -a\kappa\rvert_{\mathfrak{m}_1\times \mathfrak{m}_1} - b\kappa\rvert_{\mathfrak{m}_2\times \mathfrak{m}_2}  , \quad a,b>0.
\end{align*}
A $g_{a,b}$-orthonormal basis for $\mathfrak{m}$ is given by {\small
\begin{align*}
	 \xi_1&:= \frac{1}{\Omega}\left( iF_{1,1}^{(n)}, -\left( \frac{n+1}{2}\right) i\right), \quad
	\xi_2:= \frac{1}{\Omega}\left( -kF_{1,1}^{(n)}, \left( \frac{n+1}{2}\right) k\right), \quad 
	\xi_3:=\frac{1}{\Omega}\left(jF_{1,1}^{(n)}, -\left( \frac{n+1}{2}\right) j\right),  \\
	e_{4p}&:= \frac{1}{2\sqrt{2b(n+1)}}(jF_{1,p+1}^{(n)} ,0),\quad 
	e_{4p+1}:=  \frac{1}{2\sqrt{2b(n+1)}}(kF_{1,p+1}^{(n)} ,0),\\
	e_{4p+2}&:= \frac{1}{2\sqrt{2b(n+1)}}(iF_{1,p+1}^{(n)} ,0),\quad 
	e_{4p+3}:= \frac{1}{2\sqrt{2b(n+1)}}(E_{1,p+1}^{(n)} ,0)
\end{align*}
}for $p=1,\dots, n-1$, where $\Omega:=\sqrt{2a(n+1)(n+3)}$. In terms of this basis, the two isotropy summands are
\begin{align*}
	\mathfrak{m}_1&:=\Span_{\R} \{ \xi_1,\xi_2,\xi_3\},\quad  \mathfrak{m}_2:= \Span_{\R} \{ e_4,\dots, e_{4n-1}\} .  
\end{align*}
From Proposition \ref{MAcaseK=sp1} we obtain:
\begin{theorem}\label{explicitspinorsK=sp1}
	Using the above orthonormal basis and the corresponding description of the spinor module from Remark \ref{reorderedCliffalgrepresentation}, the space of invariant spinors on $(S^{4n-1}=\frac{\Sp(n)\Sp(1)}{\Sp(n-1)\Sp(1)}, g_{a,b})$ for any $a,b>0$ is trivial unless $n=2$, in which case $\dim_{\C}  \Sigma_{\inv} =1$. In this case, the $1$-dimensional $\Sigma_{\inv}$ is contained in the span of $y_1$ and $\omega := \sum_{i=1}^{n-1} y_{2i} \wedge y_{2i+1}$.
\end{theorem}
\begin{proof}
	The result follows directly from Proposition \ref{MAcaseK=sp1} by noting that $\omega = \sum_{i=1}^{n-1} y_{2i} \wedge y_{2i+1}$ is the symplectic form stabilized by $\mathfrak{sp}(2n-2,\C)$.
\end{proof}
The 1-dimensional space of invariant spinors obtained in dimension 7 is explicitly constructed in Example \ref{sptwisted7sphereexample}, which appears immediately after the following proposition describing the Ambrose-Singer torsion in the general case:
\begin{proposition}
	For any $a,b>0$ the sphere $(S^{4n-1}=\frac{\Sp(n)\Sp(1)}{\Sp(n-1)\Sp(1)}, g_{a,b})$ has Ambrose-Singer torsion of type $\mathcal{T}_{\totallyskew}\oplus \mathcal{T}_{\CT}$, given by
		\begin{align*}
			T^{\AS}(\xi_1,\xi_2)&= \frac{(n-1)}{\Omega} \xi_3 , \quad T^{\AS}(\xi_1,\xi_3) = - \frac{(n-1)}{\Omega} \xi_2  , \quad  T^{\AS}(\xi_2,\xi_3) =  \frac{(n-1)}{\Omega}\xi_1 ,  \\
			T^{\AS}(\xi_1, - )\rvert_{\mathfrak{m}_2} &= \frac{1}{\Omega}\Phi_1\rvert_{\mathfrak{m}_2}, \quad T^{\AS}(\xi_2, - )\rvert_{\mathfrak{m}_2} = \frac{1}{\Omega}\Phi_2\rvert_{\mathfrak{m}_2}, \quad  T^{\AS}(\xi_3, - )\rvert_{\mathfrak{m}_2} = \frac{1}{\Omega}\Phi_3\rvert_{\mathfrak{m}_2},        \\
			T^{\AS}(e_{4p},e_{4q}) &= T^{\AS}(e_{4p+1},e_{4q+1}) = T^{\AS}(e_{4p+2},e_{4q+2}) = T^{\AS}(e_{4p+3},e_{4q+3})=0\\
			T^{\AS} (e_{4p},e_{4q+1}) &= \frac{-a \delta_{p,q} }{b\Omega } \xi_1 , \quad T^{\AS}(e_{4p},e_{4q+2}) = \frac{-a\delta_{p,q}}{b\Omega} \xi_2, \quad T^{\AS}(e_{4p},e_{4q+3}) = \frac{-a\delta_{p,q} }{b\Omega } \xi_3, \\
			T^{\AS}(e_{4p+1},e_{4q+2}) &= \frac{-a\delta_{p,q} }{b\Omega } \xi_3,\quad  T^{\AS} (e_{4p+1},e_{4q+3}) = \frac{a\delta_{p,q}}{b\Omega } \xi_2, \quad T^{\AS}(e_{4p+2},e_{4q+3}) = \frac{-a\delta_{p,q}}{b\Omega}\xi_1,  
		\end{align*}
	for $p,q=1,\dots, n-1$, where $\Phi_1,\Phi_2,\Phi_3$ are defined formally as in (\ref{Phi1})-(\ref{Phi3}). The projection of $T^{\AS}$ onto $\mathcal{T}_{\totallyskew}$ is 
	\[
	T^{\AS}_{\totallyskew} = \left( \frac{n-1}{\Omega} \right)\xi_1\wedge \xi_2\wedge \xi_3    +\frac{1}{3}\left( \frac{2}{\Omega} + \frac{a}{b\Omega }  \right)\sum_{i=1}^3  \xi_i \wedge \Phi_i\rvert_{\mathfrak{m}_2}   ,
	\]
	with $T^{\AS} = T^{\AS}_{\totallyskew}$ if and only if $a=b$.
\end{proposition}
The remainder of the section is devoted to discussion of the $7$-dimensional example, $S^7= \frac{Sp(2) \Sp(1)}{\Sp(1) \Sp(1)}$. We shall explicitly determine the invariant spinor in this dimension, and discuss how it fits into the larger picture of the well-known correspondence between spinors and $\G_2$-structures in dimension $7$.
\begin{example}\label{sptwisted7sphereexample}Following the setup outlined above, the isotropy algebra is
\begin{align*}
\mathfrak{sp}(1)\oplus\mathfrak{sp}(1)&= \left\{  \left( iF_{2,2}^{(2)},0\right),\left(jF_{2,2}^{(2)},0 \right),\left( kF_{2,2}^{(2)},0\right), \left( iF_{1,1}^{(2)} ,i\right),\left( jF_{1,1}^{(2)} ,j\right),\left( kF_{1,1}^{(2)},k\right)   \right\},
\end{align*}
and the two isotropy summands $\mathfrak{m}_1$, $\mathfrak{m}_2$ are given by
\begin{align*}
\mathfrak{m}_1&=  \Span_{\R}\left\{ \frac{1}{\sqrt{30a}}\left(iF_{1,1}^{(2)},\frac{-3i}{2}\right) ,\frac{1}{\sqrt{30a}}\left(-kF_{1,1}^{(2)},\frac{3k}{2}\right)  , \frac{1}{\sqrt{30a}}\left(  jF_{1,1}^{(2)},\frac{-3j}{2}\right)                  \right\} \\
& =: \{\xi_1,\xi_2,\xi_3\}, \\
\mathfrak{m}_2&=\Span_{\R}\left\{ \frac{1}{\sqrt{24b}}\left( jF_{1,2}^{(2)} ,0\right) , \frac{1}{\sqrt{24b}}\left(kF_{1,2}^{(2)},0\right),\frac{1}{\sqrt{24b}}\left(iF_{1,2}^{(2)},0\right)     ,\frac{1}{\sqrt{24b}}\left( E_{1,2}^{(2)} ,0\right)    \right\} \\
&=: \{e_4,e_5,e_6,e_7\}.
\end{align*}
The basis $\{\xi_1,\xi_2,\xi_3, e_4,e_5,e_6,e_7\}$ for $\mathfrak{m}=\mathfrak{m}_1\oplus \mathfrak{m}_2$ is orthonormal with respect to the invariant metric $g_{a,b} $ described above, and we shall also denote $e_i:=\xi_i$ ($i=1,2,3$) in certain places. Fixing the associated Clifford algebra representation as in Remark \ref{reorderedCliffalgrepresentation}, and letting $\omega:= y_2\wedge y_3$, we have,
\begin{theorem}\label{example_invariant_spinors_sp_twisted_7sphere}
For any $a,b>0$, the space of invariant spinors on $(S^{7}=\frac{\Sp(2)\Sp(1)}{\Sp(1)\Sp(1)}, g_{a,b})$ is given by
\begin{align*}
\Sigma_{\inv}&= \Span_{\C}\{\psi_0 = \frac{1}{\sqrt{2}}(\omega+iy_1)\}.
\end{align*}
\end{theorem}
\begin{proof}
	Considering Corollary \ref{K_restriction} and the spinors (\ref{7tad_sphere_spinors}) on the $\tad$ 7-sphere, the space of invariant spinors is the subspace of $\Span_{\C}\{\psi_0,\psi_1,\psi_2,\psi_3\}$ annihilated by the spin lifts of the three additional isotropy operators
	\begin{align*}
		\ad(iF_{1,1}^{(2)}, i)\rvert_{\mathfrak{m}} =&2\xi_2\wedge \xi_3 +e_4\wedge e_5 + e_6\wedge e_7, \quad 
		\ad(jF_{1,1}^{(2)}, j)\rvert_{\mathfrak{m}}= 2\xi_1\wedge \xi_2 +  e_4\wedge e_7 +e_5\wedge e_6,\\
		\ad(kF_{1,1}^{(2)}, k)\rvert_{\mathfrak{m}}=& 2\xi_1\wedge \xi_3 -e_4\wedge e_6 +e_5\wedge e_7.
	\end{align*}
A calculation similar to Example \ref{sp2sp1} then gives the result.
\end{proof}
In order to differentiate the spinor $\psi_0$, we remark the commutator relations	
\[
[\mathcal{V},\mathcal{V}] \subseteq  \mathcal{V}\oplus (\mathfrak{sp}(1)\oplus \mathfrak{sp}(1)) ,\qquad [\mathcal{V},\mathcal{H}]\subseteq \mathcal{H},\qquad [\mathcal{H},\mathcal{H}] \subseteq \mathcal{V}\oplus (\mathfrak{sp}(1)\oplus \mathfrak{sp}(1)) , 
\]
and one then finds that the Nomizu map for the Levi-Civita connection is given by
\[\Uplambda^{g_{a,b}}(V)W = \begin{cases}
\frac{1}{2}[V,W]_{\mathfrak{m}} & V,W\in\mathcal{V}, \\
(1-\frac{a}{2b}) [V,W]_{\mathfrak{m}} & V\in \mathcal{V}, W\in\mathcal{H},\\
\frac{a}{2b}[V,W]_{\mathfrak{m}} & V\in\mathcal{H}, W\in\mathcal{V},\\
\frac{1}{2}[V,W]_{\mathfrak{m}} & V,W\in\mathcal{H}.
\end{cases}  
\]
In terms of our chosen basis, for $\mathfrak{m}_1$ this takes the form
\begin{align*}
\Uplambda^{g_{a,b}}(\xi_1)&= \frac{1}{\sqrt{30a}}(-\frac{1}{2}\xi_2\wedge \xi_3  +   (1-\frac{a}{2b})e_4\wedge e_5 +(1-\frac{a}{2b})e_6\wedge e_7 ),\\
\Uplambda^{g_{a,b}}(\xi_2)&=\frac{1}{ \sqrt{30a}}(\frac{1}{2}\xi_1\wedge \xi_3 + (1-\frac{a}{2b})e_4\wedge e_6 -(1-\frac{a}{2b})e_5\wedge e_7     ), \\
\Uplambda^{g_{a,b}}(\xi_3) &= \frac{1}{\sqrt{30a} }(-\frac{1}{2}\xi_1\wedge \xi_2 +(1-\frac{a}{2b}) e_4\wedge e_7 +(1-\frac{a}{2b})e_5\wedge e_6       )  ,
\end{align*}
and likewise for $\mathfrak{m}_2$,
\begin{align*}
\Uplambda^{g_{a,b}}(e_4)&= \frac{\sqrt{a}}{2b\sqrt{30}}(- \xi_1\wedge e_5 - \xi_2\wedge e_6 -\xi_3\wedge e_7 ), \quad 
\Uplambda^{g_{a,b}}(e_5)= \frac{\sqrt{a}}{2b\sqrt{30}}( \xi_1\wedge e_4 +\xi_2\wedge e_7 - \xi_3\wedge e_6   ) , \\
\Uplambda^{g_{a,b}}(e_6) &= \frac{\sqrt{a}}{2b\sqrt{30}}( -\xi_1\wedge e_7 +\xi_2\wedge e_4 +\xi_3\wedge e_5    ), \quad
\Uplambda^{g_{a,b}}(e_7) = \frac{\sqrt{a}}{2b \sqrt{30} }(\xi_1\wedge e_6 -\xi_2\wedge e_5 +\xi_3\wedge e_4  ).
\end{align*}
Applying the spin lifts of these operators to $\psi_0$ gives:
\begin{proposition} \label{sp1twistedGKS}
	The spinor $\psi_0$ is a generalized Killing spinor, $\nabla_X^{g_{a,b}} \psi_0=A(X)\cdot \psi_0$, for the endomorphism 
	\begin{align*}
	A= \frac{(2 a-5 b)}{4b \sqrt{30a} } \ \text{\emph{Id}}\rvert_{\mathfrak{m}_1}  -\frac{\sqrt{3a}}{4b \sqrt{10}} \ \text{\emph{Id}}\rvert_{\mathfrak{m}_2},
	\end{align*}	
and it is a Riemannian Killing spinor if and only if $a=b$ ($\iff$ $g_{a,b}$ is a multiple of the second Einstein metric).
\end{proposition}
As in Remark \ref{spn_second_einstein_metric}, we briefly discuss the second Einstein metric in this case:
\begin{remark}
	Here, the $3$-Sasakian metric (the round metric) is given by $g_{a,b}\rvert_{a=\frac{5}{24},b=\frac{1}{24}}$ and the second Einstein metric is given by rescaling by $\frac{1}{5}$ on the vertical component, yielding the normal homogeneous metric $g_{a,b}\rvert_{a=b=\frac{1}{24}}$. From Theorem \ref{example_invariant_spinors_sp_twisted_7sphere} and Proposition \ref{sp1twistedGKS}, we see that the $1$-dimensional space of invariant spinors in this example is spanned by the Killing spinor determining the proper nearly parallel $\G_2$-structure on $(S^{7}=\frac{\Sp(2)\Sp(1)}{\Sp(1)\Sp(1)}, g_{a,b}\rvert_{a=b})$.
\end{remark}
We now recall (see e.g. \cite{nearly_parallel_g2,dim67}) the well-known fact that a unit length spinor $\psi$ in dimension 7 induces a $\G_2$-structure, via the 3-form 
\begin{align}\label{G2formfromspinorindim7}
\omega_{\psi}(X,Y,Z) := \langle X\cdot Y\cdot Z\cdot \psi,\psi\rangle  .
\end{align}
In particular, by comparing Proposition \ref{sp1twistedGKS} with \cite[Table 6]{SRNI} and \cite[Lemma 4.5]{dim67}, one sees that the $\G_2$-structure on $(S^7=\frac{\Sp(2)\Sp(1)}{\Sp(1)\Sp(1)},g_{a,b})$ induced by $\psi_0$ is cocalibrated for all $a,b>0$ and nearly parallel when $a=b$.    
\begin{proposition}
The $\G_2$-form induced by $\psi_0$ is given with respect to our chosen orthonormal basis by 
%
%
%
%
%
%
%
%
%
%
\[
\omega_{\psi_0} = -e_{123} + e_{145} + e_{167} + e_{246} - e_{257} +e_{347} +e_{356} =\omega_{\psi_0} = -\eta_1\wedge \eta_2\wedge \eta_3 - \sum_{i=1}^3\eta_i\wedge \Phi_i\rvert_{\mathcal{H}}
\]
and is invariant.
\end{proposition}
Proposition 4.4 in \cite{dim67} says that the intrinsic torsion of this $\G_2$-structure is given by $\Gamma = -\frac{2}{3} A \lrcorner \omega_{\psi_0}$, where the contraction of an endomorphism into a $3$-form means the $3$-form composed with the endormorphism in the first argument: $(A\lrcorner \omega_{\psi_0})(X,Y,Z):= \omega_{\psi_0}(A(X),Y,Z)$. Considering the $\G_2$-connection $\nabla^n:= \nabla^{g_{a,b} } - \Gamma$, one easily verifies $\nabla^n \psi_0 =0$, as expected. We also find,
\begin{proposition}
The torsion of the $\G_2$-connection $\nabla^n$ is a 3-form if and only if $a=b$. In this case it is given by a multiple of the $\G_2$-form,
\[
T= \frac{-1}{\sqrt{30a}}\ \omega_{\psi_0} , 
\]
and thus in particular is invariant and $\nabla^n$-parallel.
\end{proposition}
\begin{proof}
From $\nabla^n =\nabla^g -\Gamma$ we see that the difference tensor between $\nabla^n$ and $\nabla^{g_{a,b}}$ is $-\Gamma = \frac{2}{3}A\lrcorner \omega_{\psi_0}$, and thus the torsion is a 3-form if and only if $(X,Y,Z)\mapsto \frac{2}{3}\omega_{\psi_0}(A(X),Y,Z)$ is totally skew-symmetric. Clearly this happens precisely when $A$ is a multiple of the identity, which occurs if and only if $a=b$. When $a=b$ we have $A=\frac{-\sqrt{3}}{4\sqrt{10a}} \Id $ and thus the torsion of $\nabla^n$ is given by $$T = -2\Gamma  = -2(-\frac{2}{3}A\lrcorner \omega_{\psi_0}) = -\frac{1}{\sqrt{30a}} \omega_{\psi_0}. $$  
\end{proof}
\begin{remark}
From \cite{stringtheoryparallelspinors} it is known that the characteristic connection $\nabla^c$ (which exists in this example since the $\G_2$-structure is cocalibrated) is unique, so when $a=b$ it coincides with $\nabla^n$. When $a\neq b$ we need to find a different way to describe $\nabla^c$. Using Theorem 4.8 in \cite{stringtheoryparallelspinors}, and taking into account that our $\G_2$-structure is cocalibrated and differs from the reference $\G_2$-form in \cite{stringtheoryparallelspinors} by an orientation-reversing change of basis, the torsion 3-form of $\nabla^c$ is given by
\begin{align}
T^c &= -\frac{1}{6}\langle d\omega_{\psi_0},\ast\omega_{\psi_0}\rangle \omega_{\psi_0} + \ast d\omega_{\psi_0} . \label{chartorsionform}
\end{align}
\end{remark}
%
%
%
%
%
In order to compute $T^c$ using (\ref{chartorsionform}) we first prove the following lemma:
\begin{lemma}
The Hodge dual and exterior derivative of $\omega_{\psi_0}$ are given by
\begin{enumerate}[(i)]
\item $\ast\omega_{\psi_0} = -e_{4567} +e_{2367} +e_{2345} +e_{1357} -e_{1346} +e_{1256} +e_{1247},  $ 
\item $ d\omega_{\psi_0} = \frac{a+5b}{b\sqrt{30a}} (\ast \omega_{\psi_0}) + \frac{-5a+5b}{b\sqrt{30a}} e_{4567}.  $
\end{enumerate}
\end{lemma}
\begin{proof}
	The first claim is straightforward. The proof of the second claim proceeds by a lengthy yet straightforward computation using the above expression of the Nomizu map $\Uplambda^{g_{a,b}} $ together with the identity $d = \sum_{i=1}^7 e_i\wedge \nabla_{e_i}^{g_{a,b}} $.
\end{proof}
Recalling (e.g. from \cite[Table 6]{SRNI}) that a $\G_2$-structure $\omega$ is nearly parallel if and only if $d\omega$ is a non-zero scalar multiple of $\ast \omega$, the preceding lemma also gives an alternate proof of the fact that the $\G_2$-structure induced by $\psi_0$ is nearly parallel if and only if $a=b$. 
\begin{corollary}
The characteristic connection $\nabla^c$ is given by
\begin{align*}
\nabla^c &= \nabla^{g_{a,b}} -\frac{1}{2}\left( \frac{a}{b\sqrt{30a}} \omega_{\psi_0} +\frac{5a-5b}{b\sqrt{30a}} e_{123}\right) ,
\end{align*}
and the defining 3-form $\omega_{\psi}$ and defining spinor $\psi_0$ are both $\nabla^c$-parallel.
\end{corollary}
\begin{proof}
Using (\ref{chartorsionform}) and the preceding lemma, we calculate
{\small 
	\begin{align*}
		T^c &= -\frac{1}{6} \langle \frac{a+5b}{b\sqrt{30a}} (\ast \omega_{\psi_0}) + \frac{-5a+5b}{b\sqrt{30a}} e_{4567}, \ast \omega_{\psi_0}\rangle \omega_{\psi_0} + \ast\left( \frac{a+5b}{b\sqrt{30a}} (\ast \omega_{\psi_0}) + \frac{-5a+5b}{b\sqrt{30a}} e_{4567}\right)  \\
		& = \left( \frac{-7(a+5b)}{6b\sqrt{30a}} +   \frac{-5a+5b}{6b\sqrt{30a}} \right) \omega_{\psi_0}         +  \left(  \frac{a+5b}{b\sqrt{30a}} \omega_{\psi_0} +\frac{-5a+5b}{b\sqrt{30a}} e_{123}\right)  \\
		&= -\frac{a}{b\sqrt{30a}} \omega_{\psi_0} -\frac{5a-5b}{b\sqrt{30a}} e_{123}.
	\end{align*}
}The fact that $\nabla^c\omega = 0= \nabla^c\psi_0$ follows since $\nabla^c$ is a $\G_2$-connection.
\end{proof}
A natural question is whether there are other invariant differential forms. To that end, one finds:
\begin{proposition}\label{sp2sp1_inv_forms_statement}
	Up to taking Hodge duals, the invariant differential forms on $(S^7=\frac{\Sp(2)\Sp(1)}{\Sp(1)\Sp(1)},g_{a,b})$ are given in Table \ref{Tab:inv_forms_S7_Sp2Sp1}.
\end{proposition}
\begin{proof}
	By hand, or using the LiE computer algebra package (\cite{LiE}), one finds that the dimensions of the spaces of invariant $1$-, $2$-, and $3$-forms on $\Sp(2)/\Sp(1)$ are $3$, $6$ and $10$ respectively. Bases for these spaces are given by $\{ \eta_i\}_{i=1}^3 $, $\{ \Phi_i\rvert_{\mathcal{V}} , \Phi_i\rvert_{\mathcal{H}}  \}_{i=1,2,3} $, and $\{\eta_1\wedge \eta_2\wedge \eta_3\} \cup \{ \eta_i\wedge \Phi_j\rvert_{\mathcal{H}} \}_{i,j=1,2,3} $ respectively. The result then follows by checking which elements in these spans are invariant under the additional three isotropy operators described in the proof of Theorem \ref{example_invariant_spinors_sp_twisted_7sphere}.
\end{proof}
\begin{table}[h!] 
	\centering
	\caption{Invariant Differential Forms on $(S^7=\frac{\Sp(2)\Sp(1)}{\Sp(1)\Sp(1)},g_{a,b})$}
	\begin{tabular}{|l||l|l|}\hline
		$k$  & $\dim \Lambda^k_{\inv}$ & Basis for $ \Lambda^k_{\inv}$ \\\hline
		$0$    & $1$ & $1$  \\
		$1$ & $0$ & $0$  \\
		$2$ & $0$ & $0$ \\
		$3$   & $2$ & $\omega_{\psi_0}$, $\xi_{1,2,3}$   \\ \hline
	\end{tabular}
	\label{Tab:inv_forms_S7_Sp2Sp1}
\end{table}
Noting from the preceding proposition that there are no invariant $1$-forms, one obtains:
\begin{corollary}\label{S7_sp2sp1_no_inv_sasakian_structure}
	The space $(S^7 = \frac{\Sp(2)\Sp(1)}{\Sp(1)\Sp(1)}, g_{a,b}) $ does not admit an invariant Einstein-Sasakian or 3-Sasakian structure for any values of $a,b>0$.
\end{corollary}
%
%
%
%
%
%
%
%
%
%
%
%
%
%
%
%
%
\end{example}
\subsection{$S^1$-Quaternionic Spheres, $S^{4n-1} = \frac{\Sp(n) \U(1)}{\Sp(n-1) \U(1)}$}
This is the case from Section \ref{caseIIK=U1} corresponding to $K=\U(1)$. Viewing $\U(1)$ as a subgroup of $\Sp(1)$ via $e^{i\theta}\mapsto \cos\theta + i\sin\theta +0j+0k$, and from (\ref{inclusion_isotropy_productgroup}), we have at the level of Lie algebras
\begin{align*}
	\mathfrak{sp}(n)\oplus \mathfrak{u}(1)&= \Span_{\R} \{ (iF_{p,q}^{(n)},0) , (jF_{p,q}^{(n)},0) ,(kF_{p,q}^{(n)},0) ,(E_{r,s}^{(n)},0), (0,i)     \}_{\substack{1\leq p\leq q \leq n \\ 1\leq r < s \leq n} } ,  \\
	\mathfrak{sp}(n-1)\oplus \mathfrak{u}(1)&=  \Span_{\R} \{ (iF_{p,q}^{(n)},0) , (jF_{p,q}^{(n)},0) ,(kF_{p,q}^{(n)},0) ,(E_{r,s}^{(n)},0), (iF_{1,1}^{(n)},i)     \}_{\substack{2\leq p\leq q \leq n \\ 2\leq r < s \leq n} }     .
\end{align*}
Note that the Killing form of $\mathfrak{sp}(n)\oplus \mathfrak{u}(1)$ fails to be non-degenerate, so in order to choose a reductive complement we instead take the orthogonal complement with respect to the restriction of the inner product $\kappa$ from (\ref{squashedKF}) to the subalgebra $\mathfrak{sp}(n)\oplus \mathfrak{u}(1)\subset \mathfrak{sp}(n)\oplus \mathfrak{sp}(1)$,
\begin{align*}
	\mathfrak{m}&:= (\mathfrak{sp}(n-1)\oplus \mathfrak{u}(1))^{\perp_{\kappa}}.
\end{align*}
The isotropy representation splits into one copy of the trivial representation and two non-isomorphic irreducible representations: 
\[\mathfrak{m}\simeq \mathfrak{m}_1\oplus \mathfrak{m}_2\oplus \mathfrak{m}_3,\]
where $\dim_{\R}\mathfrak{m}_1=1$, $\dim_{\R}\mathfrak{m}_2=2$, and $\dim_{\R}\mathfrak{m}_3 = 4(n-1)$. This gives a 3-parameter family of invariant metrics,
\begin{align*}
g_{a,b,c}&:= -a \kappa\rvert_{\mathfrak{m}_1\times \mathfrak{m}_1} - b \kappa\rvert_{\mathfrak{m}_2\times \mathfrak{m}_2} -c\kappa\rvert_{\mathfrak{m}_3\times \mathfrak{m}_3} , \qquad a,b,c>0. \end{align*}
In particular, by manually checking the sectional curvatures, one finds that the round metric is given by the parameters
\[
a=\frac{n+3}{8(n+1)}, \quad b= \frac{1}{4(n+1)}, \quad c= \frac{1}{8(n+1)}.
\]
For general $a,b,c>0$, a $g_{a,b,c}$-orthonormal basis for $\mathfrak{m} $ is given by
\begin{align*}
\xi_1&:= \frac{1}{\Omega}\left( iF_{1,1}, -\left( \frac{n+1}{2}\right) i\right), \quad
\xi_2:=\frac{1}{2\sqrt{b(n+1)}}\left( -kF_{1,1}, 0\right) ,\quad 
\xi_3:=\frac{1}{2\sqrt{b(n+1)}}\left(jF_{1,1}, 0\right) ,\\
e_{4p}&:= \frac{1}{2\sqrt{2c(n+1)}}(jF_{1,p+1} ,0),\quad
e_{4p+1}:=  \frac{1}{2\sqrt{2c(n+1)}}(kF_{1,p+1} ,0),\\
e_{4p+2}&:= \frac{1}{2\sqrt{2c(n+1)}}(iF_{1,p+1} ,0),\quad
e_{4p+3}:= \frac{1}{2\sqrt{2c(n+1)}}(E_{1,p+1} ,0),
\end{align*}
for $1\leq p \leq n-1$, where $\Omega:=\sqrt{2a(n+1)(n+3)}$. In terms of this basis, the isotropy summands are
\begin{align*}
\mathfrak{m}_1 &= \Span_{\R} \{ \xi_1\} , \quad
\mathfrak{m}_2= \Span_{\R} \{ \xi_2,\xi_3\} ,\quad
\mathfrak{m}_3=  \Span_{\R} \{e_{4p},e_{4p+1},e_{4p+2},e_{4p+3}\}_{p=1}^{n-1}.
\end{align*}
From (\ref{invariantelementsu1twistedMA}) we obtain:
\begin{theorem}\label{explicitspinorsK=u1}
		Using the above orthonormal basis and the corresponding description of the spinor module from Remark \ref{reorderedCliffalgrepresentation}, the space of invariant spinors on $(S^{4n-1}=\frac{\Sp(n)\U(1)}{\Sp(n-1)\U(1)}, g_{a,b,c})$ for any $a,b,c>0$ is trivial unless $n$ is even. If $n$ is even,
\[
\Sigma_{\inv} = \Span_{\C}\{ \omega^{n/2}, \  y_1\wedge \omega^{(n-2)/2}\},
\]
where $\omega := \sum_{i=1}^{n-1} y_{2i} \wedge y_{2i+1}$. 
\end{theorem}
\begin{proof}
This follows directly from (\ref{invariantelementsu1twistedMA}). Alternatively, one can argue exactly as in the proof of Theorem \ref{example_invariant_spinors_sp_twisted_7sphere}, using only the spin lift of the operator
\begin{align}\label{u1extraisotropyoperator}
 \ad( iF_{1,1}^{(n)},i)\rvert_{\mathfrak{m}} = \xi_2\wedge \xi_3-\Phi_1 .
\end{align}
\end{proof}
In order to differentiate these spinors, we calculate:
\begin{lemma}
The Nomizu map for the Levi-Civita connection of $g_{a,b,c}$ is given by
\begin{align*}
\Uplambda^{g_{a,b,c}}(x_1)x_2 &=\frac{1}{2}[x_1,x_2]_\mathfrak{m},\quad
\Uplambda^{g_{a,b,c}}(x)y =(1-\frac{a}{2b})[x,y]_\mathfrak{m},\quad
\Uplambda^{g_{a,b,c}}(y)x = \frac{a}{2b} [y,x]_\mathfrak{m},\\
\Uplambda^{g_{a,b,c}}(y_1)y_2&= \frac{1}{2}[y_1,y_2]_\mathfrak{m},\quad
\Uplambda^{g_{a,b,c}}(x)z = (1-\frac{a}{2c})[x,z]_\mathfrak{m},\quad
\Uplambda^{g_{a,b,c}}(z)x = \frac{a}{2c} [z,x]_\mathfrak{m},\\
\Uplambda^{g_{a,b,c}}(z_1)z_2&=\frac{1}{2}[z_1,z_2]_\mathfrak{m},\quad
\Uplambda^{g_{a,b,c}}(y)z = (1-\frac{b}{2c}) [y,z]_\mathfrak{m},\quad
\Uplambda^{g_{a,b,c}}(z)y = \frac{b}{2c}[z,y]_\mathfrak{m},
\end{align*} 
for all $x,x_1,x_2\in\mathfrak{m}_1$, $y,y_1,y_2\in\mathfrak{m}_2$, $z,z_1,z_2\in\mathfrak{m}_3$. 
\end{lemma}
\begin{proof}
This is a straightforward calculation using the fact that $\kappa$ is $\ad$-invariant and the commutator relations
\begin{align*} 
[\mathfrak{m}_1,\mathfrak{m}_1]&=0,\quad  [\mathfrak{m}_1,\mathfrak{m}_2]\subseteq \mathfrak{m}_2, \quad  [\mathfrak{m}_1,\mathfrak{m}_3]\subseteq \mathfrak{m}_3, \quad  [\mathfrak{m}_2,\mathfrak{m}_2]\subseteq \mathfrak{m}_1\oplus (\mathfrak{sp}(n-1)\oplus\mathfrak{u}(1)), \\
[\mathfrak{m}_2,\mathfrak{m}_3]&\subseteq \mathfrak{m}_3,\quad  [\mathfrak{m}_3,\mathfrak{m}_3]\subseteq \mathfrak{m}_1\oplus\mathfrak{m}_2\oplus (\mathfrak{sp}(n-1)\oplus\mathfrak{u}(1)) .
\end{align*}
\end{proof}
\begin{remark} \label{squashedunitaryNomizumap}
In terms of our chosen basis, the operators in the preceding lemma take the form
\begin{align*}
\Uplambda^{g_{a,b,c}}(\xi_1) &= \frac{2(1-\frac{a}{2b})}{\Omega} \  \xi_2\wedge \xi_3 + \frac{(1-\frac{a}{2c})}{\Omega} \sum_{p=1}^{n-1}(e_{4p}\wedge e_{4p+1} + e_{4p+2}\wedge e_{4p+3}  ),\\
\Uplambda^{g_{a,b,c}}(\xi_2)&= \frac{a}{b \Omega }\  \xi_3\wedge \xi_1 + \frac{(1-\frac{b}{2c})}{2\sqrt{b(n+1)}} \sum_{p=1}^{n-1} (e_{4p}\wedge e_{4p+2} - e_{4p+1}\wedge e_{4p+3}   ), \\
\Uplambda^{g_{a,b,c}}(\xi_3)&= \frac{a}{b\Omega }\  \xi_1\wedge \xi_2 + \frac{(1-\frac{b}{2c})}{2\sqrt{b(n+1)}} \sum_{p=1}^{n-1} (e_{4p}\wedge e_{4p+3} + e_{4p+1}\wedge e_{4p+2}   ) ,
\end{align*}
in the vertical directions and, in the horizontal directions,
\begin{align*}
\Uplambda^{g_{a,b,c}}(e_{4p})&= -\frac{a}{2c\Omega } \ \xi_1\wedge e_{4p+1} - \frac{\sqrt{b}}{4c\sqrt{n+1}}( \xi_2\wedge e_{4p+2} +    \xi_3\wedge e_{4p+3}),\\
\Uplambda^{g_{a,b,c}}(e_{4p+1})&= \frac{a}{2c\Omega } \ \xi_1\wedge e_{4p} + \frac{\sqrt{b}}{4c\sqrt{n+1}}( \xi_2\wedge e_{4p+3} -\xi_3\wedge e_{4p+2}   ), \\
\Uplambda^{g_{a,b,c}}(e_{4p+2}) &= -\frac{a}{2c\Omega } \ \xi_1\wedge e_{4p+3} + \frac{\sqrt{b}}{4c\sqrt{n+1}}(  \xi_2\wedge e_{4p} +\xi_3\wedge e_{4p+1}  ) ,\\
\Uplambda^{g_{a,b,c}}(e_{4p+3}) &= \frac{a}{2c\Omega } \ \xi_1\wedge e_{4p+2} - \frac{\sqrt{b}}{4c\sqrt{n+1}} ( \xi_2\wedge e_{4p+1} - \xi_3\wedge e_{4p}   ).
\end{align*}
\end{remark}
In odd dimensions, almost contact structures are natural candidates for interesting geometries; in the following proposition we study the existence of invariant almost contact structures for these spheres.
\begin{proposition} \label{u1twistedSasakian}
The sphere $(S^{4n-1}=\frac{\Sp(n) \U(1)}{\Sp(n-1) \U(1)}, g_{a,b,c})$ admits:
\begin{enumerate}[(i)] \item a compatible invariant normal almost contact metric structure for all $a,b,c >0$.
\item a compatible invariant $\alpha$-contact structure if and only if $\frac{a}{b\Omega} = \frac{a}{2c\Omega }=\alpha$.
\item a compatible invariant $\alpha$-K-contact structure if and only if $\frac{a}{b\Omega} = \frac{a}{2c\Omega }=\alpha$.
\end{enumerate}
In particular there exists a compatible invariant $\alpha$-Sasakian structure if and only if $\frac{a}{b\Omega} = \frac{a}{2c\Omega }=\alpha$.
\end{proposition}
\begin{proof}
Following a similar argument as in the proof of Proposition \ref{specialunitaryalphaSasakian}, the only choices for the Reeb vector field are $\xi:= \pm \xi_1$. We claim that the 2-form $\Phi:=g_{a,b,c}(\cdot, \varphi(\cdot))$ is invariant if and only if
\begin{align} \label{invariant2formspace}
\Phi \in (\Lambda^2 \mathfrak{m}_2)^{\Sp(n-1) \U(1)} \oplus (\Lambda^2\mathfrak{m}_3)^{\Sp(n-1) \U(1)}\simeq \Span_{\R} \{  \ad\xi\rvert_{\mathfrak{m}_2}  \} \oplus \Span_{\R} \{  \ad\xi\rvert_{\mathfrak{m}_3}  \} .
\end{align}
The proof of the claim consists of two parts: First, by noting that $\mathfrak{m}_1$, $\mathfrak{m}_2, \mathfrak{m}_3$ are irreducible and have pairwise distinct dimensions, we see that any invariant 2-form has trivial $\mathfrak{m}_1\otimes \mathfrak{m}_2$, $\mathfrak{m}_1\otimes \mathfrak{m}_3$, and $\mathfrak{m}_2\otimes \mathfrak{m}_3$ components. Second, we note that $\Lambda^2\mathfrak{m}_1=0$ (for dimension reasons) and that each $\Lambda^2 \mathfrak{m}_i$ $(i=2,3)$ has a real 3-dimensional space of $\Sp(n-1)$-invariant 2-forms, corresponding to the quaternionic structure 
\[
\mathcal{I}_i:= \Phi_1\rvert_{\mathfrak{m}_i}, \quad \mathcal{J}_i:=  \Phi_2\rvert_{\mathfrak{m}_i},\quad \mathcal{K}_i:= \Phi_3\rvert_{\mathfrak{m}_i}.
\]
Imposing the additional condition of $\U(1)$-invariance is equivalent to also requiring $\mathcal{I}$-complex linearity, giving $(\Lambda^2\mathfrak{m}_i)^{\Sp(n-1)\U(1)} \simeq \Span_{\R}\{\mathcal{I}_i\} = \Span_{\R} \{  \ad\xi\rvert_{\mathfrak{m}_i}  \}  $ for $i=2,3$.

Returning to the main proof, (\ref{invariant2formspace}) is equivalent to $\varphi = \lambda_1 \ad\xi\rvert_{\mathfrak{m}_2} \oplus \lambda_2 \ad\xi\rvert_{\mathfrak{m}_3} $ for some $\lambda_1,\lambda_2 \in\R$, and the metric compatibility condition $g_{a,b,c}(\varphi(X),\varphi(Y)) = g_{a,b,c}(X,Y)- g_{a,b,c}(\xi,X)g_{a,b,c}(\xi,Y)$ necessitates $\lambda_1 = \Omega/2$, $\lambda_2=\Omega $, i.e. 
\[
\varphi =\frac{\Omega }{2} \ad\xi\rvert_{\mathfrak{m}_2} \oplus \Omega \ad\xi\rvert_{\mathfrak{m}_3}.
\]
One then calculates that the Nijenhuis tensor vanishes for any values of $a,b,c$, and the structure is $\alpha$-contact ($d\eta = 2\alpha \Phi$) and $\alpha$-K-contact ($\nabla^g_X\xi = -\alpha\varphi(X)$) if and only if $\frac{a}{b\Omega}= \frac{a}{2c\Omega }=\alpha$. 
\end{proof}
Solving the equations in the preceding proposition for $a,b,c$, we immediately obtain:
\begin{corollary}\label{spnu1alphasasakianparamterization}
For each fixed $\alpha>0$, the invariant $\alpha$-Sasakian structures occur in a 1-parameter family:
\begin{align}
a=8\lambda^2\alpha^2(n+1)(n+3),\quad b=2\lambda, \quad c=\lambda, \qquad (\lambda>0). \label{invSasakianfamily} 
\end{align}
In particular, the round metric occurs for the parameter $\lambda=\frac{1}{8(n+1)}$.
\end{corollary}
In contrast to the case $G=\Sp(n)$, we shall see in the following remark that the relationship between invariant Einstein-Sasakian structures and invariant Killing spinors is more complicated for $G=\Sp(n)\U(1)$, stemming from the fact that this group has a non-trivial $2$-dimensional complex representation.
\begin{remark}
We remark that the above 2-form $\Phi$ takes the form (\ref{Phi1}). Therefore if an invariant Einstein-Sasakian structure exists then Proposition \ref{E1spinorbasis} and Theorem 1 in \cite{Fried90} imply that there are two linearly independent Killing spinors inside $E_1^- = \Span_{\C}\{1,\ y_1\wedge \omega^{n-1}\}$. If $n=1$ then $\dim_{\C} \Sigma =2$, and it follows that $\Sigma=E_1^-$ is spanned by two linearly independent (non-invariant) Killing spinors. In fact, for all $n\geq 1$, one sees from Theorem \ref{explicitspinorsK=u1} that the intersection of $E_1^-$ with the space of invariant spinors is trivial, so the Killing spinors spanning $E_1^-$ are not invariant. This contrasts with the behaviour in the cases $G=\Sp(n)$, $\SU(n+1)$, where the spaces $E_i^{\pm}$ had bases of invariant spinors. Indeed, in each of the three cases $G=\Sp(n)\U(1)$, $\Sp(n)$, $\SU(n+1)$ it is easy to see by arguing similarly as in the proof of \cite[Prop.\@ 7.1]{kath_Tduals} that there is a $G$-representation on $E_{i}^{\pm}$, and in the latter two cases we use the fact that $\dim_{\C} E_i^{\pm}\leq 2$ to conclude that the representation must be trivial (equivalently, the Killing spinors spanning $E_i^{\pm}$ are invariant). However, unlike the other two groups, $G=\Sp(n) \U(1)$ has a nontrivial 2-dimensional representation, which is what allows the Killing spinors spanning $E_1^-$ to be non-invariant in this case. This behaviour is somewhat surprising; one would intuitively expect the Killing spinors associated to an invariant Einstein-Sasakian structure to also be invariant.   
\end{remark}
\begin{theorem} \label{u1twistedGKS}
The sphere $(S^{4n-1}=\frac{\Sp(n) \U(1)}{\Sp(n-1) \U(1)}, g_{a,b,c})$ admits an invariant generalized Killing spinor if and only if $n=2$. If $n=2$ there exists a pair $\psi_0$, $\psi_1$ of linearly independent invariant generalized Killing spinors,
\[
\nabla_X^{g_{a,b,c}} \psi_i = A_i(X)\cdot \psi_i , \quad i=0,1, 
\]
for the endomorphisms
\begin{align*}
A_0&:=\frac{a}{2\Omega }(\frac{1}{c}-\frac{1}{b}) \Id \rvert_{\mathfrak{m}_1} + \left(\frac{a}{2b\Omega }- \frac{(1-\frac{b}{2c})}{2\sqrt{3b}}\right)\Id\rvert_{\mathfrak{m}_2} +  \left(  -\frac{a}{4c\Omega } -\frac{\sqrt{b}}{4c\sqrt{3}} \right)\Id\rvert_{\mathfrak{m}_3}  ,  \\
A_1&:= \frac{a}{2\Omega }(\frac{1}{c}-\frac{1}{b}) \Id \rvert_{\mathfrak{m}_1} + \left(\frac{a}{2b\Omega }+ \frac{(1-\frac{b}{2c})}{2\sqrt{3b}}\right)\Id\rvert_{\mathfrak{m}_2} +  \left(  -\frac{a}{4c\Omega } +\frac{\sqrt{b}}{4c\sqrt{3}} \right)\Id\rvert_{\mathfrak{m}_3}
\end{align*}
with at most three distinct eigenvalues.
\end{theorem}
\begin{proof}
Suppose there exists an invariant generalized Killing spinor $\nabla_X^{g_{a,b,c}}\psi=A(X)\cdot \psi$ and, using Theorem \ref{explicitspinorsK=u1}, write $\psi = \lambda_1\omega^{\frac{n}{2}} + \lambda_2 y_1\wedge \omega^{\frac{n-2}{2}}$ for some $\lambda_1,\lambda_2\in\C$. Let us examine the spin lift of the operator $\Uplambda^{g_{a,b,c}}(e_{4p})$. One calculates
\begin{align*}
\frac{1}{2}\xi_1\cdot e_{4p+1} \cdot \psi  &= \frac{i}{2}\left[ \lambda_1\left( \frac{n}{2}y_{2p+1}\wedge \omega^{\frac{n-2}{2}} -y_{2p}\wedge \omega^{\frac{n}{2}} \right) + \lambda_2\left(\frac{n-2}{2}y_1\wedge y_{2p+1}\wedge \omega^{\frac{n-4}{2}} -y_1\wedge y_{2p} \wedge \omega^{\frac{n-2}{2}}    \right)     \right] , \\
\frac{1}{2}\xi_2\cdot e_{4p+2}\cdot \psi &= \frac{1}{2}\left[ \lambda_1\left( \frac{n}{2}y_1\wedge y_{2p}\wedge \omega^{\frac{n-2}{2}} - y_1\wedge y_{2p+1} \wedge \omega^{\frac{n}{2}}   \right) + \lambda_2\left(  y_{2p+1}\wedge \omega^{\frac{n-2}{2}} - \frac{n-2}{2} y_{2p} \wedge \omega^{\frac{n-4}{2}} \right)   \right] ,  \\
\frac{1}{2}\xi_3\cdot e_{4p+3}\cdot \psi &= \frac{1}{2}\left[ \lambda_1\left( \frac{n}{2}y_1\wedge y_{2p} \wedge \omega^{\frac{n-2}{2}}+y_1\wedge y_{2p+1}\wedge \omega^{\frac{n}{2}}\right)+ \lambda_2\left( y_{2p+1}\wedge \omega^{\frac{n-2}{2}} +\frac{n-2}{2}y_{2p}\wedge \omega^{\frac{n-4}{2}}\right)     \right] ,
\end{align*}
and the explicit formula from Remark \ref{squashedunitaryNomizumap} then gives
\begin{align*}
\widetilde{\Uplambda^{g_{a,b,c}}}(e_{4p})\cdot \psi &=  -\left( \frac{ian\lambda_1}{8c\Omega } + \frac{\sqrt{b}\lambda_2}{4c\sqrt{n+1}} \right) y_{2p+1}\wedge \omega^{\frac{n-2}{2}} +   \left( \frac{ia\lambda_2}{4c\Omega } + \frac{n\lambda_1}{2} \right) y_1\wedge y_{2p} \wedge \omega^{\frac{n-2}{2}} \\
&\qquad + \left( \frac{ia\lambda_1}{4c\Omega }\right) y_{2p}\wedge \omega^{\frac{n}{2}} - \left( \frac{ia(n-2)\lambda_2}{8c\Omega } \right) y_1\wedge y_{2p+1}\wedge \omega^{\frac{n-4}{2}}  
\end{align*}
In order for the right hand side of this expression to be equal to the Clifford product of a (real) tangent vector with $\psi$, one sees that such a vector must be of the form $se_{4p} + te_{4p+1}$ for some $s,t\in \R$. Comparing with
\begin{align*} 
(se_{4p} +te_{4p+1})\cdot \psi &= \left( \frac{ins\lambda_1}{2} - \frac{nt\lambda_1}{2}  \right) y_{2p+1}\wedge \omega^{\frac{n-2}{2}} +(is\lambda_1 +t\lambda_1)y_{2p}\wedge \omega^{\frac{n}{2}}\\
&\qquad  + \left(-\frac{i(n-2)s\lambda_2}{2} +\frac{(n-2)t\lambda_2 }{2}\right) y_1\wedge y_{2p+1}\wedge \omega^{\frac{n-4}{2}} \\
&\qquad +  (-is\lambda_2- t\lambda_2)y_1\wedge y_{2p}\wedge \omega^{\frac{n-2}{2}} 
\end{align*} gives the necessary conditions
\begin{align*}
\left( \frac{ia\lambda_1}{4c\Omega }\right) y_{2p}\wedge \omega^{\frac{n}{2}} &= (is\lambda_1 +t\lambda_1)y_{2p}\wedge \omega^{\frac{n}{2}} , \\
\left( \frac{ia\lambda_2}{4c\Omega } + \frac{n\lambda_1}{2} \right) y_1\wedge y_{2p} \wedge \omega^{\frac{n-2}{2}} &= -(is\lambda_1 +t\lambda_1) y_1\wedge y_{2p}\wedge\omega^{\frac{n-2}{2}} . 
\end{align*}
If $n\neq 2$ then $y_{2p}\wedge \omega^{\frac{n}{2}}$ and $y_1\wedge y_{2p}\wedge \omega^{\frac{n-2}{2}} $ are non-zero (the latter is non-zero independent of $n$), and we have 
\[
\left( \frac{ia\lambda_1}{4c\Omega }\right) = -\left( \frac{ia\lambda_2}{4c\Omega } + \frac{n\lambda_1}{2} \right) .
\]
It follows by taking the real part of both sides that $\lambda_1=\lambda_2=0$, i.e. $\psi\equiv 0$ is trivial.
For $n=2$, lifting the operators in Remark \ref{squashedunitaryNomizumap} and applying them to the spinors 
\begin{align*}
\psi_0&:= \frac{1}{\sqrt{2}}(\omega +iy_1),\quad \psi_1:=\xi_1\cdot \psi_0 = \frac{1}{\sqrt{2}}(i\omega +y_1)
\end{align*}
gives the result.
\end{proof}
Next, we calculate the Ambrose-Singer torsion and determine its type:
\begin{proposition}
	For any $a,b,c>0$ the sphere $(S^{4n-1}=\frac{\Sp(n)\U(1)}{\Sp(n-1)\U(1)}, g_{a,b,c})$ has Ambrose-Singer torsion of type $\mathcal{T}_{\totallyskew}\oplus \mathcal{T}_{\CT}$, given by {\small
		\begin{align*}
			T^{\AS}(\xi_1,\xi_2)&= -\frac{2}{\Omega} \xi_3 , \quad T^{\AS}(\xi_1,\xi_3) =  \frac{2}{\Omega} \xi_2  , \quad  T^{\AS}(\xi_2,\xi_3) = - \frac{2a}{b\Omega}\xi_1 ,  \\
			T^{\AS}(\xi_1, - )\rvert_{\mathfrak{m}_3} &= \frac{1}{\Omega}\Phi_1\rvert_{\mathfrak{m}_3}, \quad T^{\AS}(\xi_2, - )\rvert_{\mathfrak{m}_3} = \frac{1}{2\sqrt{b(n+1)}}\Phi_2\rvert_{\mathfrak{m}_3}, \quad  T^{\AS}(\xi_3, - )\rvert_{\mathfrak{m}_3} = \frac{1}{2\sqrt{b(n+1)}}\Phi_3\rvert_{\mathfrak{m}_3},        \\
			T^{\AS}(e_{4p},e_{4q}) &= T^{\AS}(e_{4p+1},e_{4q+1}) = T^{\AS}(e_{4p+2},e_{4q+2}) = T^{\AS}(e_{4p+3},e_{4q+3})=0\\
			T^{\AS} (e_{4p},e_{4q+1}) &= - \frac{a\delta_{p,q}}{c\Omega } \xi_1 , \quad T^{\AS}(e_{4p},e_{4q+2}) = -\frac{\sqrt{b} \delta_{p,q} }{2c\sqrt{n+1}} \xi_2, \quad T^{\AS}(e_{4p},e_{4q+3}) = -\frac{\sqrt{b} \delta_{p,q} }{2c\sqrt{n+1}} \xi_3, \\
			T^{\AS}(e_{4p+1},e_{4q+2}) &= -\frac{\sqrt{b} \delta_{p,q} }{2c\sqrt{n+1}} \xi_3,\quad  T^{\AS} (e_{4p+1},e_{4q+3}) = \frac{\sqrt{b} \delta_{p,q} }{2c\sqrt{n+1}} \xi_2, \quad T^{\AS}(e_{4p+2},e_{4q+3}) = - \frac{a\delta_{p,q}}{c\Omega} \xi_1,  
		\end{align*}
	}
	for $p,q=1,\dots, n-1$, where $\Phi_1,\Phi_2,\Phi_3$ are defined formally as in (\ref{Phi1})-(\ref{Phi3}). The projection of $T^{\AS}$ onto $\mathcal{T}_{\totallyskew}$ is 
	\[
	T^{\AS}_{\totallyskew} = - \left( \frac{2a+4b}{3b\Omega}  \right)\xi_{1,2,3}   + \left(\frac{a+2c}{3c\Omega} \right) \xi_1\wedge \Phi_1\rvert_{\mathfrak{m}_3}
	+\left( \frac{b+2c}{6c\sqrt{b(n+1)}}  \right)\sum_{i=2}^3  \xi_i \wedge \Phi_i\rvert_{\mathfrak{m}_3}   ,
	\]
	with $T^{\AS} = T^{\AS}_{\totallyskew}$ if and only if $a=b=c$.
\end{proposition}
Before studying the situation in dimension $7$ in more detail, we conclude our discussion of the general case with a description of the invariant differential forms of degree $\leq 3$:
\begin{proposition}
	The invariant differential forms of degree less than or equal to $3$ on $(S^{4n-1}=\frac{\Sp(n)\U(1)}{\Sp(n-1)\U(1)}, g_{a,b,c})$ are given in Table \ref{Tab:inv_forms_SpnU1}.
\end{proposition}
\begin{proof}
	The result for degree $k=1$ is clear, and for $k=2$ it follows from the proof of Proposition \ref{u1twistedSasakian}. For $k=3$, one verifies that the forms given in Table  \ref{Tab:inv_forms_SpnU1} are precisely the $\Sp(n-1)$-invariant $3$-forms in $\Span_{\C}\{\xi_{1,2,3}, \xi_i\wedge \Phi_j\rvert_{\mathfrak{m}_3}\}_{i,j=1,2,3}$ which are annihilated by the additional operator (\ref{u1extraisotropyoperator}).  
\end{proof}
\begin{table}[h!] 
	\centering
	\caption{Invariant Forms of Low Degree on $(S^{4n-1}=\frac{\Sp(n)\U(1)}{\Sp(n-1)\U(1)},g_{a,b,c})$}
	\begin{tabular}{|l||l|l|}\hline
		$k$  & $\dim \Lambda^k_{\inv}$ & Basis for $ \Lambda^k_{\inv}$ \\\hline
		$0$    & $1$ & $1$  \\
		$1$ & $1$ & $\xi_1 $  \\
		$2$ & $2$ & $\Phi_1\rvert_{\mathfrak{m}_2}$, $\Phi_1\rvert_{\mathfrak{m}_3}$ \\
		$3$   & $4$ &  $\xi_{1,2,3}$, $\xi_1\wedge \Phi_1\rvert_{\mathfrak{m}_3}$, $(\xi_2\wedge \Phi_2\rvert_{\mathfrak{m}_3}+\xi_3\wedge \Phi_3\rvert_{\mathfrak{m}_3})$, $(\xi_2\wedge \Phi_3\rvert_{\mathfrak{m}_3}-\xi_3\wedge \Phi_2\rvert_{\mathfrak{m}_3})$  \\ \hline
	\end{tabular}
	\label{Tab:inv_forms_SpnU1}
\end{table}
In what follows, we examine more closely the generalized Killing spinors in dimension $7$ and compare with the known results for the round ($3$-Sasakian) metric and the second Einstein metric (see \cite{nearly_parallel_g2,3Sasdim7,GKSspheres}).
\begin{remark}\label{u1twistedroundmetric}
In the case of an invariant Sasakian metric, substituting (\ref{invSasakianfamily}) into the endomorphisms from the preceding theorem gives 
\begin{align*}
A_0 &= \frac{1}{2} \Id\rvert_{\mathfrak{m}_1\oplus \mathfrak{m}_2} + \frac{-\sqrt{6\lambda}-1}{2\sqrt{6\lambda}} \Id\rvert_{\mathfrak{m}_3},  \qquad	A_1= \frac{1}{2} \Id\rvert_{\mathfrak{m}_1\oplus \mathfrak{m}_2} + \frac{-\sqrt{6\lambda}+1}{2\sqrt{6\lambda}} \Id\rvert_{\mathfrak{m}_3} . 
\end{align*}
It is clear that $A_0$ is never a multiple of the identity map, and $A_1$ is a multiple of the identity if and only if $\lambda=1/{24}$, which occurs precisely when $g_{a,b,c}$ is the round metric (see Corollary \ref{spnu1alphasasakianparamterization}). When $\lambda=1/{24}$, we have
\begin{align*} A_0&=\frac{1}{2}\Id\rvert_{\mathfrak{m}_1\oplus\mathfrak{m}_2} -\frac{3}{2}\Id\rvert_{\mathfrak{m}_3}  ,  \qquad 	A_1= \frac{1}{2}\Id.
\end{align*}
In this case there are three linearly independent Killing spinors (the invariant Killing spinor $\psi_1$ and two non-invariant linearly independent Killing spinors inside $E_1^-$), and Theorem 6 in \cite{Fried90} then implies the existence of a $3$-Sasakian structure, as expected. Starting with the generalized Killing spinor $\psi_0$, the existence of a $3$-Sasakian structure also follows from \cite[Thm.\@ 4.10]{GKSspheres}, which notes moreover that $\psi_0$ corresponds to the canonical spinor described in \cite{3Sasdim7} for a general 3-Sasakian manifold of dimension $7$ (and we then see that $\psi_1=\xi_i\cdot \psi_0$ corresponds to one of the auxiliary spinors described therein).
\end{remark}
Finally, as in Remark \ref{spn_second_einstein_metric}, we comment on the second Einstein metric in this example: 
\begin{remark}
	Here, the second Einstein metric is given by $g_{a,b,c}\rvert_{a=\frac{1}{24}, b=\frac{1}{60} , c= \frac{1}{24}}$, and it is clear that this metric is not part of the family (\ref{invSasakianfamily}). Substituting these values of $a,b,c$ into the endomorphisms from Theorem \ref{u1twistedGKS} gives
\[
A_0 = -\frac{3}{2\sqrt{5}} \Id, \qquad A_1= -\frac{3}{2\sqrt{5}} \Id\rvert_{\mathfrak{m}_1} + \frac{13}{2\sqrt{5}} \Id\rvert_{\mathfrak{m}_2} + \frac{1}{2\sqrt{5}}\Id\rvert_{\mathfrak{m}_3},  
\]       
so $\psi_0$ is the Killing spinor determining the proper nearly parallel $\G_2$-structure and $\psi_1$ is a generalized Killing spinor with $3$ distinct eigenvalues.  
\end{remark}

\section{Exceptional Spheres}
\subsection{$S^6=\G_2/\SU(3)$}
The isotropy representation here is irreducible, so the only invariant metrics are obtained from multiples of the Killing form. For convenience we choose the invariant inner product $g=B_0$ on $\mathfrak{g}_2$, and consider the reductive complement $\mathfrak{m}:= (\mathfrak{su}(3))^{\perp}$ with respect to this inner product. Following the notation of Proposition \ref{onbasisg2su3}, we have
\begin{align*}
\mathfrak{g}_2&= \Span_{\R} \{ \nu_1,\dots, \nu_{14}   \}, \quad
\mathfrak{su}(3)= \Span_{\R} \{  \nu_1,\dots \nu_8 \},
\end{align*}
and for $\mathfrak{m}$ we take the $g$-orthonormal basis given by
\begin{align*}
\mathfrak{m}=\Span_{\R}\{ e_1,\dots ,e_6\}, \ \ \text{where } e_i:= \begin{cases} \nu_{8+i} & i= 1,2,4,6, \\ -\nu_{8+i} & i = 3,5   .   \end{cases}
\end{align*}
The motivation behind this choice of basis is as follows. It is well-known that $S^6=\G_2/\SU(3)$ admits a unique invariant nearly K\"{a}hler structure up to sign \cite{Almost_Complex_Structure_S6}, and a direct calculation (by hand or using computer algebra software) shows that the $1$-dimensional space of invariant $2$-forms is spanned by
\[
J= e_1\wedge e_2 + e_3\wedge e_4 + e_5\wedge e_6.
\]
In particular, the invariant almost complex structures are $\pm J$ and the given basis $\{e_i\}$ is \emph{adapted} to $J$ in the sense that $J(e_1)=e_2$, $J(e_3)=e_4$, $J(e_5)=e_6$. This property allows us to give an easy proof of the following theorem:
\begin{theorem} \label{g2su3theorem}
Using the above orthonormal basis and the corresponding description of the spinor module from Section \ref{understandingspinrep}, the space of invariant spinors on $(S^6 = \G_2/\SU(3), g)$ is given by
\begin{align*}
\Sigma_{\inv}&= \Span_{\C} \{ 1, y_1\wedge y_2\wedge y_3 \},
\end{align*}
and the spinors $\psi_{\pm}:= 1\pm y_1\wedge y_2\wedge y_3$ are Riemannian Killing spinors,
\[
\nabla_X^g \psi_{\pm} = \pm \frac{1}{2\sqrt{3}} X\cdot \psi_{\pm}.
\]
\end{theorem}
\begin{proof}
Since the orthonormal basis $\{e_i\}$ is adapted to the invariant almost complex structure $J$, it follows from Remark \ref{change_of_basis_effect} that $\Sigma \simeq \Lambda^{0,\bullet}\mathfrak{m}$ as complex representations of $\SU(3)$, and the space of invariants is therefore spanned by $1$ and the complex volume form $y_1\wedge y_2\wedge y_3$ (see e.g. \cite[Prop.\@ F.10]{FultonHarris1991}).
\end{proof}
\begin{remark}
The preceding theorem can also be proved directly, by performing a similar calculation as in Example \ref{sp2sp1}, using the operators
\begin{align*}
\ad(\nu_1)\rvert_{\mathfrak{m}} &= \frac{1}{2}(e_{1,2} - e_{3,4}), \quad \ad(\nu_2)\rvert_{\mathfrak{m}}= \frac{1}{2}(e_{3,5}+e_{4,6}), \quad 	\ad(\nu_3)\rvert_{\mathfrak{m}}= \frac{1}{2}( -e_{3,6} + e_{4,5}  ), \\ \ad(\nu_4)\rvert_{\mathfrak{m}} &= \frac{1}{2}(e_{1,6} - e_{2,5}  ), \quad
\ad(\nu_5)\rvert_{\mathfrak{m}} = \frac{1}{2}( -e_{1,5}-e_{2,6}  ), \quad \ad(\nu_6)\rvert_{\mathfrak{m}} = \frac{1}{2}(-e_{1,4} +e_{2,3}  ), \\ \ad(\nu_7)\rvert_{\mathfrak{m}} &= \frac{1}{2} (e_{1,3}+e_{2,4}  ), \quad \ad(\nu_8)\rvert_{\mathfrak{m}}= \frac{1}{2\sqrt{3}}(-e_{1,2}-e_{3,4} +2e_{5,6}   ).
\end{align*}
%
%
%
%
%
%
%
%
%
%
%
%
%
%
%
and
\begin{align*}
\Uplambda^g(e_1)&= \frac{1}{2\sqrt{3}}(e_{3,6}+e_{4,5} ),\quad 
\Uplambda^g(e_2)= \frac{1}{2\sqrt{3}}(e_{3,5}-e_{4,6}), \quad 
\Uplambda^g(e_3)= \frac{1}{2\sqrt{3}}(-e_{1,6}-e_{2,5} ) ,\\
\Uplambda^g(e_4)&= \frac{1}{2\sqrt{3}}(-e_{1,5}+e_{2,6}), \quad
\Uplambda^g(e_5)=\frac{1}{2\sqrt{3}}(e_{1,4}+e_{2,3}) ,\quad 
\Uplambda^g(e_6)=\frac{1}{2\sqrt{3}} (e_{1,3}-e_{2,4}).
\end{align*}
\end{remark}
\begin{remark} 
The invariant spinors $\psi_{\pm}$ induce a pair of almost complex structures $J_{\psi_{\pm}}$ via 
\[
J_{\psi_{\pm}}(X)\cdot (\psi_{\pm})_0 =  iX\cdot (\psi_{\pm})_0\quad  \text{for all } X\in TM,
\]
where $(\psi_{\pm})_0$ denotes the projection onto the even half-spinor module $\Sigma_0\subseteq \Sigma$ (see e.g. \cite[Chapter 5.2]{BFGK}); these are well-known to be strictly nearly K\"{a}hler \cite{Gray70_nearly_Kahler}, i.e. they satisfy $(\nabla_XJ_{\psi_{\pm}})X=0$ for all $X\in TM$. Indeed, $(\psi_+)_0=(\psi_-)_0=1$, and one finds by a direct calculation in the spin representation that $J_{\psi_{\pm}}$ coincide with the nearly K\"{a}hler form $J$ defined above. For a recent historical account on the nearly K\"{a}hler structure on $S^6$, we refer to \cite{ABF18_S6_nearly_Kahler_6_manifolds}; see also \cite{dim67}. 
\end{remark}	
Next, we calculate the Ambrose-Singer torsion and determine its type:
\begin{proposition}
	The sphere $(S^6=\G_2/\SU(3), g)$ has Ambrose-Singer torsion of type $\mathcal{T}_{\totallyskew}$ given by
\[
T^{\AS} = \frac{1}{\sqrt{3}} (-e_{1,3,6} - e_{1,4,5} - e_{2,3,5} + e_{2,4,6}).
\]
\end{proposition}
\subsection{$S^7=\Spin(7)/\G_2$}
The isotropy representation is again irreducible, so we use the invariant metric induced by the inner product $g=B_0$ on $\mathfrak{spin}(7)$, and choose the reductive complement $\mathfrak{m}:=(\mathfrak{g}_2)^{\perp}$ with respect to this inner product. Following the notation of Proposition \ref{onbasisspin7spin9}, we have
\begin{align*}
\mathfrak{spin}(7) &= \{\nu_1,\dots, \nu_{14}, \nu'_{15},\dots ,\nu'_{21}\}, \quad
\mathfrak{g}_2 = \{ \nu_1,\dots \nu_{14}\},
\end{align*}
and
\begin{align*}
\mathfrak{m}&= \{e_1,\dots, e_7\}, \ \ \text{where } e_i:= \nu_{14+i}.
\end{align*}
The following theorem describing the invariant spinors in this case may be compared to \cite[\S3 Case 4]{Wang}, where he gives an expression for the parallel spinor on a $\G_2$-manifold.
\begin{theorem}\label{G2Spin7_inv_spinors_theorem} Using the above orthonormal basis and the corresponding description of the spinor module from Remark \ref{reorderedCliffalgrepresentation}, the space of invariant spinors on $(S^7= \Spin(7)/\G_2, g)$ is given by
$$\Sigma_{\inv} = \Span_{\C}\{  -1+y_1\wedge y_2\wedge y_3 \}  , $$
and the spinor $\psi := -1+y_1\wedge y_2\wedge y_3$ is a Riemannian Killing spinor,
\[
\nabla_X^g \psi = \frac{\sqrt{3}}{4\sqrt{2}} X \cdot \psi.
\]
\end{theorem}
\begin{proof}
This follows from the same calculation as in Example \ref{sp2sp1}, using the operators {\small
\begin{align*}
\ad(\nu_1)\rvert_{\mathfrak{m}}&= \frac{1}{2}(e_{2,3}-e_{6,7}), \quad \ad(\nu_2)\rvert_{\mathfrak{m}}= \frac{1}{2}(-e_{2,4}-e_{3,5} ), \quad \ad(\nu_3)\rvert_{\mathfrak{m}}= \frac{1}{2}(-e_{2,5}+e_{3,4}),\\
\ad(\nu_4)\rvert_{\mathfrak{m}}&= \frac{1}{2}(e_{4,7}-e_{5,6}), \quad \ad(\nu_5)\rvert_{\mathfrak{m}} = \frac{1}{2}(-e_{4,6}-e_{5,7}), \quad \ad(\nu_6)\rvert_{\mathfrak{m}}= \frac{1}{2}( e_{2,7}-e_{3,6}),\\ 
\ad(\nu_7)\rvert_{\mathfrak{m}}&= \frac{1}{2}(-e_{2,6}-e_{3,7}),\quad \ad(\nu_8)\rvert_{\mathfrak{m}} = \frac{1}{2\sqrt{3}}(e_{2,3} -2e_{4,5} +e_{6,7}), \quad
\ad(\nu_9)\rvert_{\mathfrak{m}}= \frac{1}{2\sqrt{3}}( 2e_{1,7} -e_{2,5} -e_{3,4}), \\ \ad(\nu_{10})\rvert_{\mathfrak{m}} &= \frac{1}{2\sqrt{3}}( 2e_{1,6} +e_{2,4} -e_{3,5}), \quad
\ad(\nu_{11})\rvert_{\mathfrak{m}}= \frac{1}{2\sqrt{3}}( -2e_{1,3} +e_{4,7} +e_{5,6} ),\\ \ad(\nu_{12})\rvert_{\mathfrak{m}} &= \frac{1}{2\sqrt{3}}( 2e_{1,2} +e_{4,6} -e_{5,7}), \quad
\ad(\nu_{13})\rvert_{\mathfrak{m}} = \frac{1}{2\sqrt{3}}(  2e_{1,5} +e_{2,7}+e_{3,6} ), \\ \ad(\nu_{14})\rvert_{\mathfrak{m}}&= \frac{1}{2\sqrt{3}}(-2e_{1,4} +e_{2,6} -e_{3,7}   ),
\end{align*}
}
and
{\small
\begin{align*}
\Uplambda^g(e_1) &= \frac{1}{2\sqrt{6}}(e_{2,3}+e_{4,5}+e_{6,7}), \quad \Uplambda^g(e_2) = \frac{1}{2\sqrt{6}}(-e_{1,3}-e_{4,7} -e_{5,6}),\quad  \Uplambda^g(e_3) = \frac{1}{2\sqrt{6}}(e_{1,2}-e_{4,6} +e_{5,7}),\\
\Uplambda^g(e_4) &= \frac{1}{2\sqrt{6}}(-e_{1,5}+e_{2,7} +e_{3,6} ),\quad
\Uplambda^g(e_5) = \frac{1}{2\sqrt{6}}(e_{1,4}+e_{2,6} -e_{3,7}),\quad
\Uplambda^g(e_6) =\frac{1}{2\sqrt{6}}(-e_{1,7}-e_{2,5} -e_{3,4} ),\\
\Uplambda^g(e_7) &= \frac{1}{2\sqrt{6}}(e_{1,6} -e_{2,4} +e_{3,5} ).
\end{align*}
}
\end{proof}
%
%
%
%
%
%
%
%
%
%
%
Next, we calculate the Ambrose-Singer torsion and determine its type:
\begin{proposition}
	The sphere $(S^7=\Spin(7)/\G_2, g)$ has Ambrose-Singer torsion of type $\mathcal{T}_{\totallyskew}$ given by
	\[
	T^{\AS} = \frac{1}{\sqrt{6}} ( - e_{1,2,3} - e_{1,4,5} - e_{1,6,7} + e_{2,4,7} + e_{2,5,6} + e_{3,4,6} - e_{3,5,7}).
	\]
\end{proposition}
\subsection{$S^{15}=\Spin(9)/\Spin(7)$}\label{spin9spin7}
The isotropy representation in this case splits into two non-equivalent modules (one copy of the spin representation and one copy of the standard representation), hence there is a two-dimensional family of invariant metrics. Indeed, include $\mathfrak{spin}(7)\subseteq \mathfrak{spin}(9)$ as in Proposition \ref{onbasisspin7spin9} and set $\mathfrak{m}:=(\mathfrak{spin}(7))^{\perp}$, where orthogonality is taken with respect to the Killing form on $\mathfrak{spin}(9)$. Explicitly, in the notation of Proposition \ref{onbasisspin7spin9},
\begin{align*}
\mathfrak{spin}(9)&= \Span_{\R} \{\iota(\nu_1),\dots,\iota(\nu_{14}), \iota(\nu'_{15}),\dots, \iota(\nu'_{21}), \nu'_{22},\dots, \nu'_{36}     \}, \\
\mathfrak{spin}(7)&= \Span_{\R} \{ \nu'_{22},\dots, \nu'_{36} \},
\end{align*}
and
\begin{align*}
\mathfrak{m}&= \Span_{\R} \{ \widehat{e}_1,\dots, \widehat{e}_{15}\}, \ \ \text{where } \widehat{e}_i:= \nu'_{21+i}.
\end{align*}
The two irreducible isotropy summands are given by 
\begin{align*}
\mathfrak{m}_F&:= \Span_{\R} \{ \widehat{e}_1,\dots ,\widehat{e}_7\},\quad
\mathfrak{m}_B:= \Span_{\R} \{ \widehat{e}_8,\dots, \widehat{e}_{15} \},
\end{align*}
corresponding to the tangent spaces of the fiber and base respectively of the octonionic Hopf fibration,
\begin{align}\label{octhopf}
S^{7}=\frac{\Spin(8)}{\Spin(7)} \hookrightarrow S^{15}=\frac{\Spin(9)}{\Spin(7)} \to S^8=\frac{\Spin(9)}{\Spin(8)} .
\end{align}
The two-dimensional family of invariant metrics is parameterized by 
\[
g_{a,b}:= aB_0\rvert_{\mathfrak{m}_F\times \mathfrak{m}_F} + b B_0\rvert_{\mathfrak{m}_B\times \mathfrak{m}_B} , \quad a,b>0   
\]
and a $g_{a,b}$-orthonormal basis of $\mathfrak{m}$ is given by $\{e_i\}_{i=1}^{15}$, where 
\begin{align}\label{S15ONB}
e_i:= \begin{cases*}
\frac{1}{\sqrt{a}} \widehat{e}_i & if $i=1,\dots, 7$, \\
\frac{1}{\sqrt{b}}\widehat{e}_i   & if $i=8,\dots, 15$. 
\end{cases*}
\end{align}
By manually checking the sectional curvatures, one finds that the round metric corresponds to the parameters $a=\frac{1}{2}$, $b=\frac{1}{8}$ and we recall (see e.g. \cite{Ziller_homogeneous_einsten_metrics}) that the round metric is the only member of the family $g_{a,b}$ which can be isometric to one of the metrics for $G=\U(n+1)$, $\SU(n+1)$, $\Sp(n)$, or $\Sp(n)\Sp(1)$ considered in previous sections. In the following proposition we give a complete description of the invariant differential forms:
\begin{proposition}\label{S15invariantforms}
Up to taking Hodge duals, the invariant differential forms on $(S^{15}=\Spin(9)/\Spin(7), g_{a,b})$ are given in Table \ref{Tab:invariant_forms_S15}, where $\omega \in \Lambda^3\mathfrak{m}$, $\Psi\in\Lambda^4\mathfrak{m}$ and their exterior derivatives are defined by the formulas in Appendix \ref{AppendixB_forms}, and $\text{\emph{pr}}_{i,j}$ denotes the projection onto $\Lambda^i\mathfrak{m}_F\otimes \Lambda^j\mathfrak{m}_B$.
\end{proposition}
\begin{table}[h!] 
	\centering
	\caption{Invariant Differential Forms on $(S^{15}=\Spin(9)/\Spin(7), g_{a,b})$}
	\begin{tabular}{|l||l|l|}\hline
		$k$  & $\dim \Lambda^k_{\inv}$ & Basis for $ \Lambda^k_{\inv}$ \\ \hline
		$0$    & $1$ & $1$  \\
		$1$ & $0$ & $0$  \\
		$2$ & $0$ & $0$ \\
		$3$   & $1$ & $\omega $   \\
		$4$  & $2$   & $d\omega, \Psi$  \\
		$5$ &  $1$  &   $d\Psi$  \\
		$6$ &  $0$   &  $0$  \\
		$7$ &   $4$  &   $\text{pr}_{1,6}(\omega\wedge d\omega),\ \text{pr}_{3,4}(\omega\wedge d\omega),\ \ast(\omega \wedge d\Psi),\ \ast(\Psi\wedge \Psi )$   \\ \hline
	\end{tabular}
\label{Tab:invariant_forms_S15}
\end{table}
In particular, this shows that the invariant differential forms on $S^{15}=\Spin(9)/\Spin(7)$ are generated by $\omega$ and $\Psi$ and their derivatives, and in what follows we shall examine more closely the geometric and spinorial features of these forms. 
\begin{remark} The 4-form $\Psi$ is purely horizontal, i.e. $\Psi \in \Lambda^4\mb$, and the horizontal component of $d\omega$ is a multiple of $\Psi$. One finds that $\Psi$ is \emph{not} invariant for the larger group $\Spin(8)$, hence it does not descend to an invariant 4-form on the base space $S^8 = \Spin(9)/\Spin(8)$ of the octonionic Hopf fibration (\ref{octhopf}).
\end{remark}
\begin{remark}\label{Uwe2}
	Note that the existence of an invariant spinor and an invariant $4$-form on 
	$S^{15}$ is clear by the following argument. In this case the isotropy representation is realized as a direct sum of the standard
	and spin representations $\mathbb{R}^{15} \cong \mathbb{R}^7 \oplus \Delta_7$.
	This decomposition of $\Spin(7)$-representations is orthogonal with respect to the form. Thus it induces a homomorphism $\varphi: \Spin(7) \to \SO(7) \times \SO(8)$, which lifts to $\Spin(7) \to \Spin(7) \times \Spin(8) \subseteq \Spin(15)$.  
	
	It is known that there is an invariant vector in $\Lambda^4 \Delta_7$ (corresponding to the $4$-form $\Psi$ in Table \ref{Tab:invariant_forms_S15}). By the above, $\Lambda^4 \Delta_7 \subseteq \Lambda^4 \Delta_{15}$. Therefore there is also an invariant vector in $\Lambda^4 \Delta_{15} |_{\varphi}$, i.e., a four-form on $\mathbb{S}^{15}$ invariant under $\Spin(7)$.
	
	In this situation, $\Delta_{15}|_{\Spin(7) \times \Spin(8)} \cong \Delta_7 \otimes \Delta_8$. We now restrict this to a representation of $\Spin(7)$ via $\varphi$.  Note that $\varphi(g)=(g,\rho(g))$ where $\rho: \Spin(7) \to \Spin(8)$ is the lift of the spin representation to $\Spin(8)$.  Thanks to triality (see \cite{Wang}, (iii) on p.~61), we can decompose 
	\[
	\Delta_8|_{\rho} \cong (\Delta_8^+ \oplus \Delta_8^-)|_{\rho} \cong \Delta_7 \oplus \R^7 \oplus \R.
	\]
	Finally, tensoring by $\Delta_7$, we get:
	\[
	\Delta_{15}|_{\varphi} \cong (\Delta_7 \otimes \Delta_7) \oplus (\Delta_7 \otimes \R^7) \oplus \Delta_7.
	\]
	Since $\Delta_7, \R^7$, and $\R$ are real irreducible representations, a tensor product of two of these has an invariant vector if and only if the two are the same representation.  So there is a unique invariant spinor, corresponding to the invariant element of $\Delta_7 \otimes \Delta_7$ which expresses the invariant inner product on $\Delta_7$.
\end{remark}
In the following theorem we describe the space of invariant spinors:
\begin{theorem}
	Using the above $g_{a,b}$-orthonormal basis and the corresponding description of the spinor module from Remark \ref{reorderedCliffalgrepresentation}, the space of invariant spinors on $(S^{15}=\Spin(9)/\Spin(7),g_{a,b})$ is given by 
	\[
	\Sigma_{\inv} = \Span_{\C} \{ \psi \},
	\]
where
\[
	\psi := \frac{1}{2\sqrt{2}}( -iy_{1,5} + y_{1,2,3} +y_{2,5,7} -y_{3,5,6} +iy_{1,2,4,7} -iy_{1,3,4,6} -iy_{4,5,6,7} + y_{2,3,4,6,7}  ),
\]
and we use the notation $y_{i_1,\dots,i_p}:= y_{i_1}\wedge \dots \wedge y_{i_p}$. It satisfies the spinorial equation 
\begin{align} \label{spin9spin7spinorialeqn}
\nabla^g_{e_i} \psi = \begin{cases*}
 \left(\frac{a-2 b}{40 b \sqrt{2a} } \right)   \omega\cdot e_i\cdot \psi & if $i=1,\dots, 7$, \\ \left(  \frac{a+2 b}{16 b \sqrt{2}  (a-4 b)}\right) \omega \cdot e_i\cdot \psi +   \left( \frac{\sqrt{a}}{16 (a-4 b)} \right) d\omega \cdot e_i\cdot \psi    & if $i=8,\dots, 15$. 
\end{cases*}  
\end{align}
\end{theorem}
\begin{proof}
	Follows the same procedure as the proof of Theorem \ref{G2Spin7_inv_spinors_theorem}. To differentiate $\psi$, we note that the Nomizu map of the Levi-Civita connection is given by
	\[
	\Uplambda^g(X)Y  = \begin{cases*}
	\frac{1}{2}[X,Y]_{\mathfrak{m}} & if $X,Y\in\mf $, \\
	 (1-\frac{a}{2b}) [X,Y]_{\mathfrak{m}} & if $X\in\mf$, $Y\in\mb$, \\
	 \frac{1}{2b}  [X,Y]_{\mathfrak{m}} & if $X\in\mb$, $Y\in\mf$,\\
	 \frac{1}{2} [X,Y]_{\mathfrak{m}} & if $X,Y\in \mb$.
	\end{cases*}
	\] 
\end{proof}
The fact that the spinorial equation (\ref{spin9spin7spinorialeqn}) depends only on the invariant 3-form $\omega$ suggests that there is an intrinsic relationship between $\psi $ and $\omega$. Indeed, the spinor $\psi$ determines $\omega$ via the \emph{squaring construction}:
	\begin{align*}
		\omega(X,Y,Z) &= -2\langle (X\wedge Y \wedge Z) \cdot \psi, \psi \rangle \quad \text{for all } X,Y,Z\in TM,
	\end{align*}
where $\langle \cdot , \cdot \rangle$ denotes the usual Hermitian inner product on the spinor bundle. Conversely, (\ref{spin9spin7spinorialeqn}) shows that $\omega$ determines $\psi$ up to first order. In general, for each integer $k\geq 0$ the squaring construction determines an invariant $k$-form $\omega_{(k)}$ via 
\[
\omega_{(k)}(X_1,\dots,X_k) := \Re \langle (X_1\wedge \dots\wedge X_k)\cdot \psi,\psi\rangle \quad \text{for all } X_1,\dots X_k\in TM,
\]
and one can ask whether $\psi$ is related to other invariant forms from Table \ref{Tab:invariant_forms_S15} by this construction. We find:
\begin{proposition}
	Up to taking Hodge duals, the differential forms obtained from the invariant spinor $\psi$ via the squaring construction are given in Table \ref{Tab:S15squaringconstruction}.
\end{proposition}
	\begin{table}[h!] 
		\centering
		\caption{Forms on $(S^{15}=\Spin(9)/\Spin(7),g_{a,b})$ obtained from $\psi$ via the squaring construction}
		\begin{tabular}{|l||l|}\hline
			$k$  &  $ \omega_{(n)}$ \\ \hline
			$0$    & $1$  \\
			$1$ & $0$   \\
			$2$ & $0$  \\
			$3$   & $\omega$    \\
			$4$  & $-\frac{\sqrt{a}}{2\sqrt{2}} \text{pr}_{2,2}(d\omega)$    \\
			$5$ &  $0$   \\
			$6$ &  $0$   \\
			$7$ &   $-\frac{b\sqrt{2}}{18\sqrt{a}} \text{pr}_{1,6}(\omega\wedge d\omega) + \frac{\sqrt{a}}{6\sqrt{2}} \text{pr}_{3,4}(\omega\wedge d\omega) -\frac{\sqrt{a}}{8\sqrt{2}} \ast(\omega\wedge d\Psi) - \frac{1}{14} \ast(\Psi\wedge \Psi)$     \\ \hline
		\end{tabular}
	\label{Tab:S15squaringconstruction}
	\end{table}
\begin{proof}Note that invariance of $\psi$ implies that each $\omega_{(n)}$ is invariant too, and thus Proposition \ref{S15invariantforms} greatly limits the possibilities for $\omega_{(n)}$. One immediately sees that $\omega_{(1)}=\omega_{(2)}=\omega_{(6)} =0$, and the case $k=3$ is given above. Examining next the case of $4$-forms, we consider the projection 
\[
\text{pr}_{2,2}(d\omega) = d\omega + \frac{3\sqrt{a}}{b\sqrt{2}}\Psi
\]
of the invariant $4$-form $d\omega$ onto $\Lambda^2\mathfrak{m}_F\otimes \Lambda^2\mathfrak{m}_B$, and it is easy to see that $\{\Psi, \ \text{pr}_{2,2}(d\omega)\}$ is a basis for the space of invariant $4$-forms. Writing $\omega_{(4)}=\lambda_{1}\Psi + \lambda_{2}\text{pr}_{2,2}(d\omega)$, a straightforward calculation in the spin representation shows that \[\omega_{(4)}(e_8,e_9,e_{10},e_{11}) = 0 ,\qquad \omega_{(4)}(e_1,e_2,e_{8},e_{11})=-\frac{1}{2},\]
hence $\lambda_{1}=0$ and $\lambda_{2}=-\frac{\sqrt{a}}{2\sqrt{2}}$ (by comparing with the formulas in Appendix \ref{AppendixB_forms})
Similarly, for $5$-forms, one calculates 
\[
\omega_{(5)}(e_1,e_8,e_{10},e_{12},e_{15}) = 0 \neq \sqrt{\frac{2}{a}} = d\Psi (e_1,e_8,e_{10},e_{12},e_{15}),
\]
and we conclude from Table \ref{Tab:invariant_forms_S15} that $\omega_{(5)}=0$. In degree $7$, we write 
\[
\omega_{(7)}= \lambda_1 \text{pr}_{1,6}(\omega\wedge d\omega) + \lambda_2 \text{pr}_{3,4}(\omega\wedge d\omega) + \lambda_3(\ast(\omega \wedge d\Psi)) + \lambda_4(\ast(\Psi\wedge \Psi))  
\]
as a linear combination of the $7$-forms from Table \ref{Tab:invariant_forms_S15}. A straightforward calculation gives
\begin{align*}
	\omega_{(7)}(e_1,e_8,e_9,e_{10},e_{11},e_{12},e_{13}) &=  \frac{1}{2}, \qquad \omega_{(7)}(e_1,e_2,e_3,e_8,e_9,e_{10},e_{11})= 1\\
	\omega_{(7)}(e_3,e_4,e_5,e_6,e_7,e_{13},e_{14}) &= \frac{1}{2} , \qquad \omega_{(7)}(e_1,e_2,e_3,e_4,e_5,e_6,e_7) = -1 ,
\end{align*}
and it then follows from the formulas in Appendix \ref{AppendixB_forms} that
\[
\lambda_1 = -\frac{b\sqrt{2}}{18\sqrt{a}}, \quad \lambda_2 =\frac{\sqrt{a}}{6\sqrt{2}} , \quad \lambda_3= -\frac{\sqrt{a}}{8\sqrt{2}}, \quad \lambda_4=-\frac{1}{14} .
\]
For higher degrees, we note that the squaring construction is compatible with the Hodge star in the following sense: For any multi-index $I=\{i_1,\dots,i_k\}$, define a multi-vector $ e_I:=e_{i_1}\wedge \dots \wedge e_{i_k}$, and recall that $\ast e_I =\text{sign}(I\cup \widehat{I}) e_{\widehat{I}}$, where $\widehat{I}=\{1,2,\dots, 15\}\setminus I$ (with union and complement taken as ordered sets). Since $e_I\wedge (e_{\widehat{I}})=\text{sign}(I\cup \widehat{I}) e_1\wedge e_2\wedge \dots \wedge e_{15}$ acts on the spin representation by $\text{sign}(I\cup \widehat{I}) \Id$, we then have
\begin{align*}
\omega_{(k)}(e_I)&= \Re \langle   e_{I}\cdot \psi, \psi\rangle = \text{sign}(I\cup \widehat{I}) \ \Re \langle   e_{I}\cdot e_I\cdot e_{\widehat{I}}\cdot \psi, \psi\rangle = (-1)^{\frac{k(k+1)}{2}}\text{sign}(I\cup \widehat{I})  \ \Re \langle  e_{\widehat{I}}\cdot \psi, \psi\rangle \\
&= (-1)^{\frac{k(k+1)}{2}} \omega_{(15-k)}(\ast e_I),
\end{align*}
and it follows that $\omega_{(15-k)} = (-1)^{\frac{k(k+1)}{2}} \ast\omega_{(k)}$. In particular, the forms $\omega_{(k)}$ with $k\geq 8$ are determined by those in Table \ref{Tab:S15squaringconstruction}. 
\end{proof}
%
%
%
%
%
%
%
%
%
%
%
%
%
%
Finally, we calculate the Ambrose-Singer torsion and determine its type:
\begin{proposition}
For any $a,b>0$ the sphere $(S^{15}= \Spin(9)/\Spin(7), g_{a,b})$ has Ambrose-Singer torsion of type $\mathcal{T}_{\totallyskew}\oplus \mathcal{T}_{\CT}$, given by {\small
\begin{align*}
	T^{\AS}(e_i,-) &=\frac{1}{2\sqrt{2a}} e_i\lrcorner \omega \quad  (i=1,\dots, 7),\quad
	T^{\AS}(e_i,-) = \omega_i' \quad (i=8,\dots,15),  
\end{align*}
} 
where $\omega_i'$ are the $(0,2)$-tensors determined by 
\[
\omega_i'(e_j,e_k) = \begin{cases*}
\frac{1}{2\sqrt{2a}}  \omega(e_i,e_j,e_k)  & if $j<k$, \\ \frac{\sqrt{a}}{2b\sqrt{2}} \omega(e_i,e_j,e_k)  & if $j>k$, \\
0 & if $j=k$.
\end{cases*}
\] 
The projection of $T^{\AS}$ onto $\mathcal{T}_{\totallyskew}$ is 
\[
T^{\AS}_{\totallyskew} = \frac{(a+2b)}{6b\sqrt{2a}} \omega,
\]
with $T^{\AS}= T^{\AS}_{\totallyskew}$ if and only if $a=b$. 
\end{proposition}

%
%
%
%
%
%
%
%
%
%
\section{Generalized Killing Spinors on Round Spheres}
In this section we revisit our findings for the round metric in each case, and compare with the known results of Moroianu and Semmelmann \cite{GKSEinstein,GKSspheres}. In Section 4 of \cite{GKSspheres} it is shown that, for the round sphere $S^n$ ($n\geq 3$), if there exists a generalized Killing spinor with exactly two eigenvalues, then those eigenvalues are equal to $\frac{1}{2}$, $-\frac{3}{2}$ (up to a change of orientation) and they occur with multiplicity $m_{\frac{1}{2}}$ and $m_{-\frac{3}{2}}= m_{\frac{1}{2}} +1$ respectively. Furthermore, they show that if such a spinor exists, then $n=3$ or $7$, giving the multiplicities $(m_{\frac{1}{2}},m_{-\frac{3}{2}})=(1,2)$ or $(3,4)$. In what follows we examine each of the 3- and 7-dimensional cases from our classification and determine whether, when equipped with the round metric, they admit invariant generalized Killing spinors with exactly two eigenvalues.  
\begin{enumerate}[(I).]
	\item $G=\SO(4)$, $\SO(8)$, $\U(2)$, or $\U(4)$. By Theorems \ref{SOtheorem} and \ref{Utheorem} there are no invariant spinors to consider in these cases.
	\item $G=\SU(2)\cong \Sp(1)$. By Theorem \ref{deformedSasakianinvariantspinors} and Corollary \ref{sunweakprop}, the round metric $g_{a,b}\rvert_{a= b=\frac{1}{2}}$ admits a pair of invariant Killing spinors for the constant $\frac{1}{2}$, but no invariant generalized Killing spinors.
	\item $G=\Sp(1)\Sp(1)$. By Theorem \ref{explicitspinorsK=sp1} there are no invariant spinors to consider in this case.
	\item $G=\Sp(1)\U(1)$. By Theorem \ref{explicitspinorsK=u1} there are no invariant spinors to consider in this case.
	\item $G=\Sp(2)$. Considering Example \ref{sp2sp1} with the round metric $g_{\alpha,\delta}\rvert_{\alpha=\delta=1}$, one recovers the canonical spinor described in \cite{3Sasdim7, 3str}, consistent also with Theorem 4.10 in \cite{GKSspheres}.
	\item $G=\Sp(2)\Sp(1)$. The round metric in this case is given by $g_{a,b}\rvert_{a=\frac{5}{24}, \ b=\frac{1}{24}}$. Substituting these values of $a$ and $b$ into the endomorphism $A$ from Proposition \ref{sp1twistedGKS} shows that the spinor $\psi_0$ from Theorem \ref{example_invariant_spinors_sp_twisted_7sphere} is an invariant generalized Killing spinor with two distinct eigenvalues, and the associated endomorphism is consistent with Theorem 4.10 in \cite{GKSspheres}; this theorem also implies the existence of a 3-Sasakian structure with $\psi_0$ as the canonical spinor, however this structure cannot be invariant as a consequence of Corollary \ref{S7_sp2sp1_no_inv_sasakian_structure}.
	\item $G=\Sp(2)\U(1)$. The round metric in this case is given by $g_{a,b,c}\rvert_{a=\frac{5}{24}, \ b= \frac{1}{12}, \ c=\frac{1}{24}}$, and by Corollary \ref{spnu1alphasasakianparamterization} there is a compatible invariant Sasakian structure. By Remark \ref{u1twistedroundmetric}, the spinor $\psi_0$ from Theorem \ref{u1twistedGKS} is an invariant generalized Killing spinor with two distinct eigenvalues, and the associated endomorphism is consistent with Theorem 4.10 in \cite{GKSspheres}; this theorem also implies the existence of a 3-Sasakian structure with $\psi_0$ as the canonical spinor, however this structure cannot be invariant since the space of invariant vectors is 1-dimensional.
	\item $G=\Spin(7)$. In this case, for any invariant metric, the 1-dimensional space of invariant spinors consists of Killing spinors. 
\end{enumerate}
%
%
%
%
%
%
%
%
%
%
\appendix
\section{An Explicit Construction of the Lie Algebras $\mathfrak{su}(3)$, $\mathfrak{g}_2$, $\mathfrak{spin}(7)$, and $\mathfrak{spin}(9)$}
The exceptional cases in Table \ref{Tab:homogeneousspheres} can be greatly simplified by constructing in a unified manner the Lie algebra inclusions $\mathfrak{su}(3)\subset \mathfrak{g}_2 \subset \mathfrak{spin}(7) \subset \mathfrak{spin}(9)$. This Appendix is devoted to the exposition of this construction, parts of which may be found in Chaper 4.4 of \cite{BFGK}:
\begin{lemma}
	(Based on Chapter 4.4, Lemma 15 in \cite{BFGK}). The Lie algebra $\mathfrak{g}_2$ (resp. $\mathfrak{su}(3)$) may be realized as the stabilizer of one (resp. two) spinors in the real spin representation $\Sigma_8:=\R^8$ of $\mathfrak{spin}(7)$. Explicitly, if $\epsilon_1,\dots\epsilon_7$ is the standard basis of $\R^7$ and $\phi_1,\dots,\phi_8$ is the standard basis of the real spinor module $\Sigma_7=\R^8$, then $\mathfrak{g}_2$ and $\mathfrak{su}(3)$ are realized inside $\mathfrak{spin}(7)=\Span_{\R}\{ \epsilon_i\epsilon_j\}  $ via
	\begin{align*}
		\mathfrak{g}_2& \cong \text{\emph{stab}}_{\mathfrak{spin}(7)}\{\phi_1 \} \cong  \left\{\sum_{1\leq i<j \leq 7} \omega_{i,j} \epsilon_i\epsilon_j\: \begin{multlined} \omega_{1,2}+\omega_{3,4} + \omega_{5,6}=0, \\ -\omega_{1,3}+\omega_{2,4} -\omega_{6,7}=0, \ -\omega_{1,4}-\omega_{2,3} -\omega_{5,7} =0, \\ -\omega_{1,6}-\omega_{2,5} +\omega_{3,7}=0, \ \omega_{1,5}-\omega_{2,6} -\omega_{4,7}=0, \\ \omega_{1,7}+\omega_{3,6} +\omega_{4,5} =0, \ \omega_{2,7}+ \omega_{3,5}-\omega_{4,6}=0  \end{multlined}     \right\} , \\
		\mathfrak{su}(3)&\cong \text{\emph{stab}}_{\mathfrak{spin}(7)} \{\phi_1,\phi_2\} = \left\{ \sum_{1\leq i<j\leq 7} \omega_{i,j}\epsilon_i\epsilon_j\: \begin{multlined} \omega_{1,2}+\omega_{3,4}+\omega_{5,6}=0 \\ \omega_{1,3}=\omega_{2,4}, \ \omega_{1,4}+\omega_{2,3}=0, \  \omega_{1,5}=\omega_{2,6}, \\ \omega_{1,6}+\omega_{2,5}=0, \ \omega_{3,5}=\omega_{4,6}, \ \omega_{3,6}+\omega_{4,5}=0, \\ \omega_{1,7}=\omega_{2,7}=\dots=\omega_{6,7}=0 \end{multlined}    \right\}. 
	\end{align*}
\end{lemma}
\begin{remark}
		In order to find bases for these Lie algebras we use the explicit realization of the spin representation obtained from the following matrices: 
		\begin{align*} 
		\rho(\epsilon_1) &:= E^{(8)}_{1, 8} + E^{(8)}_{2, 7} - E^{(8)}_{3, 6} - E^{(8)}_{4, 5}, \quad 
		\rho(\epsilon_2) := -E^{(8)}_{1, 7} + E^{(8)}_{2, 8} + E^{(8)}_{3, 5} - 
		E^{(8)}_{4, 6},\\
		\rho(\epsilon_3) &:= -E^{(8)}_{1, 6} + E^{(8)}_{2, 5} - E^{(8)}_{3, 8} + 
		E^{(8)}_{4, 7},\quad 
		\rho(\epsilon_4) := -E^{(8)}_{1, 5} - E^{(8)}_{2, 6} - E^{(8)}_{3, 7} - 
		E^{(8)}_{4, 8},\\
		\rho(\epsilon_5) &:= -E^{(8)}_{1, 3} - E^{(8)}_{2, 4} + E^{(8)}_{5, 7} + 
		E^{(8)}_{6, 8},\quad 
		\rho(\epsilon_6) := E^{(8)}_{1, 4} - E^{(8)}_{2, 3} - E^{(8)}_{5, 8} + E^{(8)}_{6, 7},\\
		\rho(\epsilon_7) &:= E^{(8)}_{1, 2} - E^{(8)}_{3, 4} - E^{(8)}_{5, 6} + E^{(8)}_{7, 8} 
		\end{align*}
		(see Chapter 4.4 in \cite{BFGK}). By substituting these into the equations of the preceding lemma and subsequently orthogonalizing with respect to $B_0$, one obtains the following proposition.
\end{remark}
\begin{proposition} \label{onbasisg2su3}
	A $B_0$-orthonormal basis for $\mathfrak{g}_2$ given by 
	\begin{align*}
	\nu_1& := \frac{1}{4} (\rho (\epsilon_1) \rho (\epsilon_2)-\rho (\epsilon_5) \rho (\epsilon_6)),\quad \nu_2:= \frac{1}{4} (\rho (\epsilon_3) \rho (\epsilon_5)+\rho (\epsilon_4) \rho (\epsilon_6)), \\ \nu_3&:= \frac{1}{4} (\rho (\epsilon_3) \rho (\epsilon_6)-\rho (\epsilon_4) \rho (\epsilon_5)), \quad \nu_4:= \frac{1}{4} (\rho (\epsilon_1) \rho (\epsilon_3)+\rho (\epsilon_2) \rho (\epsilon_4)),\\ \nu_5&:= \frac{1}{4} (\rho (\epsilon_1) \rho (\epsilon_4)-\rho (\epsilon_2) \rho (\epsilon_3)), \quad \nu_6:= \frac{1}{4} (\rho (\epsilon_1) \rho (\epsilon_5)+\rho (\epsilon_2) \rho (\epsilon_6)),\\ \nu_7&:= \frac{1}{4} (\rho (\epsilon_1) \rho (\epsilon_6)-\rho (\epsilon_2) \rho (\epsilon_5)), \quad \nu_8:= -\frac{\rho (\epsilon_1) \rho (\epsilon_2)-2 \rho (\epsilon_3) \rho (\epsilon_4)+\rho (\epsilon_5) \rho (\epsilon_6)}{4 \sqrt{3}},\\ \nu_9&:= \frac{2 \rho (\epsilon_1) \rho (\epsilon_7)-\rho (\epsilon_3) \rho (\epsilon_6)-\rho (\epsilon_4) \rho (\epsilon_5)}{4 \sqrt{3}},\quad \nu_{10}:= \frac{2 \rho (\epsilon_2) \rho (\epsilon_7)-\rho (\epsilon_3) \rho (\epsilon_5)+\rho (\epsilon_4) \rho (\epsilon_6)}{4 \sqrt{3}}, \\ \nu_{11}&:= \frac{\rho (\epsilon_1) \rho (\epsilon_3)-\rho (\epsilon_2) \rho (\epsilon_4)-2 \rho (\epsilon_6) \rho (\epsilon_7)}{4 \sqrt{3}},\quad \nu_{12} :=\frac{\rho (\epsilon_1) \rho (\epsilon_4)+\rho (\epsilon_2) \rho (\epsilon_3)-2 \rho (\epsilon_5) \rho (\epsilon_7)}{4 \sqrt{3}},\\ \nu_{13}&:= \frac{\rho (\epsilon_1) \rho (\epsilon_5)-\rho (\epsilon_2) \rho (\epsilon_6)+2 \rho (\epsilon_4) \rho (\epsilon_7)}{4 \sqrt{3}},\quad \nu_{14}:= \frac{\rho (\epsilon_1) \rho (\epsilon_6)+\rho (\epsilon_2) \rho (\epsilon_5)+2 \rho (\epsilon_3) \rho (\epsilon_7)}{4 \sqrt{3}},
		\end{align*}
		with the first 8 elements $\nu_1,\dots \nu_8$ forming a $B_0$-orthonormal basis for the subalgebra $\mathfrak{su}(3)$.
\end{proposition}
We now wish to extend this to $B_0$-orthonormal bases of $\mathfrak{spin}(7)$ and $\mathfrak{spin}(9)$. Denoting by $\iota\: \text{Mat}_{8}(\R)\hookrightarrow \text{Mat}_9(\R)$ the embedding as the lower right hand $8\times 8$ block, one has,
\begin{proposition} \label{onbasisspin7spin9}
	The basis $\{\nu_1,\dots,\nu_{14}\}$ extends to a $B_0$-orthonormal basis of $\mathfrak{spin}(7)$ given by
	$\{\nu_1,\dots,\nu_{14},\nu'_{15},\dots,\nu'_{21}\}$ 
	and a $B_0$-orthonormal basis of $\mathfrak{spin}(9)$ given by \linebreak 
	$
	\{\iota(\nu_1),\dots,\iota(\nu_{14}),\iota(\nu'_{15}),\dots,\iota(\nu'_{21}),\nu'_{22},\dots , \nu'_{36}\}
	$, where
	\begin{align*}
	\nu'_{15}&:= 
	\frac{\rho (\epsilon_1) \rho (\epsilon_2)+\rho (\epsilon_3) \rho (\epsilon_4)+\rho (\epsilon_5) \rho (\epsilon_6)}{2 \sqrt{6}},\quad \nu'_{16}:= \frac{\rho (\epsilon_1) \rho (\epsilon_3)-\rho (\epsilon_2) \rho (\epsilon_4)+\rho (\epsilon_6) \rho (\epsilon_7)}{2 \sqrt{6}},\\ \nu'_{17}&:= \frac{\rho (\epsilon_1) \rho (\epsilon_4)+\rho (\epsilon_2) \rho (\epsilon_3)+\rho (\epsilon_5) \rho (\epsilon_7)}{2 \sqrt{6}},\quad \nu'_{18}:= \frac{-\rho (\epsilon_1) \rho (\epsilon_5)+\rho (\epsilon_2) \rho (\epsilon_6)+\rho (\epsilon_4) \rho (\epsilon_7)}{2 \sqrt{6}},\\ \nu'_{19}&:= -\frac{\rho (\epsilon_1) \rho (\epsilon_6)+\rho (\epsilon_2) \rho (\epsilon_5)-\rho (\epsilon_3) \rho (\epsilon_7)}{2 \sqrt{6}},\quad \nu'_{20}:=\quad \frac{\rho (\epsilon_1) \rho (\epsilon_7)+\rho (\epsilon_3) \rho (\epsilon_6)+\rho (\epsilon_4) \rho (\epsilon_5)}{2 \sqrt{6}}, \\ \nu'_{21}&:= -\frac{\rho (\epsilon_2) \rho (\epsilon_7)+\rho (\epsilon_3) \rho (\epsilon_5)-\rho (\epsilon_4) \rho (\epsilon_6)}{2 \sqrt{6}},
	\end{align*}
	and {\small
\begin{align*}
	\nu'_{22}&:= \sqrt{2}\left(E_{2,3}^{(9)} -\sqrt{\frac{3}{2}} \  \iota(\nu'_{15}) \right), \quad \nu'_{23}:= \sqrt{2}\left(E_{2,4}^{(9)} +\sqrt{\frac{3}{2}} \  \iota(\nu'_{16}) \right), \quad \nu'_{24}:=  \sqrt{2}\left(E_{2,5}^{(9)} +\sqrt{\frac{3}{2}} \  \iota(\nu'_{17}) \right),\\
	\nu'_{25}&:=  \sqrt{2}\left(E_{2,6}^{(9)} -\sqrt{\frac{3}{2}} \  \iota(\nu'_{19}) \right),\quad \nu'_{26}:=  \sqrt{2}\left(E_{2,7}^{(9)} +\sqrt{\frac{3}{2}} \ \iota( \nu'_{18}) \right), \quad \nu'_{27}:=  \sqrt{2}\left(E_{2,8}^{(9)} -\sqrt{\frac{3}{2}} \ \iota( \nu'_{20}) \right), \\
	\nu'_{28}&:=  \sqrt{2}\left(E_{2,9}^{(9)} +\sqrt{\frac{3}{2}} \ \iota( \nu'_{21}) \right), \quad \nu'_{28+i}:=\frac{1}{\sqrt{2}} E^{(9)}_{1,1+i} \ \forall i=1,\dots,8.
\end{align*}
}
\end{proposition}
\newpage
\section{Invariant Differential Forms on $S^{15}=\Spin(9)/\Spin(7)$}\label{AppendixB_forms}
Here we give explicit formulas, in terms of the basis (\ref{S15ONB}), for the differential forms on $S^{15}=\Spin(9)/\Spin(7)$ discussed in Section \ref{spin9spin7}:
\begin{align*}
	\omega&:=   -e_{1,  8,  9}+e_{1,  10,  11}+e_{1,  12,  13}-e_{1,  14,  15}-e_{2,  8,  10}-e_{2,  9,  11}+e_{2,  12,  14}+e_{2,  13,  15} \\ &\qquad    -e_{3,  8,  11}+e_{3,  9,  10}+e_{3,  12,  15}-e_{3,  13,  14}-e_{4,  8,  12}-e_{4,  9,  13}-e_{4,  10,  14}-e_{4,  11,  15} \\ &\qquad   -e_{5,  8,  13}+e_{5,  9,  12}-e_{5,  10,  15}+e_{5,  11,  14}-e_{6,  8,  14}+e_{6,  9,  15}+e_{6,  10,  12}-e_{6,  11,  13} \\ &\qquad   -e_{7,  8,  15}-e_{7,  9,  14}+e_{7,  10,  13}+e_{7,  11,  12}     , \\
	\Psi &:= e_{8,9,10,11} + e_{8,9,12,13} -e_{8,9,14,15} + e_{8,10,12,14} + e_{8,10,13,15} + e_{8,11,12,15} \\ &\qquad  -e_{8,11,13,14}   -e_{9,10,12,15}  + e_{9,10,13,14} + e_{9,11,12,14} + e_{9,11,13,15} \\ &\qquad  -e_{10,11,12,13} + e_{10,11,14,15} + e_{12,13,14,15} , \\
	\sqrt{\frac{a}{2}} \ d\omega &= e_{1,2,8,11}-e_{1,2,9,10}+e_{1,2,12,15}-e_{1,2,13,14}-e_{1,3,8,10}-e_{1,3,9,11}-e_{1,3,12,14}-e_{1,3,13,15}\\
	&\qquad +e_{1,4,8,13}-e_{1,4,9,12}-e_{1,4,10,15}+e_{1,4,11,14}-e_{1,5,8,12}-e_{1,5,9,13}+e_{1,5,10,14}+e_{1,5,11,15}\\
	&\qquad -e_{1,6,8,15}-e_{1,6,9,14}-e_{1,6,10,13}-e_{1,6,11,12}+e_{1,7,8,14}-e_{1,7,9,15}+e_{1,7,10,12}-e_{1,7,11,13}\\
	&\qquad +e_{2,3,8,9}-e_{2,3,10,11}+e_{2,3,12,13}-e_{2,3,14,15}+e_{2,4,8,14}+e_{2,4,9,15}-e_{2,4,10,12}-e_{2,4,11,13}\\
	&\qquad +e_{2,5,8,15}-e_{2,5,9,14}-e_{2,5,10,13}+e_{2,5,11,12}-e_{2,6,8,12}+e_{2,6,9,13}-e_{2,6,10,14}+e_{2,6,11,15}\\
	&\qquad -e_{2,7,8,13}-e_{2,7,9,12}-e_{2,7,10,15}-e_{2,7,11,14}+e_{3,4,8,15}-e_{3,4,9,14}+e_{3,4,10,13}-e_{3,4,11,12}\\
	&\qquad -e_{3,5,8,14}-e_{3,5,9,15}-e_{3,5,10,12}-e_{3,5,11,13}+e_{3,6,8,13}+e_{3,6,9,12}-e_{3,6,10,15}-e_{3,6,11,14}\\
	&\qquad -e_{3,7,8,12}+e_{3,7,9,13}+e_{3,7,10,14}-e_{3,7,11,15}+e_{4,5,8,9}+e_{4,5,10,11}-e_{4,5,12,13}-e_{4,5,14,15}\\
	&\qquad +e_{4,6,8,10}-e_{4,6,9,11}-e_{4,6,12,14}+e_{4,6,13,15}+e_{4,7,8,11}+e_{4,7,9,10}-e_{4,7,12,15}-e_{4,7,13,14}\\
	&\qquad -e_{5,6,8,11}-e_{5,6,9,10}-e_{5,6,12,15}-e_{5,6,13,14}+e_{5,7,8,10}-e_{5,7,9,11}+e_{5,7,12,14}-e_{5,7,13,15}\\
	&\qquad -e_{6,7,8,9}-e_{6,7,10,11}-e_{6,7,12,13}-e_{6,7,14,15}-\frac{3 a e_{8,9,10,11}}{2b}-\frac{3 a e_{8,9,12,13}}{2b}+\frac{3 a e_{8,9,14,15}}{2b}\\
	&\qquad -\frac{3 a e_{8,10,12,14}}{2b}-\frac{3 a e_{8,10,13,15}}{2b}-\frac{3 a e_{8,11,12,15}}{2b}+\frac{3 a e_{8,11,13,14}}{2b}+\frac{3 a e_{9,10,12,15}}{2b}\\
	&\qquad -\frac{3 a e_{9,10,13,14}}{2b} -\frac{3 a e_{9,11,12,14}}{2b}-\frac{3 a e_{9,11,13,15}}{2b}+\frac{3 a e_{10,11,12,13}}{2b}-\frac{3 a e_{10,11,14,15}}{2b}\\
 &\qquad -\frac{3 a e_{12,13,14,15}}{2b}    \\
	\sqrt{\frac{a}{2}} \ d\Psi&= e_{1,8,10,12,15}-e_{1,8,10,13,14}-e_{1,8,11,12,14}-e_{1,8,11,13,15}+e_{1,9,10,12,14}+e_{1,9,10,13,15}\\
	&\qquad +e_{1,9,11,12,15}-e_{1,9,11,13,14}-e_{2,8,9,12,15}+e_{2,8,9,13,14}+e_{2,8,11,12,13}-e_{2,8,11,14,15}\\
	&\qquad -e_{2,9,10,12,13}+e_{2,9,10,14,15}+e_{2,10,11,12,15}-e_{2,10,11,13,14}+e_{3,8,9,12,14}+e_{3,8,9,13,15}\\
	&\qquad -e_{3,8,10,12,13}+e_{3,8,10,14,15}-e_{3,9,11,12,13}+e_{3,9,11,14,15}-e_{3,10,11,12,14}-e_{3,10,11,13,15}\\
	&\qquad +e_{4,8,9,10,15}-e_{4,8,9,11,14}+e_{4,8,10,11,13}-e_{4,8,13,14,15}-e_{4,9,10,11,12}+e_{4,9,12,14,15}\\
	&\qquad -e_{4,10,12,13,15}+e_{4,11,12,13,14}-e_{5,8,9,10,14}-e_{5,8,9,11,15}-e_{5,8,10,11,12}+e_{5,8,12,14,15}\\
	&\qquad -e_{5,9,10,11,13}+e_{5,9,13,14,15}+e_{5,10,12,13,14}+e_{5,11,12,13,15}+e_{6,8,9,10,13}+e_{6,8,9,11,12}\\
	&\qquad -e_{6,8,10,11,15}-e_{6,8,12,13,15}-e_{6,9,10,11,14}-e_{6,9,12,13,14}+e_{6,10,13,14,15}+e_{6,11,12,14,15}\\
	&\qquad -e_{7,8,9,10,12}+e_{7,8,9,11,13}+e_{7,8,10,11,14}+e_{7,8,12,13,14}-e_{7,9,10,11,15}-e_{7,9,12,13,15}\\
	&\qquad -e_{7,10,12,14,15}+e_{7,11,13,14,15}.
\end{align*} 

We also record in Table \ref{Tab:isotropytypesS15} the \emph{isotropy types} of the forms $\omega, \Psi, d\omega, d\Psi$ from Section \ref{spin9spin7}, i.e. the number of factors from each isotropy component. The isotropy types of all other invariant forms in Table \ref{Tab:invariant_forms_S15} may be easily deduced from these. 
\begin{table}[h!] 
	\centering	
	\caption{Isotropy Types of Invariant Forms on $S^{15}=\Spin(9)/\Spin(7)$}
	\begin{tabular}{ |l||l| }
		\hline
		Form & Isotropy Type \\
		\hline
		$\omega $   &    $\mathfrak{m}_F\otimes \Lambda^2\mathfrak{m}_B$   \\
		$\Psi$ & $\Lambda^4\mathfrak{m}_B$ \\
		$d\omega$ & $(\Lambda^2\mathfrak{m}_F \otimes \Lambda^2\mathfrak{m}_B) \oplus (\Lambda^4 \mathfrak{m}_B)$ \\
		$d\Psi$ & $\mathfrak{m}_F\otimes \Lambda^4 \mathfrak{m}_B$ \\
		\hline
	\end{tabular}
	\label{Tab:isotropytypesS15}
\end{table}

\bigskip
\bigskip

\bibliography{bibliodatabase}
\bibliographystyle{alpha}


\end{document}